\documentclass[a4paper] {article}
[12pt]

\makeatletter
\renewcommand*\l@section{\@dottedtocline{1}{1.5em}{2.3em}}
\makeatother

\usepackage{CJK}
\usepackage{amsmath}

\usepackage{amsfonts}
\usepackage{amssymb}
\usepackage{amsthm}
\usepackage{amssymb}
\usepackage{enumerate}
\usepackage[calc]{picture}
\usepackage[all,cmtip]{xy}

\usepackage[mathscr]{eucal}
\usepackage{eqlist}

\usepackage{color}
\usepackage{abstract}
\usepackage[T1]{fontenc}

\usepackage[top=2cm, bottom=2cm, left=2 cm, right= 2 cm]{geometry}

\theoremstyle{plain}
\newtheorem{theorem}{Theorem}
\newtheorem{proposition}[theorem]{Proposition}
\newtheorem{lemma}[theorem]{Lemma}
\newtheorem{corollary}[theorem]{Corollary}

\newtheorem{example}[theorem]{Example}

\theoremstyle{definition}
\newtheorem{definition}{Definition}
\theoremstyle{remark}
\newtheorem{remark}[theorem]{Remark}

\numberwithin{equation}{section}
\numberwithin{theorem}{section}

\begin{document}

\vspace{3cm}

\begin{center}

{\Large {\textbf { Topological persistence  of   configuration  spaces  and  
 independence  complexes  for  digraphs  }}}

 \vspace{0.58cm}

Shiquan Ren

\smallskip

\begin{quote}
\begin{abstract}
We  study  the  topological persistence  of  the  (path)  configuration  spaces  and 
 the   (path)  
independence  complexes  
for  digraphs  as  well  as     their  underlying  graphs.  
We  construct   some  canonical  embeddings  from  the    (path)  
independence  complexes  
of  the  underlying  graphs  to  the   (path)  
independence  complexes 
of  the  digraphs  as  well  as   some canonical  embeddings  between  
the       (path)  
independence  complexes  
induced  by  strong totally  geodesic  immersions  and  strong totally  geodesic  embeddings  
of  (di)graphs.  
We  apply    the  path  homology  to  
the  path  independence  complexes  of  (di)graphs.  
As  by-products,  we   derive   some  consequences  about  the  Shannon  capacities.      
\end{abstract}

\smallskip

{ {\bf 2010 Mathematics Subject Classification.}  	 Primary  
  51E30,  55N31;
  Secondary  	05B25, 05B40
 }

{{\bf Keywords and Phrases.}  
 configuration  spaces,  independence  complexes,   digraphs,    immersions and  embeddings,  
 chain  complexes,  persistent  homology

  }

\end{quote}

\end{center}

\section{Introduction}

{\bf (a)  Configuration  spaces}.  
Let  $X$  be  a  space.  The   {\it  $k$-th  ordered  configuration  space}
${\rm  Conf}_k(X)$  is  the  subspace  of  the  Cartesian  product  $X^k$
 such that  the  $k$-coordinates  in  $X$  are  distinct.  
 The  $k$-th  symmetric  group  $\Sigma_k$  acts  on  ${\rm  Conf}_k(X)$  
 freely  by  
 perturbing  the  coordinates.   The  {\it  $k$-th  unordered  configuration  space}  is  the  orbit  space  
     ${\rm  Conf}_k(X)/\Sigma_k$.

  Suppose  in  addition   that  $X$   is  equipped  with  a  metric  $d: X\times  X\longrightarrow  [0,+\infty]$.       
  For any  $r\geq  0$,  
  the   $k$-th {\it  ordered  configuration  space  of  hard  $r$-spheres}   
${\rm  Conf}_k(X,r)$  is  the  subspace  of  ${\rm  Conf}_k(X)$  
 consisting  of  the configurations 
   $(x_1,x_2,\ldots,x_k)$  such  that  $d(x_i,x_j)>2r$  for  any  $i\neq  j$.  
 The  $k$-th   {\it unordered  configuration  space    of  hard  $r$-spheres}    is  the  orbit  space  
     ${\rm  Conf}_k(X,r)/\Sigma_k$.  
       For  any    
$0\leq  r<s\leq   \infty$,
  we  define  
the   $k$-th  {\it   configuration  space}   of   $X$  {\it  with  constraint  $(r,s)$}    as  the  space   
\begin{eqnarray*}
{\rm  Conf}_k(X,r,s) =  \{(x_1,\ldots,x_k)\in  X^k\mid   2r< d(x_i,x_j)\leq  2s {\rm~for~any~}i\neq  j\},     
\end{eqnarray*}
which  is   the  complement  of  ${\rm  Conf}_k(X,s)$ 
     in ${\rm  Conf}_k(X,r)$,  with   the  product  metric $d^k$.  
     Note  that    ${\rm  Conf}_k(X,r,s)$    is  $\Sigma_k$-invariant thus  
     we  have an  orbit  space  ${\rm  Conf}_k(X,r,s)/\Sigma_k$.  
We  have a   double-parametrized  filtration  
\begin{eqnarray*}
{\rm  Conf}_k(X,-,-)  =  \{  {\rm  Conf}_k(X,r,s) \mid   0\leq  r<s\leq   \infty \}
\end{eqnarray*}
which  induces  a  double-persistent  homology 
\begin{eqnarray*}
H_*({\rm  Conf}_k(X,-,-) ) =  \{ H_*( {\rm  Conf}_k(X,r,s)) \mid   0\leq  r<s\leq   \infty \}.  
\end{eqnarray*}
 The  symmetric  group  $\Sigma_k$  acts  on  ${\rm  Conf}_k(X,-,-) $   freely
 such  that  the  double-filtration  is  $\Sigma_k$-equivariant.   
  This    induces  a   $\Sigma_k$-action  on 
   the     homology  $H_*({\rm  Conf}_k(X,-,-) )$ 
 such  that the  double-persistence  is  $\Sigma_k$-equivariant.

    One   special  case  for  configuration  spaces   is  that  $X$  is  a  manifold  $M$.  
    In   1978,  F.  R.  Cohen  and  L.  R.  Taylor  
    \cite{gelfand,gelfand2}  studied   the cohomology  of  the  ordered  configuration  space 
    ${\rm  Conf}_k(M)$  and    the  unordered  configuration  space   ${\rm  Conf}_k(M)/\Sigma_k$. 
    In  2010,  
    an  introduction  to     ${\rm  Conf}_k(M)$  and   
     ${\rm  Conf}_k(M)/\Sigma_k$   as  well  as their   applications  
       is given  by  F.  R.  Cohen     \cite{cohen2010}. 
            For  the  special  case  that   $M$   is  the  Euclidean  space,  
      F.  R.  Cohen  \cite{cohen1973,cohen1973.1}  obtained the    information on
   the  cohomology  of  
       ${\rm  Conf}_k(\mathbb{R}^m)$  and   ${\rm  Conf}_k(\mathbb{R}^m)/\Sigma_k$
       by  using  $m$-fold loop  spaces;    
      and  the  cohomology  of  ${\rm  Conf}_k(\mathbb{R}^2)/\Sigma_k$ 
       is  applied   by  F.  Cohen  and  D.  Handel  \cite{cohen1}  to  study 
         the  $k$-regular  embeddings  
       of  the  plane  into  ambient  Euclidean  spaces.

       In  addition,  if  $M$  has  a  Riemannian  metric  thus  has  an  induced  distance  $d$,  
       then  we  have  the  configuration  spaces  of  hard  $r$-spheres
       ${\rm  Conf}_k(M,r)$  and  ${\rm  Conf}_k(M,r)/\Sigma_k$,  
       which   give  information  about  the  sphere-packings  on  $M$.    
       As  $r$  varies,  the  persistent  homology  of   the  configuration  spaces  of  hard  $r$-spheres
       in  a  strip  is  studied  by  H.   Alpert  and  Fedor Manin  \cite{gt}.  
         With  the  help  of  the  Min-type  Morse  theory  (cf.  \cite{min-type}),  
       it  is  proved  by Y.  Baryshnikov, P.   Bubenik  and   M. Kahle  \cite{imrn1}  
       that mechanically  balanced  configurations  in a  bounded  region  in  Euclidean  spaces  
       play  the  role  of  critical  points.

  Another  special   case   for  configuration  spaces  is  that  $X$  is  a  graph  $G$.  
     Consider  the geometric  realization  $|G|$  of  a  graph  $G$,  which  is  a  
     $1$-dimensional  cell  complex. 
     In  recent  years,  
     the    ordered   configuration  space   ${\rm  Conf}_k(|G|)$  and  
     the  unordered  configuration  space  ${\rm  Conf}_k(|G|)/\Sigma_k$   
     of  $k$-distinct  points  in  $|G|$ 
      have  been  extensively  studied,  for  example,   R.  Ghrist \cite{ghrist},  
      B.  Knudsen \cite{gconf1},  F.  R.  Cohen  and R.  Huang  \cite{cohen22},  
      etc.  
      The  homeomorphic type  of    ${\rm  Conf}_k(|G|)$  as  well  as  ${\rm  Conf}_k(|G|)/\Sigma_k$ 
is  determined  by  the  homeomorphic  type  of  $|G|$,  
which  is  determined  by  the combinatorial  structures  of  $G$.  
 However,  even  if  the  geometric  realizations  $|G|$  and  $|G'|$  of  two graphs  
 $G$  and  $G'$  are  homeomorphic,  
 the  combinatorial structures   of  $G$  and  $G'$  could  be  different.

 {\bf  (b)   Digraphs and their path  complexes}.  
 A  digraph  $\vec  G$  is  obtained  by  
  assigning  a  direction  or  both  directions  to  each  edge  of  a  graph  $G$  
  while  a  graph  $G$  is  obtained  by  forgetting  the  direction   on  each  
  arc   of  a  digraph  $\vec  G$  (see   Definition~\ref{def-dig-25-apr-1}).

  Let  $V$  be  the  vertex  set  of  $\vec   G$.  
  An   {\it  elementary  $n$-path}  on  $V$   is  a  sequence  
   $v_0v_1\ldots  v_n$    of  vertices    $v_0,v_1,\ldots,v_n\in   V$.  
   In  addition,  $v_0v_1\ldots  v_n$ 
  is    {\it regular}  if     $v_{i-1}\neq  v_i$  for each  $1\leq  i\leq  n$
  and  is  {\it   non-regular}  if  $v_{i-1}= v_i$  for  some  $1\leq  i\leq  n$.  
  The  collection  of  all  the  elementary  paths  of  finite  lengths  on  $V$  generates  a  
  free  abelian  group  
  $\Lambda_*(V)= \bigoplus_{n\geq  0}  \Lambda_n(V)$,  which  is  
  a  chain  complex  with  its  boundary  map  
  sending   each  elementary $n$-path  $v_0v_1\ldots  v_n$  to  
  a   linear  combination  of  elementary  $(n-1)$-paths   
  $\sum_{i=0}^n   (-1)^i v_0\ldots \widehat{v_i} \ldots  v_n$.  
  The  collection  of  all  the   non-regular  elementary  paths  of  finite  lengths  on  $V$  
 generates  a   sub-chain  complex   $I_*(V)= \bigoplus_{n\geq  0}  I_n(V)$  
 of  $\Lambda_*(V)$  (cf.  \cite[Lemma~2.9~(a)]{lin2}).  
  The  quotient  chain  complex  $\mathcal{R}_*(V)= \bigoplus_{n\geq  0}\mathcal{R}_n(V)$,   
  where    $\mathcal{R}_n(V)= \Lambda_n(V) /  I_n(V) $   for  each  $n\geq  0$,  
    is   generated  by  the  collection  of  all  the  regular  elementary   paths  on  $V$
    and  is  equipped  with  the  quotient boundary  map  by  dropping  
      all the non-regular components  
     (cf.  \cite[Definition~2.10]{lin2}).

  An  {\it  allowed}  elementary   $n$-path  on  a  digraph  
  $\vec  G$  is  a sequence  of  vertices $v_0v_1\ldots  v_n$ 
  such  that  for each  $1\leq  i\leq  n$,  
  either  $v_{i-1}=v_i$  or  $v_{i-1}\to  v_i$  is  an arc  of  $\vec  G$. 
  The  collection  of  all  the  allowed   regular  elementary    paths   of  finite  lengths  on  
    $\vec  G$  generates 
 a  subgroup  $\mathcal{A}_*(\vec  G)=\bigoplus_{n\geq  0} \mathcal{A}_n(\vec  G)$
    of   $\mathcal{R}_*(V)$.   
   Given  an   allowed  elementary  $n$-path  $v_0v_1\ldots  v_n$  in  $\mathcal{A}_n(\vec  G)$,  
   both  $\widehat{v_0} v_1\ldots  v_n$  and   $v_0v_1\ldots  \widehat{v_n}$
   are  elementary  $(n-1)$-paths  in  $\mathcal{A}_{n-1}(\vec  G)$.  
   However,  $v_0\ldots\widehat{v_i}\ldots  v_n$  may  not  be    an  elementary  $(n-1)$-path 
    in  $\mathcal{A}_{n-1}(\vec  G)$,   
       for  $1\leq  i\leq  n-1$. 
    Thus  $\mathcal{A}_*(\vec  G)$  may  not  be   a  sub-chain  complex  of   $\mathcal{R}_*(V)$.    
       A  sub-chain   complex   $ \Omega_*(\vec  G)=\bigoplus_{n\geq  0} \Omega_n(\vec  G)$
       of  $\mathcal{R}_*(V)$, 
 where  
 \begin{eqnarray*}
 \Omega_n(\vec  G)=  \mathcal{A}_n(\vec  G)  \cap \partial_n^{-1} \mathcal{A}_{n-1}(\vec  G),  
 \end{eqnarray*}
 is  constructed  
   and      the path homology  $H(  \Omega_*(\vec  G))$  of  $\vec  G$   is  studied  by  
 A.  Grigor'yan,  Y.  Lin, Y.  Muranov  and  S.-T. Yau  \cite{lin2,lin3,lin6}.  
 Later,   with the  help  of  \cite[Sec. 2]{h1},     $ \Omega_*(\vec  G)$  
  is  the  largest  chain  complex    contained  in  $\mathcal{A}_*(\vec  G)$,   
  denoted  as  ${\rm  Inf}( \mathcal{A}_*(\vec  G))$,  which  is  quasi-isomorphic  to  
  the  smallest  chain  complex   containing  $\mathcal{A}_*(\vec  G)$,  
  denoted  as   ${\rm  Sup}( \mathcal{A}_*(\vec  G))$.

 {\bf   (c)   The  Shannon  capacities}.  
 For  any  graph  $G$,  an  independent  set  is  a  collection  of  some  vertices  of  $G$ 
 such  that  any  two  of them  are  non-adjacent.  
All  the  finite  independent  sets  of  $G$  form   a  simplicial  complex  ${\rm  Ind}(G)$
which  is  called  the  
{\it  independence  complex}   (cf.  \cite{ind1,ind2}).   
Let  $\alpha(G)$  be  the  maximal  size  of  the  independent  sets  of   $G$,  i.e. 
$\alpha(G)-1$  is  the  dimension  
of  the  independence  complex.

  Given two  graphs $ G_1=(V_1,E_1)$ and $   G_2=(V_2,E_2)$,  their  {\it strong product} 
  $  G_1\boxtimes  G_2$    is   the    graph  whose  vertex set  is  $V_1\times V_2$   and  whose  edge  set  is  specified  by the following  rule:   
 for  any  distinct two  vertices  $(v_1,v_2)$  and  $ (u_1,u_2)$,    
 there  is  an  edge  between  them  
  iff  for each $i=1,2$, either $v_i=u_i$  or  $\{v_i,u_i\}\in E_i$  (cf. \cite{2,1,4}). 
    Let    $G^{\boxtimes  n}$ be the $n$-fold  self-strong product of $G$.  
 Motivated by the study of the channels in information theory,  C.  E.   Shannon \cite{4}  in 1956 
  introduced    the  capacity  $c(G)$,  which  is  given  by  
 (cf. \cite[p. 1]{1}  and   \cite{2,7,4})
\begin{eqnarray*} 
c(G)=\sup_{n\geq 1}~ \big(\alpha(G^{\boxtimes n})\big)^{\frac{1}{n}}=\lim _{n\to \infty}  \big(\alpha(G^{\boxtimes  n})\big)^{\frac{1}{n}}.
\end{eqnarray*} 
 So  far,  the study of  the  Shannon capacity of  graphs  has attracted  lots   of attention  (cf. \cite{2,7,8,1,5,4}). 
Moreover,  the  Shannon  capacity  of  graphs  is  generalized  to  the  capacity  of digraphs 
in the  sense  of  the  adjacency  matrices   by 
E. Bidamon  and  H.  Meyniel  \cite{dig-shan}.

{\bf  (d)  Results  of  this  paper}.  
  Let  $\vec  G$  be  a  digraph  and  let   $G$  be  its  underlying  graph.   
    We  consider  the  configuration  space   ${\rm  Conf}_k(\vec  G)$    
    consisting  of  all  the   ordered   $k$-tuples  of  mutually  non-adjacent  distinct  vertices in  $\vec  G$,  
    which  equals to 
   the  configuration  space  ${\rm  Conf}_k(G)$   
    consisting  of  all  the    ordered   $k$-tuples  of  mutually  non-adjacent  distinct  vertices of  $   G$.  
   The  family  of  configuration  spaces   ${\rm  Conf}_k(\vec  G)/\Sigma_k$,  
   which  equal  to  
      ${\rm  Conf}_k(  G)/\Sigma_k$,    for  $k\geq  1$,    
   gives  the  skeleton  of  the  independence  complex  ${\rm  Ind}(\vec  G)$  of  $\vec  G$,  
   which  equals to   the  independence  complex
     ${\rm  Ind}(  G)$  of  $G$.   The  Shannon  capacity  of   $G$  is  expressed  in  terms  of  the 
     dimension  of  the  independence  complex  of   the  self-strong  products  of  $G$.

We   take  the  canonical  distance $d_{\vec  G}$  on  the  vertex  set  such  that  
the  distance  between  any  two  vertices  is  the  minimal  length  of the paths  in  $\vec  G$
 connecting  the   two  vertices.   Similarly,  we  take the  canonical  distance  $d_G$ on  the vertex  set  
 by  the  the  minimal  length   of  the  paths  in  $G$.   
 For  any  $0\leq  r<s\leq  \infty$,  
consider  the {\it  constraint  configuration  space}   
${\rm  Conf}_k(\vec  G,r,s)/\Sigma_k$   consisting  of  all  the  ordered  $k$-tuples 
of  vertices  such  that  their  mutual  distances      $d_{\vec  G}$   lie  in   the  interval  $(2r,2s]$   
and  the constraint  configuration  space   
${\rm  Conf}_k(   G,r,s)/\Sigma_k$   consisting  of  all  the  ordered  $k$-tuples 
of  vertices  such  that  their  mutual  distances      $d_{   G}$   lie  in   $(2r,2s]$.   
Let  $k$  run  over  all  positive  integers.  
The  family  of  configuration  spaces  ${\rm  Conf}_k(\vec  G,r,s)$  gives  
a  {\it  constraint  independence  complex}   ${\rm  Ind}(\vec  G,r,s)$  
and  the  family  of  configuration  spaces  ${\rm  Conf}_k(   G,r,s)$  gives  
a  constraint  independence  complex  ${\rm  Ind}(   G,r,s)$.  
Let  $0\leq  r<s\leq  \infty$  run  over  all  possible  pairs  of  nonnegative  real  numbers
and  infinity.   
The  next  theorem  will be  proved  in  Subsection~\ref{ssa-3.2}.

\begin{theorem}\label{th-main-intro-1}
For  any  digraph  $\vec  G$  with  its  underlying graph  $G$,  
   we  have  a    family  of  
    persistent  $\Sigma_k$-equivariant   isometric   embeddings  
    $i_{\vec  G,k}(-,\infty)$  of  ${\rm  Conf}_k(   G , -,\infty)  $  into  $ {\rm  Conf}_k( \vec  G, -,\infty)$ 
    and   a    family  of  
    persistent  $\Sigma_k$-equivariant   isometric   embeddings  $j_{\vec  G,k}(1/2,-)$  of  ${\rm  Conf}_k(  \vec G ,1/2,-)  $  into  $ {\rm  Conf}_k(   G, 1/2,-)$  for  $k\geq  1$,   
   which  induce    a   persistent  
   simplicial  
   embedding
   \footnote[1]{A  (persistent)  simplicial embedding  is  an  injective  
   (persistent)  simplicial  map  between  (filtered) 
   simplicial  complexes.  }
    $i_{\vec  G }(-,\infty)$ 
of  ${\rm  Ind}(   G,-,\infty)$  into  ${\rm  Ind}(\vec  G,-,\infty)$
and  a   persistent  
   simplicial  
   embedding  $j_{\vec  G }(1/2,-)$ 
of  ${\rm  Ind}(  \vec  G,1/2,-)$  into  ${\rm  Ind}(   G,1/2,-)$,  
   such  that
   \begin{enumerate}[(1)]
   \item
     $ i_{\vec  G,k}(1/2,\infty)=j_{\vec  G,k}(1/2,\infty)^{-1}$  and   
   $i_{\vec  G}(1/2,\infty)=j_{\vec  G}(1/2,\infty)^{-1}$  are  the  identity  maps,
   \item  
      $ i_{\vec  G,k}(n/2,\infty)$  and  $i_{\vec  G}(n/2,\infty)$  
    are  inclusions  for  any  $2\leq  n<   \infty$,
   \item  
      $ j_{\vec  G,k}(1/2,n/2)$  and  $i_{\vec  G}(1/2,n/2)$  
    are  inclusions  for  any  $2\leq  n<   \infty$.  
    \end{enumerate}
\end{theorem}

A  strong  totally  geodesic  embedding  of  graphs     
is  a  graph  morphism preserving  the   distances  of  vertices  (cf.  \cite{2025-reg}).  
Similarly,  a strong  totally  geodesic  immersion     of  graphs   with  radius  $r$
is  a  graph  morphism preserving  the   distances    locally  
 in   the   geodesic balls  of  radius  $r$ (see  Definition~\ref{def-2}~(2)). 
Similarly,  by  using  the  distances  of  digraphs,     
strong  totally  geodesic  embeddings  of  digraphs  
and    strong  totally  geodesic  immersions  of  digraphs  
can  be  defined  
   (see  Definition~\ref{def-2}~(1)). 
   The  next  theorem  will be  proved  in  Subsection~\ref{ssa-3.3}.

   \begin{theorem}\label{th-main-intro-2}
     A  strong   totally geodesic  immersion  
   with radius  $m_0/2$  (resp.   a  strong   totally geodesic  embedding) 
    $\varphi:   \vec  G\longrightarrow   \vec   G'$  
      induces  a   family  of   double-persistent      $\Sigma_k$-equivariant     isometric  embeddings  
   of  $ {\rm  Conf}_k( { \vec  G}, -,-)$  into  $   {\rm  Conf}_k( {\vec   G'}, -,-) $
    for   $k\geq  1$,  
   and  thereby  induces  a        double-persistent  simplicial  embedding 
     of  $ {\rm  Ind} ( { \vec  G}, -,-)$  into  $   {\rm  Ind}( {\vec   G'}, -,-) $, 
   for  $1\leq  n<m\leq  m_0 $  (resp.   for  $1\leq  n<m\leq  \infty $), 
    where  $n/2$  is  the  first parameter  and  $m/2$  is  the  second  parameter  
   in  the  double-persistence.   
   \end{theorem}
   
   \begin{remark}
     A   similar  statement  is   satisfied  by  substituting  $\vec  G$
  and  $\vec  G'$  with  $G$  and  $G'$  respectively throughout  Theorem~\ref{th-main-intro-2}.   
   \end{remark}

Consider  the  {\it  path  configuration  space}   
$\overrightarrow {\rm  Conf}_k(\vec  G,r,s)$   consisting  of  all  the  ordered  $k$-tuples 
of  vertices  such  that  the  distances   in  $d_{\vec  G}$  between  any   two  adjacent  coordinates 
    lie  in   $(2r,2s]$.     
Similarly,  consider  the  path  configuration  space   
$\overrightarrow  {\rm  Conf}_k(   G,r,s)$   consisting  of  all  the  ordered  $k$-tuples 
of  vertices  such  that   the  distances   in  $d_{   G}$  between  any  two  adjacent  coordinates  
 lie  in   $(2r,2s]$.   
 Let  $k$  run  over  all  positive  integers.  
The  family  of   path  configuration  spaces  $\overrightarrow {\rm  Conf}_k(\vec  G,r,s)$  gives  
a  {\it   path  independence  complex}  $\overrightarrow {\rm  Ind}(\vec  G,r,s)$  
and  the  family  of  path  configuration  spaces  $\overrightarrow {\rm  Conf}_k(   G,r,s)$  gives  
a  path   independence  complex  $\overrightarrow {\rm  Ind}(   G,r,s)$.

Let  $\mathcal{D}_k(\vec  G, -,-)$  be  the  double-persistent  free  $R$-module 
spanned  by  $\overrightarrow {\rm  Conf}_k(\vec  G,-,-)$     and  let   $\mathcal{D}(\vec  G, -,-)= \bigoplus_{k\geq  1} \mathcal{D}_k(\vec  G, -,-)$,  where  $R$  is  a  commutative  ring  with  unit.  
Recall  that  by  \cite[Sec. 2]{h1}  or  an  analog  of  \cite[Sec.  9]{2025-jgp}, 
  the  largest  double-persistent  chain  complex  ${\rm  Inf}(\mathcal{D}(\vec  G, -,-))$ 
contained  in  $\mathcal{D}(\vec  G, -,-)$  and  the  smallest   double-persistent  chain  complex  
${\rm   Sup}(\mathcal{D}(\vec  G, -,-))$ 
containing    $\mathcal{D}(\vec  G, -,-)$  are  quasi-isomorphic.  
Therefore, 
   it  is  reasonable  to  define the  double-persistent  path  homology  of  $\mathcal{D}(\vec  G, -,-)$   
as  the  double-persistent  homology  of  ${\rm  Inf}(\mathcal{D}(\vec  G, -,-))$,
  which  is  isomorphic  to
 the  double-persistent  homology  of     
${\rm  Sup}(\mathcal{D}(\vec  G, -,-))$.  
 Similar  definitions  and    notations  apply   if  we   substitute  $\vec  G$  with  $G$.  
 The  next  two theorems  are path  versions  of  Theorem~\ref{th-main-intro-1}
 and  Theorem~\ref{th-main-intro-2}  respectively.  
 They  will be  proved  in  Section~\ref{ssa-4}.  

\begin{theorem}\label{th-main-intro-3}
For  any  digraph  $\vec  G$  with  its  underlying  graph  $G$, 
we  have   a   family  of    persistent   $\mathbb{Z}_2$-equivariant    isometric  embeddings  
 $I_{\vec  G}(-,\infty)$
of  $\overrightarrow {\rm  Conf}_k(   G,-,\infty)$  into  $\overrightarrow {\rm  Conf}_k(  \vec  G,-,\infty)$
and   a   family  of    persistent   $\mathbb{Z}_2$-equivariant    isometric  embeddings  
$J_{\vec  G}(1/2,-)$
of  $\overrightarrow {\rm  Conf}_k( \vec   G,1/2,-)$  into  $\overrightarrow {\rm  Conf}_k(    G,1/2,-)$
for  $k\geq  1$,  
which  respectively  induce      
 a persistent   
 $\mathbb{Z}_2$-equivariant  homomorphism  $I_{\vec  G}(-,\infty)_*$    from  the   persistent  homology 
  $H_*(\mathcal{D} (  G,  -,  \infty ))$  to  the   persistent  homology  
  $H_*(\mathcal{D} (  \vec G,  -,  \infty))$
  and  a persistent   
 $\mathbb{Z}_2$-equivariant  homomorphism   $J_{\vec  G}(1/2,-)_*$  from  the   persistent  homology 
  $H_*(\mathcal{D} ( \vec  G, 1/2,  - ))$  to  the   persistent  homology  
  $H_*(\mathcal{D} (   G,  1/2,  -))$,  such  that  
  $I_{\vec  G}(1/2,\infty)=J_{\vec  G}(1/2,\infty)^{-1}$  is   the  identity.  
\end{theorem}

   \begin{theorem}\label{th-main-intro-4}
     A  strong   totally geodesic  immersion  
   with radius  $m_0/2$  (resp.   a  strong   totally geodesic  embedding) 
    $\varphi:   \vec  G\longrightarrow   \vec   G'$  
      induces  a  family  of   double-persistent      $\mathbb{Z}_2$-equivariant  
        isometric   embeddings    
   of  $\overrightarrow  {\rm  Conf}_k( { \vec  G}, -,-)$  into  
   $ \overrightarrow  {\rm  Conf}_k( {\vec   G'}, -,-) $
      for  $1\leq  n<m\leq  m_0 $  (resp.   for  $1\leq  n<m \leq  \infty$),  
   and  thereby  induces  a    double-persistent          
 $\mathbb{Z}_2$-equivariant  homomorphism    
 from  $H_*(\mathcal{D} ( \vec  G,  -,  - ))$  to  $H_*(\mathcal{D} (  \vec G',  -,  - ))$
    for  $1\leq  n<m\leq  m_0 $  (resp.   for  $1\leq  n<m \leq  \infty$), 
    where  $n/2$  is  the  first parameter  and  $m/2$  is  the  second  parameter  
   in  the  double-persistence.   
   \end{theorem}
   
    \begin{remark}
     A   similar  statement  is   satisfied  by  substituting  $\vec  G$
  and  $\vec  G'$  with  $G$  and  $G'$  respectively throughout  Theorem~\ref{th-main-intro-4}.   
   \end{remark}

In  Section~\ref{s-5}, 
we  apply  Theorem~\ref{th-main-intro-1}  and   Theorem~\ref{th-main-intro-2}  
to give  some consequences  about  the Shannon  capacities.  
By  the  proof  of  Theorem~\ref{th-main-intro-1},  we  obtain 
  that  the  Shannon  capacity  of  the  underlying  graph  is  
smaller than  or  equal to the  Shannon  capacity  of  the  digraph
(see   Proposition~\ref{pr-25-apr-12-sha5}).   
By  the  proof  of   Theorem~\ref{th-main-intro-2},   
We  obtain  that   for  a   strong  totally  geodesic  immersion  or   
a   strong  totally  geodesic   embedding    
  of  (di)graphs,  
the  Shannon capacity  of  the  immersed  or  embedded  (di)graph  
is  smaller  than  or  equal  to  the  Shannon capacity  of  the  ambient  (di)graph
(see    Proposition~\ref{pr-25-apr-shan-98}).

\section{Digraphs  and  their  distances}\label{ssa.2}

In  this  section,  we  review  the  definitions  of  digraphs, their  underlying  graphs, 
and  the  canonical  distances  on  (di)graphs.  
We  discuss   the  strong  totally  geodesic    immersions  and   the  strong  totally  geodesic embeddings  of  (di)graphs.

\begin{definition}(cf.  \cite[pp. 2  - 4]{digraphs})\label{def-dig-25-apr-1}
Let  $V$  be  a  discrete  set.   
A  {\it  digraph}  $\vec  G=(V_{\vec G},    E_{\vec G})$  on  $V$   is  a  pair  such  that  
$V_{\vec G}$  is  a  subset  of  $V$  and  $   E_{\vec G}$  is   set  of   ordered  pairs  
of  distinct  vertices  in  $V_{\vec G}$.  
 The  elements  of  $V_{\vec G}$  are  {\it  vertices}  of  ${\vec G}$.  
 The  elements  of  $E_{\vec G}$  are  {\it  arcs}  of  ${\vec G}$,  denoted  by  $(u,v)$
 or  $u\to  v$.  
 Two vertices  $u$  and  $v$  in   $\vec  G$ are  {\it adjacent}  if  there  is  a  directed  edge $u\to  v$  
 or  a  directed  edge  $v\to  u$  in  $\vec  G$.  
 \end{definition}

 Let  $\vec  G=(V_{\vec G},    E_{\vec G})$  be  a  digraph.  
  Let  $n\in \mathbb{N}$. 
  
  \begin{definition} (cf.  \cite[Sec. 2]{lin3}  and  \cite{lin2,lin6})\label{def--2}
   An  {\it     elementary   $n$-path}  $\gamma_n$  on  $V$  is  a  sequence  $v_0v_1\ldots  v_n$  
  such that  $v_i\in  V$  for  each  $0\leq  i\leq  n$.    We  call  $n$  the  {\it  length}  of  $\gamma_n$.   
  In  addition,  if  $v_{j-1}\neq  v_j$  for  each  $1\leq  j\leq  n$,  then     $\gamma_n$ 
   is  called  {\it  regular};  otherwise  $\gamma_n$  is  called  {\it   non-regular}.  
  An  {\it   allowed  elementary   $n$-path}  $\gamma_n$  on  $\vec  G$  is  an    
   elementary   $n$-path   $v_0v_1\ldots  v_n$   on  
   $   V_{\vec  G}$    such that   
    either  $(v_{j-1}, v_j)\in  E_{\vec  G}$  or  $v_{j-1}=v_j$  for  each  $1\leq  j\leq  n$.  
   If  $u=v_0$  and  $v=v_n$,  
   then  we  say  that  $\gamma_n$  is  from  $u$  to  $v$.  
  \end{definition}

  \begin{definition}(cf.  \cite[Chap. 3]{digraphs})\label{def--3}
    The  {\it   distance}  on  $\vec  G$  is  a     function  
    $d_{\vec  G}:  V_{\vec  G} \times  V_{\vec  G}\longrightarrow   \mathbb{N}\cup \{\infty\} $
   such that  
   $d_{\vec  G}(u,v)$  is  the  smallest  length  of   allowed  elementary  paths  on  $\vec  G$   from  $u$  to  $v$
   or  from  $v$  to  $u$,  or  equivalently,  the  smallest  length  of  regular  allowed   elementary  paths  on  $\vec  G$   from  $u$  to  $v$
   or  from  $v$  to  $u$.  
   If  there  does  not  exist  any  (regular)  allowed  elementary    path  on  $\vec  G$  from  $u$  to  $v$  nor    from  $v$  to  $u$,  
   then  we  set  $d_{\vec  G}(u,v)=\infty$. 
  \end{definition}

 The  equivalence  relation  $(u,v)\sim  (v,u)$  on  $V \times  V $
 for   any  $u,v\in  V $   gives  a  projection    
   $ \pi:  V\times  V\longrightarrow  V\times  V/\sim$.

       \begin{definition}(cf.  \cite[p. 20]{digraphs})
The  {\it  underlying  graph}   $\pi(\vec  G)$   of  $\vec  G$  is  a  graph  $G=(V_G,E_G)$  
such  that  $V_G=V_{\vec  G}$  and  $E_G= \pi(E_{\vec  G})$.   
The  elements  of  $E_{ G}$  are  {\it   edges}  of  ${  G}$,  which  are  sets  of  two  vertices  
 of  the  form
$\{u,v\}$.      
Two  vertices  $u$  and  $v$  in  $G$  are  {\it adjacent}  if  $\{u,v\}$  is  an  edge  of  $G$.  
\end{definition}

Let  $G$  be  the  underlying  graph  of  $\vec  G$.  
The    pre-image  $\pi^{-1}(G)$  is  a  digraph 
    such  that  $(u,v)\in  E_{\pi^{-1}(G)}$  iff  $(v,u)\in  E_{\pi^{-1}(G)}$ iff  $\{u,v\}\in  E_{G}$.  
An  allowed  elementary  $n$-path   $\gamma_n$  on  $\pi^{-1}(G)$  (cf.  Definition~\ref{def--2}), 
which  will  also  be  called  an  {\it      allowed  elementary  $n$-path}  on  $G$,  
    is  a  sequence  $v_0v_1\ldots  v_n$  
  such that  $v_i\in  V_{   G}$  for  each  $0\leq  i\leq  n$  
  and  either  $\{v_{j-1}, v_j\}\in  E_{   G}$  of  $v_{j-1}=v_j$   for  each  $1\leq  j\leq  n$.  
The  distance  on  $\pi^{-1}(G)$  (cf.  Definition~\ref{def--3}),  
 which  will   be  called  the  {\it  distance}  on  $G$,    
 is  a  function   
    $d_{ G}:  V_{   G} \times  V_{   G}\longrightarrow   \mathbb{N}\cup \{\infty\} $  
   such that  
   $d_{  G}(u,v)$  is  the  smallest  length  of  (regular)  allowed  elementary    paths  on  $   G$   from  $u$  to  $v$.

\begin{lemma}\label{le-2.1}
   For  any  digraph  $\vec  G$  with   its  underlying  graph  $G$   and  any  $u,v\in  V_{  G}$, 
   we  have   
   \begin{eqnarray}\label{eq--2.a1}
   d_{\vec  G}(u,v) \geq  d_G(u,v).  
   \end{eqnarray}
   Moreover,  if  $\vec  G=\pi^{-1}(G)$,  then  the  equality  of  (\ref{eq--2.a1})  is  satisfied 
    for  any  $u,v\in  V_{  G}$.  
          \end{lemma}
          
          \begin{proof}
           Let  $\gamma $  be  any  (regular)   allowed  elementary      path  on  $\vec  G$  from  $u$  to  $v$  or  from  $v$  to  $u$.  
   Then  $\gamma $  is  a    (regular)     allowed  elementary  path  on  $G$  from  $u$  to  $v$.  
   We  obtain  (\ref{eq--2.a1}).  
   Suppose  in  addition  $\vec  G=\pi^{-1}(G)$.  
   Then  $\gamma $  is   a  (regular)     allowed  elementary   path  on  $\vec  G$  from  $u$  to  $v$  or  from  $v$  to  $u$   iff   
    $\gamma $  is  a  (regular)     allowed  elementary  path  on  $G$  from  $u$  to  $v$.  
    Thus the  equality  of  (\ref{eq--2.a1})  is  satisfied.  
          \end{proof}
          
          The  next  example  shows  that   the  condition  
          $\vec  G=\pi^{-1}(G)$  is  not  necessary  for   the  equality  of  (\ref{eq--2.a1}).  
          
          \begin{example}
          Let  $G$  be  the  complete  graph $K_n$  on  $n$  vertices. 
          Let $\vec  G$  be  any  digraph  such  that  its  underlying graph  is  $G$.  
          Then  for  any  distinct  two  vertices  $u,v\in  V_G$,  
          we  have  $d_{\vec  G}(u,v)=d_{G}(u,v)=1$.  
          \end{example}

    \begin{definition}(cf.  \cite[Definition~2.2]{lin3})  \label{def-1}
   Let  $\vec  G$  and  $\vec  G'$  be  digraphs.  
   A  {\it  morphism}  of  digraphs  $\varphi:  \vec  G\longrightarrow  \vec  G'$  
   is  a  map  $\varphi:  V_{\vec  G}\longrightarrow  V_{\vec G'}$  such  that  
   for  any  $(u,v)\in  E_{\vec  G}$,  
   either    $(\varphi(u),\varphi(v))\in  E_{\vec  G'}$  or  $\varphi(u)=\varphi(v)$.  
   In  particular,  let  $\vec  G=\vec  G'$.  An  {\it  automorphism}  of  $\vec  G$ 
   is  an  invertible  morphism  of  digraphs  $\varphi$  from  $\vec  G$  to  itself  
   such  that  its  inverse  is   also  a  morphism  of  digraphs.  
     \end{definition}

         Recall  that  graphs  are  $1$-dimensional  simplicial  complexes.  
       For  any  graphs  $   G$  and  $ G'$,     
   a  {\it  morphism}  of   graphs  $\varphi:    G\longrightarrow     G'$  
   is  a  simplicial  map,  i.e.  
     a  map  $\varphi:  V_{   G}\longrightarrow  V_{  G'}$  such  that  
   for  any  $\{u,v\}\in  E_{   G}$,  
   either   $\{\varphi(u),\varphi(v)\}\in  E_{   G'}$  or $\varphi(u)=\varphi(v)$. 
In  particular,  let  $   G=   G'$.  An  {\it  automorphism}  of  $  G$ 
   is  an  invertible  morphism  of   graphs  $\varphi$  from  $  G$  to  itself  
   such  that  its  inverse  is   also  a  morphism  of   graphs.

    \begin{lemma}\label{le-3.3a}
    Let  $\varphi:  \vec  G\longrightarrow  \vec  G'$  be  a  morphism  of  digraphs.  
     Let  $G$  and  $G'$   be  the  underlying  graphs  of   $\vec  G$  and  $\vec  G'$  respectively.  
     Then we  have  an  induced  morphism  of  graphs  
     $\varphi:  G\longrightarrow  G'$.
          \end{lemma}
     \begin{proof}
     From  the  morphism  of  digraphs  $\varphi:  \vec  G\longrightarrow  \vec  G'$,  
     we  have  an  induced  morphism  of  graphs  
     $\varphi:  G\longrightarrow  G'$  sending  any  edge  $\{u,v\} \in  E_G$  to
       an  edge  $\{\varphi(u),\varphi(v)\}\in  E_{G'}$  if  $\varphi(u)\neq \varphi(v)$  
       and  to  a  vertex  $\varphi(u)=\varphi(v)$  of  $G'$  otherwise.        
        \end{proof}

       \begin{lemma}\label{le-2.91}
        For  any    morphism    $\varphi:  \vec  G\longrightarrow  \vec  G'$
        (resp.  $\varphi:     G\longrightarrow   G'$)  
         and
          any  $u,v\in  V_{\vec  G}$  (resp.    $u,v\in  V_{   G}$),  we  have 
     \begin{eqnarray}\label{eq-25-03-15-1}
     d_{\vec  G} (u,v)\geq  d_{\vec  G'}(\varphi(u),\varphi(v))  ~~~
          ({\rm  resp.}~     d_{   G} (u,v)\geq  d_{  G'}(\varphi(u),\varphi(v))). 
     \end{eqnarray}
     \end{lemma}
     \begin{proof}
     By  Lemma~\ref{le-3.3a},  any    morphism  of  digraphs  
     $\varphi:  \vec  G\longrightarrow  \vec  G'$
     induces  a  morphism of  the  underlying  graphs  $\varphi:  G\longrightarrow  G'$.  
     We  only  prove  (\ref{eq-25-03-15-1})  for  the  digraph  case.  
     Let  $\gamma$  be  an       allowed  elementary   $n$-path  $v_0v_1\ldots  v_n$     in  $\vec  G$  
     such  that  $u=v_0$  and  $v=v_n$.  
     Then    
      $\varphi(\gamma)$  is  an    allowed  elementary  $n$-path   in  $\vec  G'$  
       from  $\varphi(v)$  to $\varphi(u)$.     
       Note  that  the  regularity  of  $\gamma$  does  not  imply  the  regularity  of  
       $\varphi(\gamma)$  while  the     regularity  of  $\varphi(\gamma)$   implies  the  regularity  of  
       $\gamma$.   
      We  obtain  (\ref{eq-25-03-15-1}).  
     \end{proof}

     \begin{definition}\label{def-2}
     \begin{enumerate}[(1)]
     \item
     We  say  that  a  morphism  of  digraphs  
     $\varphi:  \vec  G\longrightarrow  \vec  G' $  
     is  a  {\it  strong   totally geodesic  embedding}  of  digraphs   if  
    \begin{eqnarray}\label{eq-def-7-1}
          d_{\vec  G} (u,v)= d_{\vec  G'}(\varphi(u),\varphi(v))
  \end{eqnarray}
       for  any  $u,v\in  V_{\vec  G}$
     and  say  that  $\varphi$  is  a  {\it  strong   totally geodesic  immersion} with radius $n/2$
      if  (\ref{eq-def-7-1})  is  satisfied 
         for  any  $u,v\in  V_{\vec  G}$
     such that   $d_{\vec  G}(u,v)\leq  n$
     \footnote[2]{The  meaning   of  strong  totally  geodesic  embedding   comes 
      from  ramifications  of   totally  geodesic  submanifolds.  
     The  reason   is   explained  in  detail  
     in  \cite[Section~5.1]{regular}. }; 
     \item
        We  say  that  a  morphism  of   graphs  
     $\varphi:    G\longrightarrow    G' $
    is  a  {\it  strong   totally geodesic  embedding}    of  graphs  if  
          \begin{eqnarray}\label{eq-def-7-2}
          d_{   G} (u,v)= d_{   G'}(\varphi(u),\varphi(v))
   \end{eqnarray}
    for  any  $u,v\in  V_{ G}$
     and  say  that  $\varphi$  is  a  {\it  strong   totally geodesic  immersion} with radius $n/2$
      if   (\ref{eq-def-7-2})  is  satisfied  
        for  any  $u,v\in  V_{   G}$
     such  that  $d_{   G}(u,v)\leq  n$.  
     \end{enumerate}
         \end{definition}

  Let   $r(\vec  G)$    and   $ r( G)$  
        be    values   in  $\mathbb{N}\cup \{\infty\}$  given by 
      \begin{eqnarray*} 
      r(\vec  G) &=& \frac{1}{2} \sup \big \{d_{\vec  G}(u,v)\mid  u,v\in V_{\vec  G}\big\},\\     
            r(   G) &=& \frac{1}{2} \sup \big \{d_{   G}(u,v)\mid  u,v\in V_{   G}\big\}.     
      \end{eqnarray*}

            \begin{proposition}\label{pr-2.b2}
            If  $\varphi:  \vec  G\longrightarrow  \vec  G' $  (resp.   $\varphi:     G\longrightarrow    G' $) 
  is  a   strong   totally geodesic  immersion  with radius 
  ${n}/{2}\geq  r(\vec  G)$  (resp.    ${n}/{2}\geq  r(   G)$),  
  then  $\varphi$  is  a  strong   totally geodesic   embedding.  
            \end{proposition}
            
            \begin{proof}
         It  follows from      ${n}/{2}\geq  r(\vec  G)$   
           that  
   $d_{\vec  G}(u,v)\leq  n$  for  any  $u,v\in  V_{\vec  G}$.  
            Let
            $\varphi:  \vec  G\longrightarrow  \vec  G' $   be 
             a   strong   totally geodesic  immersion  with radius 
            ${n}/{2}\geq  r(\vec  G)$.  
            With  the  help  of   Definition~\ref{def-2},  
             we  have  (\ref{eq-def-7-1})    for  any  $u,v\in  V_{\vec  G}$.  
             Thus   $\varphi$  is  a  strong   totally geodesic   embedding.   
            \end{proof}

    A  digraph  
         is  {\it  path-connected} if  for  any two  vertices    there  exists 
        an  allowed  elementary  path    from  one  of  the  vertices  to  the  other.  
        A   graph  is   {\it  path-connected} if  it  is  path-connected  
        as  a  $1$-dimensional  simplicial  complex. 
        Note  that     $d_{\vec  G}$  (resp.  $d_G$)  has  finite  value  iff 
            $\vec  G$  (resp.  $G$)   is  path-connected.  
           By  Lemma~\ref{le-2.1},  
            if a digraph  $\vec  G$  is  path-connected,  then  its  underlying  
        graph  $G$  is  path-connected.  However,  the   converse  is  not  true.  
        
 \begin{proposition}
 \label{le-2.2}
 Suppose  $\vec  G$  (resp.  $G$)  is  path-connected.  
 If  $\varphi:  \vec  G\longrightarrow  \vec  G' $  
 (resp.   $\varphi:     G\longrightarrow    G' $) 
  is  a   strong   totally geodesic  immersion  with radius $n/2$  for   any  $n\in \mathbb{N}$,  
 then  $\varphi$  is  a  strong   totally geodesic  embedding.
  \end{proposition}
 
 \begin{proof}
Let   $u,v\in  V_{\vec  G}$. 
 Since   $\vec  G$     is  path-connected,  
 there  exists  an  allowed  elementary  path  on  $\vec G$  
 from  $u$ to  $v$  or  from  $v$  to  $u$.   
  Thus  $d_{\vec  G}(u,v)<\infty$. 
  Choose  $n\in \mathbb{N}$  such  that  $n >d_{\vec  G}(u,v)$.  
  Since  $\varphi:  \vec  G\longrightarrow  \vec  G' $   
  is  a   strong   totally geodesic  immersion  with radius $n/2$,  
  we  have    (\ref{eq-def-7-1}).  
   Thus   $\varphi$  is  a  strong   totally geodesic   embedding. 
 \end{proof}

\begin{proposition}
\label{co-2.cc1}
Let  $\vec  G$  be  a  digraph  with  underlying  graph  $G$.  
Let  ${\rm  Aut}(\vec  G)$  be  the  automorphism  group  of  $\vec   G$ 
and  let  ${\rm  Aut}(   G)$  be  the  automorphism  group  of  $   G$.  
Then  
\begin{enumerate}[(1)]
\item
${\rm  Aut}(\vec  G)$   is  a  subgroup  of   ${\rm  Aut}(   G)$;
\item
any  automorphism  of  $\vec  G$  is  strong    totally geodesic  with  respect  to  $d_{\vec  G}$;  
\item
any  automorphism  of  $ G$  is  strong    totally geodesic  with  respect  to  $d_{ G}$.   
\end{enumerate}   
\end{proposition}

\begin{proof}
(1)   By  Lemma~\ref{le-3.3a},   any  automorphism   of  $\vec   G$  
is  an  automorphism   of     $  G$.  
Thus  ${\rm  Aut}(\vec  G)$   is  a  subgroup  of   ${\rm  Aut}(   G)$,  both  of  which  
are  subgroups  of  the  permutation  group  of  the  vertices.

(2)    Let  $\varphi\in  {\rm  Aut}(\vec  G)$.      
 By  Lemma~\ref{le-2.91},   for  any  vertices   $u$  and  $v$,  we  have
$d_{\vec  G} (u,v)\geq  d_{\vec  G }(\varphi(u),\varphi(v))$.    
  Substituting   $\varphi$  with  $\varphi^{-1}$,  we  have  
  $d_{\vec  G} (\varphi(u),\varphi(v))\geq  d_{\vec  G }(u,v)$.  
  Thus     $  d_{\vec  G }(u,v)=d_{\vec  G} (\varphi(u),\varphi(v))$,  
  which  implies  that  $\varphi$  is   strong    totally geodesic  with  respect  to  $d_{\vec  G}$.

  (3)   The  proof  is  analogous with  (2).  
\end{proof}

      The  following  example  shows  that   a  strong   totally geodesic  immersion  of  (di)graphs  may  not  be  induced  by  any  injections  of  the  vertices.

       \begin{example}\label{ex-3.3}
      Let  $\vec {L}$  be  the  {\it  line digraph}  with  vertices  $v_k$    and  arcs  
      $(v_k,  v_{ k+1})$  for  all  $k\in \mathbb{Z}$.  
      Then   for  any  $p,q\in\mathbb{Z}$, 
       \begin{eqnarray*}
       d_{\vec {L}}(v_p,v_q)= |p-q|.   
       \end{eqnarray*}
      Let  $\vec {C}_r$  be the  {\it  cyclic  digraph}  with  vertices  $u_{[k]}$    and   arcs  
      $(u_{[k]},  u_{ [k+1]})$  
      for   all  $[k]\in \mathbb{Z}/r\mathbb{Z}$.   Then  for  any 
       $[p], [q]\in   \mathbb{Z}/r\mathbb{Z}$,  
       \begin{eqnarray*}
       d_{\vec {C}_r}(u_{[p]},u_{[q]})=  \min\big\{  |p_0-q_0|,   r- |p_0-q_0|    \big\}  
       \end{eqnarray*}
       where  $ 0\leq  p_0,q_0 < r$  are  the  representatives  
       of  the  residue  classes  $[p]$   and  $[q]$  respectively.   
      Let 
      $\varphi:  \vec {L}\longrightarrow  \vec {C}_r$  
      be  the   canonical  morphism  of  digraphs  sending  $v_k$  to  $u_{[k]}$  and  sending  
       $(v_k,  v_{ k+1})$  to  $(u_{[k]},  u_{ [k+1]})$  for  any  $k\in \mathbb{Z}$.  
       Then  
      for any     $|p-q|\leq  r/2$,  
      we  have  
      \begin{eqnarray*}
      d_{\vec {C}_r}(u_{[p]}, u_{[q]})=d_{\vec {L}}(v_p,v_q).  
      \end{eqnarray*}
       Thus   $\varphi$ 
      is  a  strong   totally geodesic  immersion  of  digraphs  with  radius  $n/2$  for  any  $n\leq  r/2$. 
    On  the  other  hand,  
        for any     $  r/2<  |p-q|$,
         we  have  
         \begin{eqnarray*}
         d_{\vec {C}_r}(u_{[p]}, u_{[q]})<  d_{\vec {L}}(v_p,v_q). 
         \end{eqnarray*}  
          Thus  $\varphi$ 
      is  not  a  strong   totally geodesic  immersion  of  digraphs  
      with  radius  $n/2$  for  any  $n>  r/2$.

      Let  $L$  and $C_r$  be  the  underlying  graphs  of  $\vec  L$  and  $\vec C_r$  respectively.  
      Note  that  $d_{L}=d_{\vec  L}$  and  $d_{C_r}=d_{\vec  C_r}$.  
      The induced  morphism  of  graphs  $\varphi:  L\longrightarrow  C_r$  is  a  covering  map
      of  graphs.  
      It   is  
        a  strong   totally geodesic  immersion  of  graphs   with  radius  $n/2$  for  any  $n\leq  r/2$
         is  not  a  strong   totally geodesic  immersion  of  graphs  with  radius  $n/2$
           for  any  $n>  r/2$.  
      \end{example}

    The  following  example   shows  that
        the  path-connected  condition  in   Proposition~\ref{le-2.2}  is   essential.

      \begin{example}\label{ex-3.5}
      Let  $\vec  Z$  be  the  zigzag  digraph  with  vertices  $v_k$  for  all  $k\in  \mathbb{Z}$  
      and  with  arcs   $(v_{2l}, v_{2l-1})$,  $(v_{2l},  v_{2l+1})$   for  all   $l\in  \mathbb{Z}$.  
      Then  for  any  $p,q\in  \mathbb{Z}$  with  $p\neq  q$,  
      \begin{eqnarray*}
      d_{\vec  Z}(v_p,  v_q) = 
      \begin{cases}
      1   &{\rm~if~}  |p-q|=1,\\
      \infty   &{\rm~otherwise}.
      \end{cases}
      \end{eqnarray*}
      Let  $\vec  I_2$  be  the  segment  digraph  with  
        two  vertices  $ u_0 $  and  $ u_1$  
      and     a  unique   arc   $(u_0, u_1)$.  
      Then  
      \begin{eqnarray*}
      d_{\vec  I_2} (u_0, u_1)=1.  
      \end{eqnarray*}
      Let  $\varphi:  \vec  Z\longrightarrow   \vec  I_2$  be  the   canonical  morphism  of  digraphs   
      sending  $v_{2l}$  to  $u_0$  and  sending  $v_{2l+1}$  to $u_1$  for  any  $l\in \mathbb{Z}$.  
      Then   
      $\varphi$  sends  both the  arcs  $(v_{2l},  v_{2l-1})$  and   $(v_{2l},  v_{2l+1})$ 
      to  $(u_0,u_1)$  for  any  $l\in \mathbb{Z}$.  
     For  any  $n\in  \mathbb{N}$,     $\varphi$  is  a  strong   totally geodesic  immersion 
      of  digraphs   with radius $n/2$.  
      Note that  $\vec  Z$  is  not  path-connected  and      
 $\varphi:  \vec  Z\longrightarrow   \vec  I_2$     
   is     not   a  strong   totally geodesic  embedding.

      Let  $Z$  and  $I_2$  be the  underlying  graphs  of   $\vec Z$  and  $\vec  I_2$  respectively.  
      Note  that  $Z= L$  and     $I_2=C_2$.   
          By  letting   $m=2$  in  the  second  paragraph  of   Example~\ref{ex-3.3},  
      we  have  that   $\varphi$  is  
        a  strong   totally geodesic  immersion  of  graphs   with  radius  $1/2$  
         is  not  a  strong   totally geodesic  immersion  of  graphs  with  radius  $n/2$
           for  any  $n>  1$.  
      \end{example}

   \section{Configuration  spaces  and  independence  complexes  for  digraphs}
   \label{ssa.3}
   
   In  this  section,  we  study   configuration  spaces  for  digraphs  and  the  underlying graphs.  
   In  Subsection~\ref{ssa-3.1},
   we  construct the independence  complexes  by  using  the  configuration  spaces  as  
   the  skeletons. 
   In  Subsection~\ref{ssa-3.2},  we  give  an   
  isometric  embedding  from  the  configuration  space  of  the  underlying  graph 
    into the configuration  space  of the digraph,  
  which  induces  a    simplicial  embedding  between  the  independence  complexes. 
   In  Subsection~\ref{ssa-3.3},  
   we  prove  that  a  strong  totally  geodesic  embedding  of  (di)graphs 
   will  induce     isometric  embeddings  between  the  configuration  spaces  and  
   consequently  induce embeddings  between   the  independence  complexes. 
   In  Subsection~\ref{s-7},  
   we  describe    geometric  realizations  of  the  independence  complexes  
   by    affinely  regular  embeddings  of  (di)graphs.  
   
   \subsection{The configuration  spaces  and  the  independence  complexes}
   \label{ssa-3.1}
   
   Let  $\vec  G$  be  a  digraph.  
   Then  $(V_{\vec  G},   d_{\vec  G})$  is  a  metric  space.  
   We  have  the  $k$-fold  product  metric  space  
  $((V_{\vec  G})^{k},  ( d_{\vec  G})^{k})$.  
  For  any  positive  integers   $k$  and  any   $1\leq   n<m\leq  \infty$,  
  consider   the  $k$-th      {\it  ordered  constraint   configuration  space}    
    \begin{eqnarray}\label{eq-79.a-a1}
  {\rm  Conf}_k(V_{\vec  G}, \frac{n}{2}, \frac{m}{2})=  \{(v_1,\ldots, v_k)\in  (V_{\vec G})^k\mid
    n<  d_{\vec  G}(v_i,v_j)\leq  m  {\rm~for~any~}  i\neq  j \}.          
  \end{eqnarray} 
  Note  that  (\ref{eq-79.a-a1})  is  a   subspace   of   
  $((V_{\vec  G})^{k},  ( d_{\vec  G})^{k})$. 
     The  symmetric group  $\Sigma_k$  acts  on    (\ref{eq-79.a-a1})    
  freely  
  by  permuting  the  coordinates 
  \begin{eqnarray}\label{eq-5.zxcv1}
 \sigma(v_1,\ldots,v_k) =  (v_{\sigma(1)},\ldots, v_{\sigma(k)}),~~~~~~\sigma\in\Sigma_k 
  \end{eqnarray}
  such  that  for  any $(u_1,\ldots,u_k)$  and  $(v_1,\ldots,v_k)$, 
  \begin{eqnarray}\label{eq-5.isom1}
 ( d_{\vec  G})^{k}(\sigma(u_1,\ldots,u_k), \sigma(v_1,\ldots,v_k)) =  ( d_{\vec  G})^{k}((u_1,\ldots,u_k), (v_1,\ldots,v_k)).  
  \end{eqnarray}
  Thus  the  $\Sigma_k$-action  (\ref{eq-5.zxcv1})  on     (\ref{eq-79.a-a1})     
   is  isometric.  
With  the  help  of      (\ref{eq-5.zxcv1}),
   we  define  the  $k$-th   {\it  unordered   constraint configuration  space}   to  be  the  orbit  space 
   \begin{eqnarray}\label{eq-uccs-2}
 {\rm  Conf}_k(V_{\vec  G}, \frac{n}{2}, \frac{m}{2})\Big/ \Sigma_k  = 
  \Big \{\{v_1,\ldots, v_k\}\in  2^{V_{\vec G}}~\Big |~   
   n< d_{\vec  G}(v_i,v_j)\leq   m  {\rm~for~any~}  i\neq  j  \Big \}.  
 \end{eqnarray}
    It  follows from  (\ref{eq-5.isom1})  that  
    there  is  an  induced  metric  $ ( d_{\vec  G})^{k}  /\Sigma_k$  on  
     (\ref{eq-uccs-2}).

    In  particular,  if  we  let  $m=\infty$  in (\ref{eq-79.a-a1})  and   (\ref{eq-uccs-2}),  
      then  they  give   the  $k$-th     ordered     configuration  space  of  hard  spheres   
      and  the  $k$-th  unordered     configuration  space  of  hard  spheres, 
       with  radius  $n/2$.

    \begin{lemma}\label{le-2503a1z}
  For  any  digraph  $\vec  G$  and  any  positive  integer  $k$,    
  we have  a  $\Sigma_k$-equivariant    double-filtration   
    \begin{eqnarray}\label{eq-e.1}
    {\rm  Conf}_k(V_{\vec  G}, -,-)=\Big\{{\rm  Conf}_k(V_{\vec  G}, \frac{n}{2}, \frac{m}{2})
    ~\Big|~  1\leq  n<m\leq  \infty \Big\}
    \end{eqnarray}
       such  that 
    \begin{eqnarray*}
    {\rm  Conf}_k(V_{\vec  G}, \frac{n_1}{2}, \frac{m}{2})  \supseteq 
     {\rm  Conf}_k(V_{\vec  G}, \frac{n_2}{2}, \frac{m}{2})
    \end{eqnarray*}
   is  an   isometric  embedding  for  any  $n_1<n_2<m$
    and 
        \begin{eqnarray*}
    {\rm  Conf}_k(V_{\vec  G}, \frac{n }{2}, \frac{m_1}{2})  \subseteq 
     {\rm  Conf}_k(V_{\vec  G}, \frac{n }{2}, \frac{m_2}{2})
    \end{eqnarray*}
      is  an   isometric  embedding  for  any  $n<m_1<m_2$. 
    \end{lemma}
    
    \begin{proof}
    The   double-filtration   (\ref{eq-e.1})   follows  from  (\ref{eq-79.a-a1}).  
    The  embeddings  are  isometries  with  respect to  $ ( d_{\vec  G})^{k}$.  
    The  $\Sigma_k$-invariance  follows  from  (\ref{eq-5.zxcv1}).  
    \end{proof}
    
    \begin{corollary}\label{co-2503-bma1}
    For  any  digraph  $\vec  G$  and  any  positive  integer  $k$,    
  we have  a       double-filtration   
    \begin{eqnarray}\label{eq-e.2}
    {\rm  Conf}_k(V_{\vec  G}, -,-) /\Sigma_k=
    \Big\{{\rm  Conf}_k(V_{\vec  G}, \frac{n}{2}, \frac{m}{2})\Big/\Sigma_k    
    ~\Big|~  1\leq    n<m \leq  \infty\Big\}
    \end{eqnarray}
    such  that 
    \begin{eqnarray*}
    {\rm  Conf}_k(V_{\vec  G}, \frac{n_1}{2}, \frac{m}{2}) \Big/\Sigma_k \supseteq 
     {\rm  Conf}_k(V_{\vec  G}, \frac{n_2}{2}, \frac{m}{2})\Big/\Sigma_k
    \end{eqnarray*}
   is  an   isometric  embedding  for  any  $n_1<n_2<m$
    and 
        \begin{eqnarray*}
    {\rm  Conf}_k(V_{\vec  G}, \frac{n }{2}, \frac{m_1}{2})\Big/\Sigma_k  \subseteq 
     {\rm  Conf}_k(V_{\vec  G}, \frac{n }{2}, \frac{m_2}{2})\Big/\Sigma_k
    \end{eqnarray*}
      is  an   isometric  embedding  for  any  $n<m_1<m_2$.
    \end{corollary}
    
    \begin{proof}
  Taking  the  $\Sigma_k$-orbit  spaces  in  the   double-filtration  
  (\ref{eq-e.1})  in  Lemma~\ref{le-2503a1z},    
    we  obtain  the   double-filtration   (\ref{eq-e.2}).  
       The  embeddings  are  isometries  with  respect to  $ ( d_{\vec  G})^{k}/\Sigma_k$.  
    \end{proof}

    \begin{corollary}
    \label{co-250328-cv}
    For  any  digraph  $\vec  G$  and  any  positive  integer  $k$, 
    we  have a    double-persistent  isometric    covering  map  
    \begin{eqnarray}\label{eq-co-0328-2}
    \pi_{\vec  G,k}(-,-):  ({\rm  Conf}_k(V_{\vec  G}, -,-),( d_{\vec  G})^{k})
    \longrightarrow  ({\rm  Conf}_k(V_{\vec  G}, -,-) /\Sigma_k,  ( d_{\vec  G})^{k}/\Sigma_k). 
    \end{eqnarray}
        \end{corollary}
    
    \begin{proof}
    Let  $1\leq  n<m\leq  \infty$. 
       Since  the  $\Sigma_k$-action  is  free  and  isometric  on  ${\rm  Conf}_k(V_{  G}, n/2,m/2)$,  
    we  have  an  isometric   covering  map  
    \begin{eqnarray}\label{eq-cover}
    \pi_{\vec  G, k}(\frac{n}{2},\frac{m}{2}):     {\rm  Conf}_k(V_{\vec  G}, \frac{n}{2},\frac{m}{2}) \longrightarrow 
        {\rm  Conf}_k(V_{\vec  G}, \frac{n}{2},\frac{m}{2})\Big/\Sigma_k.  
    \end{eqnarray}
     For  any  $1\leq  n_1\leq  n_2<m$,  the  diagram  commutes 
     \begin{eqnarray*}
    \xymatrix{
     {\rm  Conf}_k(V_{\vec  G}, \dfrac{n_2}{2}, \dfrac{m}{2}) 
      \ar[rr]^-{\pi_{\vec  G, k}(\frac{n_2}{2}, \frac{m}{2})} \ar[d] 
     &&{\rm  Conf}_k(V_{\vec  G}, \dfrac{n_2}{2}, \dfrac{m}{2})\Big/\Sigma_k\ar[d]\\
      {\rm  Conf}_k(V_{\vec  G}, \dfrac{n_1}{2}, \dfrac{m}{2})  
      \ar[rr]^-{\pi_{\vec  G, k}(\frac{n_1}{2}, \frac{m}{2})}  
     &&{\rm  Conf}_k(V_{\vec  G}, \dfrac{n_1}{2}, \dfrac{m}{2})\Big/\Sigma_k 
    }
    \end{eqnarray*}
    where  the  vertical  maps  are  canonical  inclusions.  
     For  any  $1\leq  n <m_1\leq  m_2$,  the  diagram  commutes 
     \begin{eqnarray*}
    \xymatrix{
     {\rm  Conf}_k(V_{\vec  G}, \dfrac{n}{2}, \dfrac{m_1}{2})  
     \ar[rr]^-{\pi_{\vec  G, k}(\frac{n}{2}, \frac{m_1}{2})} \ar[d] 
     &&{\rm  Conf}_k(V_{\vec  G}, \dfrac{n }{2}, \dfrac{m_1}{2})\Big/\Sigma_k\ar[d]\\
      {\rm  Conf}_k(V_{\vec  G}, \dfrac{n }{2}, \dfrac{m_2}{2})  
      \ar[rr]^-{\pi_{\vec  G, k}(\frac{n }{2}, \frac{m_2}{2})}  
     &&{\rm  Conf}_k(V_{\vec  G}, \dfrac{n }{2}, \dfrac{m_2}{2})\Big/\Sigma_k 
    }
    \end{eqnarray*}
    where  the  vertical  maps  are  canonical  inclusions.  
    Therefore,    taking  double-persistence  in  (\ref{eq-cover}),   we  obtain      (\ref{eq-co-0328-2}).  
    \end{proof}

    \begin{definition}\label{def-5.a1}
  We  define   the  {\it   constraint   independence  complex}      of  $\vec  G$  to  be  the  simplicial  complex    
  \begin{eqnarray}\label{eq-3.a1}
  {\rm  Ind}(\vec  G,  \frac{n}{2},  \frac{m}{2})=  \bigcup_{k\geq  1}  
  \Big( {\rm  Conf}_k(V_{\vec  G}, \frac{n}{2},  \frac{m}{2})\Big/ \Sigma_k\Big)   
  \end{eqnarray}
  such  that  for  any  $k \geq  1$,   
 the    $(k-1)$-simplices   of  (\ref{eq-3.a1})   are      given  by
    the    elements  of
      ${\rm  Conf}_k(V_{\vec  G}, n/2,   {m}/{2}) / \Sigma_k$.
  \end{definition}
  
  \begin{corollary}
  \label{le-7.012m}
  For  any  digraph  $\vec  G$, 
    we  have a     double-filtration   of     simplicial  complexes 
    \begin{eqnarray}\label{eq-09.h1}
    {\rm  Ind}(\vec  G, -,  -)= \Big\{  {\rm  Ind}(\vec  G,  \frac{n}{2},  \frac{m}{2}) 
   ~\Big|~  1\leq  n<m  \leq  \infty  \Big\}
    \end{eqnarray}
     such  that 
    \begin{eqnarray*}
     {\rm  Ind}(\vec  G,  \frac{n_1}{2},  \frac{m}{2})   \supseteq 
      {\rm  Ind}(\vec  G,  \frac{n_2}{2},  \frac{m}{2}) 
    \end{eqnarray*}
   is  an  embedding  of  simplicial  complexes  for  any  $n_1<n_2<m$
    and 
        \begin{eqnarray*}
     {\rm  Ind}(\vec  G,  \frac{n}{2},  \frac{m_1}{2})  \subseteq 
     {\rm  Ind}(\vec  G,  \frac{n}{2},  \frac{m_2}{2}) 
    \end{eqnarray*}
      is  an  embedding  of  simplicial  complexes  for  any  $n<m_1<m_2$. 
  \end{corollary}
  
  \begin{proof}
  Apply  (\ref{eq-3.a1})  to   (\ref{eq-e.2}).   We  obtain the  
  double-filtration  of     simplicial  complexes  
  (\ref{eq-09.h1}).  
  \end{proof}
  
  Let  $G$  be a  graph.  
  Then  $(V_G,  d_G)$  is  a  metric  space  and    
  $((V_{   G})^{k},  ( d_{   G})^{k})$  is  the $k$-fold  product  metric  space.  
    For any  positive  integer  $k$  and  any  $1\leq  n<m\leq  \infty$,    
    consider   the  $k$-th   {\it  ordered  constraint   configuration  space}   
    \begin{eqnarray}\label{eq-79.b-b1}
  {\rm  Conf}_k(V_{   G}, \frac{n}{2}, \frac{m}{2})=  \{(v_1,\ldots, v_k)\in  (V_{  G})^k\mid
    n<  d_{   G}(v_i,v_j)\leq  m  {\rm~for~any~}  i\neq  j \},        
  \end{eqnarray} 
    which  is  a  subspace  of   
  $((V_{  G})^{k},  ( d_{   G})^{k})$.    
  There  is   a  free  and   isometric  $\Sigma_k$-action  on   (\ref{eq-79.b-b1})
    by  permuting  the  coordinates.
    The  $k$-th  {\it  unordered   constraint configuration  space}  is  the  orbit  space 
   \begin{eqnarray*}
 {\rm  Conf}_k(V_{   G}, \frac{n}{2}, \frac{m}{2})\Big/ \Sigma_k  = 
  \Big \{\{v_1,\ldots, v_k\}\in  2^{V_{  G}}~\Big |~   
   n< d_{   G}(v_i,v_j)\leq   m  {\rm~for~any~}  i\neq  j  \Big \} 
 \end{eqnarray*}
   with   the  induced  metric  $ ( d_{  G})^{k}  /\Sigma_k$.

 \begin{lemma}\label{le-2503b2z}
  For  any   graph  $  G$  and  any  positive  integer  $k$,    
  we have  a  $\Sigma_k$-equivariant    double-filtration   
    \begin{eqnarray}\label{eq-2b-e.1}
    {\rm  Conf}_k(V_{ G}, -,-)=\Big\{{\rm  Conf}_k(V_{ G}, \frac{n}{2}, \frac{m}{2})
    ~\Big|~   1\leq   n<m \leq  \infty\Big\}
    \end{eqnarray}
    such that  the  inclusions  are  isometric  embeddings,  which      induces  a    double-filtration   
    \begin{eqnarray}\label{eq-2b-e.2}
    {\rm  Conf}_k(V_{ G}, -,-) /\Sigma_k=
    \Big\{{\rm  Conf}_k(V_{  G}, \frac{n}{2}, \frac{m}{2})\Big/\Sigma_k    
    ~\Big|~  1\leq   n<m\leq   \infty \Big\}
    \end{eqnarray} 
     such that  the  inclusions  are  isometric  embeddings.  
    \end{lemma}
    
    \begin{proof}
    The  proof  is    an  analog  of  Lemma~\ref{le-2503a1z} 
and  Corollary~\ref{co-2503-bma1}.  
    \end{proof}

    \begin{corollary}
    \label{co-250328-cv-a}
    For  any   graph  $  G$  and  any  positive  integer  $k$, 
    we  have a    double-persistent   isometric   covering  map  
    \begin{eqnarray}\label{eq-co-0328-22x}
    \pi_{ G,k}(-,-):  ({\rm  Conf}_k(V_{ G}, -,-),( d_{  G})^{k})
    \longrightarrow  ({\rm  Conf}_k(V_{  G}, -,-) /\Sigma_k,  ( d_{ G})^{k}/\Sigma_k). 
    \end{eqnarray}
        \end{corollary}
        
        \begin{proof}
        The  corollary  follows  from Lemma~\ref{le-2503b2z}. 
        The  proof  is  an  analog  of  Corollary~\ref{co-250328-cv}.    
        \end{proof}

    \begin{definition}\label{def-5.b2}
  We  define   the  {\it   constraint   independence  complex}      of  $  G$  to  be  the  simplicial  complex    
  \begin{eqnarray}\label{eq-3.b2}
  {\rm  Ind}(  G,  \frac{n}{2},  \frac{m}{2})=  \bigcup_{k\geq  1}  
  \Big( {\rm  Conf}_k(V_{   G}, \frac{n}{2},  \frac{m}{2})\Big/ \Sigma_k\Big)   
  \end{eqnarray}
  such  that  for  any  $k \geq  1$,   
 the    $(k-1)$-simplices   of  (\ref{eq-3.b2})  are      given  by
    the    elements  of
      ${\rm  Conf}_k(V_{  G}, n/2,   {m}/{2}) / \Sigma_k$.
  \end{definition}

     \begin{corollary}
  \label{le-7.2b-12m}
  For  any   graph  $  G$, 
    we  have a     double-filtration   of     simplicial  complexes 
    \begin{eqnarray}\label{eq-09.2b-h1}
    {\rm  Ind}( G, -,  -)= \Big\{  {\rm  Ind}(  G,  \frac{n}{2},  \frac{m}{2}) 
   ~\Big|~   1\leq   n<m\leq   \infty  \Big\}. 
    \end{eqnarray} 
  \end{corollary}
  
  \begin{proof}
The  corollary  follows from (\ref{eq-2b-e.2})  and   is    an  analog  of    Corollary~\ref{le-7.012m}.    
  \end{proof}

Given  a  simplicial  complex  $\mathcal{K}$,  
An  {\it  automorphism}  of  $\mathcal{K}$  is  an  invertible   simplicial  map  
$\varphi:  \mathcal{K}\longrightarrow  \mathcal{K}$   
such  that  the  inverse  $\varphi^{-1}:  \mathcal{K}\longrightarrow  \mathcal{K}$ 
is  also  a  simplicial  map.  
The  collection  of  all  the  automorphisms  of  $\mathcal{K}$  is  a  group,
which  will be  called  the   {\it  automorphism  group}  of  $\mathcal{K}$  
and  denoted  by  ${\rm  Aut}(\mathcal{K})$.

\begin{corollary}\label{pr-05-05-22-1}
For  any  digraph  $\vec  G$   with  its   underlying  graph  $G$   and  any 
  $1\leq  n<m\leq \infty $,  we  have    canonical   group  homomorphisms 
\begin{eqnarray}\label{eq-250522-a1}
\alpha_{\vec  G}(\frac{n}{2},\frac{m}{2}): &&  {\rm  Aut}(\vec   G)\longrightarrow    {\rm  Aut}({\rm Ind}(\vec  G,\frac{n}{2},\frac{m}{2})), 
 \\
\alpha_{ G}(\frac{n}{2},\frac{m}{2}):  && {\rm  Aut}(   G)\longrightarrow    {\rm  Aut}({\rm Ind}(   G,\frac{n}{2},\frac{m}{2})).  
\label{eq-250522-a2}
\end{eqnarray}
\end{corollary}

\begin{proof}
Let   $\varphi\in  {\rm  Aut}(\vec  G)$.  
By   Proposition~\ref{co-2.cc1}~(2),  
 $\varphi$  induces  a  simplicial  map  from ${\rm  Ind}(\vec  G,  n/2,m/2)$  to  itself
   sending     
  a  simplex  $\{v_0, \ldots,v_n\}$     to  the  simplex  
$\{\varphi(v_0),\ldots, \varphi(v_n)\}$.    
Similarly,    
$\varphi^{-1}$    induces  a  simplicial  map  from ${\rm  Ind}(\vec  G,  n/2,m/2)$  to  itself 
 sending     
  a  simplex  $\{v_0, \ldots,v_n\}$    to  the  simplex  
$\{\varphi^{-1}(v_0),\ldots, \varphi^{-1}(v_n)\}$. 
 We  obtain  the  map  (\ref{eq-250522-a1}), 
   which  can  be  verified   directly  to  be     a  group  homomorphism.  
Similarly, we  have the  group  homomorphism    (\ref{eq-250522-a2}).  
\end{proof}

    \begin{example}\label{ex-0.8}
   Consider  the  line  digraph  $\vec  L$  (cf.  Example~\ref{ex-3.3}) 
    and  the  zigzag  digraph  $\vec  Z$  (cf.  Example~\ref{ex-3.5}).   
    The  underlying  graph  of  both  $\vec  L$  and  $\vec  Z$  is  the  line  graph  
    $L$   with  vertices  $v_k$    and  edges  
      $\{v_k,  v_{ k+1}\}$  for  all  $k\in \mathbb{Z}$.
      \begin{enumerate}[(1)]
      \item
      The   constraint  independence  complex  of  $\vec  L$  is  given  by  
      \begin{eqnarray*}
      {\rm  Ind}(\vec  L, \frac{n}{2},\frac{m}{2}) 
      =  \bigcup_{k=0}^\infty 
      \big\{\{v_{i_0},v_{i_1}, \ldots, v_{i_k}\}\mid
         n < i_{l}-i_{j} \leq  m{\rm~for~any~}   
      0\leq  j<l \leq  k \big\} 
      \end{eqnarray*}
      such  that  for  any  $n>1$  fixed,  
      \begin{eqnarray*}
      {\rm  Ind}(\vec  L, \frac{n}{2},\infty) 
      =\bigcup_{n<m<\infty }{\rm  Ind}(\vec  L, \frac{n}{2},\frac{m}{2});    
      \end{eqnarray*}
      \item
      The   constraint  independence  complex  of  $\vec  Z$  is  given  by  
      \begin{eqnarray*}
        {\rm  Ind}(\vec  Z, \frac{n}{2}, \frac{m}{2})= \emptyset 
      \end{eqnarray*}
      for  any  $n<m<\infty$  
      and 
      \begin{eqnarray*}
      {\rm  Ind}(\vec  Z, \frac{n}{2},\infty)=
        \bigcup_{k=0}^\infty 
        \big\{\{v_{i_0},v_{i_1}, \ldots, v_{i_k}\}\mid
           1 <i_{j}-i_{j-1}{\rm~for~any~}   
      1\leq  j\leq  k \big\},   
      \end{eqnarray*}
      which  does  not  depend   on  the  choice  of  $n\geq  1$;   
      \item
      The  constraint  independence  complex  of  $L$  is  given  by  
      \begin{eqnarray*}
      {\rm  Ind}( L, \frac{n}{2},\frac{m}{2})=   {\rm  Ind}(\vec  L, \frac{n}{2},\frac{m}{2}).      
      \end{eqnarray*}
      \end{enumerate}
   \end{example}

   \begin{example}
   Consider  the  cyclic  digraph  $\vec  C_r$ 
   with  its  underlying  graph  $C_r$   (cf.  Example~\ref{ex-3.3}).  
    Then  for  any  $n\geq  1$,    both  of  the  constraint  independence  complexes  
    ${\rm  Ind}(\vec  C_r,  {n}/{2},{m}/{2})$    and     $ {\rm  Ind}(   C_r,  {n}/{2},{m}/{2}) $ 
      are   given  by  
      \begin{eqnarray*}
          \bigcup_{k=0}^{[r/n]-1} \big\{\{v_{i_0},v_{i_1}, \ldots, v_{i_k}\}  &\mid &
          0\leq  i_0<\cdots  <i_k < r{\rm~such~that~}\\
       &&    n <\min\{ i_{l}-i_{j},   i_j+r-i_l \}  \leq  m\\
       && {\rm  for~any~}   
      0\leq  j<l \leq  k 
        \big\}.     
      \end{eqnarray*}
         In  particular,  
      \begin{enumerate}[(1)]
      \item
         if  $n>r/2$,  then  ${\rm  Ind}(\vec  C_r,  {n}/{2},{m}/{2})$    and     $ {\rm  Ind}(   C_r,  {n}/{2},{m}/{2}) $    are  of  dimension  zero  and  are  the  discrete  vertex  set  $\mathbb{Z}/r \mathbb{Z}$;    
    \item
    if  $m=n+1$,  then  ${\rm  Ind}(\vec  C_r,  {n}/{2},{m}/{2})$   
     and     $ {\rm  Ind}(   C_r,  {n}/{2},{m}/{2}) $    are  of  dimension  $0$,  $1$  or $2$. 
     Moreover,   they  are  of  dimension  $1$  if  
     $r=2(n+1)$   in  which  case   the $1$-simplices  are  antipodal  vertices;   and  
     they  are  of  dimension  $2$     
   if      $r=3(n+1)$  in  which  case  
   the  $2$-simplices  are  equilateral triangles  of  vertices.  
    \end{enumerate}
   \end{example}

    \subsection{Configuration  spaces  and  independence  complexes  for   digraphs
    and  their  underlying  graphs}
    \label{ssa-3.2}

   \begin{proposition}\label{pr-3.a1}
   For  any  digraph  $\vec  G$  with  its  underlying graph  $G$,  
   we  have  a    persistent       $\Sigma_k$-equivariant   isometric  
    embeddings  of   filtered  metric  spaces  
   \begin{eqnarray} \label{eq-3.ace19}
   i_{\vec  G,k}(-,\infty):  && ({\rm  Conf}_k(V_{  G}, -,\infty),   (d_{  G})^k) \longrightarrow 
    ( {\rm  Conf}_k(V_{\vec  G}, -,\infty),   (d_{\vec  G})^k),\\  
     j_{\vec  G,k}(\frac{1}{2},-):  && ({\rm  Conf}_k(V_{\vec  G}, \frac{1}{2},-),   (d_{  \vec
     G})^k) \longrightarrow 
    ( {\rm  Conf}_k(V_{  G}, \frac{1}{2},-),   (d_{   G})^k)
    \label{eq-3.ace1999}
        \end{eqnarray}
   such  that 
   \begin{enumerate}[(1)]
   \item
    $ i_{\vec  G,k}(1/2,\infty)= j_{\vec  G,k}(1/2,\infty)^{-1}$  is  the  identity  map,
   \item
         $ i_{\vec  G,k}(n/2,\infty)$ 
    is  an  inclusion  from  the  configuration space  of  $G$  to the  configuration  space  of  $\vec  G$
    for  any  $2\leq  n<   \infty$,
    \item
        $ j_{\vec  G,k}(1/2,n/2)$ 
    is  an  inclusion  from  the  configuration space  of  $\vec  G$  to the  configuration  space  of  $ G$ for  any  $2\leq  n<   \infty$. 
        \end{enumerate}
   \end{proposition}
   
   \begin{proof}
    By  (\ref{eq--2.a1}),     
   we  obtain  
   \begin{eqnarray}\label{eq-zlaoq}
    {\rm  Conf}_k(V_{  G}, \frac{n}{2},\infty)  & \subseteq &  {\rm  Conf}_k(V_{\vec  G}, \frac{n}{2},\infty),\\
    {\rm  Conf}_k(V_{  G}, \frac{1}{2},\frac{n}{2})   &\supseteq&   {\rm  Conf}_k(V_{\vec  G}, \frac{1}{2},\frac{n}{2})
   \end{eqnarray}
   for  any  $1\leq  n\leq    \infty$.
   Hence $ i_{\vec  G,k}(n/2,\infty)$ 
   as  well  as   $ i_{\vec  G,k}(1/2,n/2)$ 
    is  an  inclusion  for  any  $1\leq  n\leq  \infty$.  
       For  any  $n_1\leq  n_2$,
    the      diagram     commutes 
      \begin{eqnarray}\label{eq-3.a8}
   \xymatrix{
   {\rm  Conf}_k(V_{  G}, \dfrac{n_2}{2},\infty)\ar[rr]^-{i_{\vec  G,k}( \frac{n_2}{2},\infty)}  \ar[d]&&   {\rm  Conf}_k(V_{ \vec  G}, \dfrac{n_2}{2},\infty)
   \ar[d]\\
      {\rm  Conf}_k(V_{  G}, \dfrac{n_1}{2},\infty)\ar[rr]^-{i_{\vec  G,k}( \frac{n_1}{2},\infty)}   &&   {\rm  Conf}_k(V_{ \vec  G}, \dfrac{n_1}{2},\infty)
   }
   \end{eqnarray}
   where all  the  maps  in  (\ref{eq-3.a8})  are  canonical  inclusions  thus  are  $\Sigma_k$-equivariant;   and 
    the      diagram     commutes 
      \begin{eqnarray}\label{eq-3.a98}
   \xymatrix{
   {\rm  Conf}_k(V_{  G}, \dfrac{1 }{2},\dfrac{n_1}{2})\ar[rr]^-{i_{\vec  G,k}( \frac{1 }{2},\frac{n_1}{2})}  \ar[d]&&   {\rm  Conf}_k(V_{ \vec  G}, \dfrac{1}{2},\dfrac{n_1}{2})
   \ar[d]\\
      {\rm  Conf}_k(V_{  G}, \dfrac{1 }{2},\dfrac{n_2}{2})\ar[rr]^-{i_{\vec  G,k}( \frac{1 }{2},\frac{n_2}{2})}   &&   {\rm  Conf}_k(V_{ \vec  G}, \dfrac{1 }{2},\dfrac{n_2}{2})
   }
   \end{eqnarray}
   where all  the  maps  in  (\ref{eq-3.a98})  are  canonical  inclusions thus  are  $\Sigma_k$-equivariant.   
  By  (\ref{eq-3.a8})   and  (\ref{eq-3.a98}),
     we  have        persistent   $\Sigma_k$-equivariant  isometric  
     embeddings  of   filtered  metric  spaces
   (\ref{eq-3.ace19})  and  (\ref{eq-3.ace1999})  respectively.  
   Moreover, 
   two  vertices  are  adjacent  in  $\vec  G$  iff  they  are  adjacent  in  $G$.  Thus  
   \begin{eqnarray}\label{eq-mbzao1a}
     {\rm  Conf}_k(V_{  G}, \frac{1}{2},\infty) = {\rm  Conf}_k(V_{\vec  G}, \frac{1}{2},\infty). 
   \end{eqnarray}
   Hence $ i_{\vec  G,k}(1/2,\infty)$,  which  is  the  inverse  of  $ j_{\vec  G,k}(1/2,\infty)$,
     is  the  identity  map.   
      \end{proof}
   
     \begin{corollary}
    \label{le-99.01}
     For  any  digraph  $\vec  G$  with   its  underlying graph  $G$, 
the      double-persistent  isometric      covering  map   $  \pi_{\vec  G, k}(-,-)$  in   (\ref{eq-co-0328-2})
and  the      double-persistent  isometric     covering  map  $\pi_{  G, k}(-,-)$  in  (\ref{eq-co-0328-22x})
satisfy   the  commutative     diagrams
    \begin{eqnarray}\label{eq-diag-9871}
    \xymatrix{
    {\rm  Conf}_k(V_{  G}, -,\infty) \ar[rr]^-{ i_{ \vec  G,k}(-,\infty)}  \ar[d]_-{ \pi_{  G, k}(-,\infty)} 
    && {\rm  Conf}_k(V_{\vec  G}, -,\infty) \ar[d]^-{ \pi_{\vec  G, k}(-,\infty)} \\
     {\rm  Conf}_k(V_{  G}, -,\infty)/\Sigma_k  \ar[rr]^-{ i_{\vec  G,k}(-,\infty)/\Sigma_k} 
      && {\rm  Conf}_k(V_{\vec  G}, -,\infty)/\Sigma_k, 
    }
    \\
     \xymatrix{
    {\rm  Conf}_k(V_{\vec  G}, \frac{1}{2},-) \ar[rr]^-{ j_{\vec   G,k}(\frac{1}{2},-)}  \ar[d]_-{ \pi_{\vec  G, k}(\frac{1}{2},-)} 
    && {\rm  Conf}_k(V_{    G}, \frac{1}{2},-) \ar[d]^-{ \pi_{   G, k}(\frac{1}{2},-)} \\
     {\rm  Conf}_k(V_{ \vec  G}, \frac{1}{2},-)/\Sigma_k  \ar[rr]^-{ j_{\vec  G,k}(\frac{1}{2},-)/\Sigma_k} 
      && {\rm  Conf}_k(V_{  G}, \frac{1}{2},-)/\Sigma_k.  
    }\label{eq-diag-98722}
    \end{eqnarray}
    \end{corollary}
    
    \begin{proof}
   With  the  help  of   Proposition~\ref{pr-3.a1},  
    the   persistent  $\Sigma_k$-equivariant 
     isometric     embedding   (\ref{eq-3.ace19})  
   induces  a     persistent   isometric   embedding 
   \begin{eqnarray}\label{eq-3.0bma12}
 i_{\vec  G,k}(-,\infty)/\Sigma_k:   {\rm  Conf}_k(V_{  G}, -,\infty)/\Sigma_k  \longrightarrow  
        {\rm  Conf}_k(V_{\vec  G}, -,\infty)/\Sigma_k 
   \end{eqnarray}
   such  that  the  diagram  (\ref{eq-diag-9871})  commutes.  
   Similarly,   
    the   persistent  $\Sigma_k$-equivariant 
     isometric     embedding   (\ref{eq-3.ace1999})  
   induces  a     persistent   isometric   embedding 
   \begin{eqnarray}\label{eq-3.0bma1222}
 j_{\vec  G,k}(\frac{1}{2},-)/\Sigma_k:   {\rm  Conf}_k(V_{\vec   G}, \frac{1}{2},-)/\Sigma_k  \longrightarrow  
        {\rm  Conf}_k(V_{   G}, \frac{1}{2},-)/\Sigma_k 
   \end{eqnarray}
   such  that  the  diagram  (\ref{eq-diag-98722})  commutes.  
    \end{proof}
      
   \begin{proposition}\label{pr-3.v1}
   For  any  digraph  $\vec  G$   with   its  underlying graph  $G$,  
   we  have  
    persistent  simplicial  embeddings  
     \begin{eqnarray}\label{eq-3.b79}
i_{\vec  G}(-,\infty):   &&  {\rm  Ind}( {  G},  -,\infty)  \longrightarrow  {\rm  Ind}( {\vec  G}, -,\infty),\\
j_{\vec  G}(\frac{1}{2},-):  &&   {\rm  Ind}( { \vec G},  \frac{1}{2},-)  \longrightarrow 
 {\rm  Ind}( {  G}, \frac{1}{2},-)
 \label{eq-3.b7999}
   \end{eqnarray}
   such  that 
   \begin{enumerate}[(1)]
   \item
   $ i_{\vec  G}(1/2,  \infty)=j_{\vec  G}(1/2,  \infty)^{-1}$   is  the  identity  map,
      \item
         $ i_{\vec  G }(n/2,\infty)$ 
    is  a  simplicial  inclusion  from  the independence complex  of  $G$  to the  independence complex
     of  $\vec  G$
    for  any  $2\leq  n<   \infty$,
    \item
        $ j_{\vec  G }(1/2,n/2)$ 
    is  a  simplicial   inclusion  from  the independence complex   of  $\vec  G$  to the independence complex   of  $ G$ for  any  $2\leq  n<   \infty$. 
      \end{enumerate}
       \end{proposition}
   
   \begin{proof}
   We  have   the      persistent  simplicial  embedding     
   (\ref{eq-3.b79})  sending  a  $(k-1)$-simplex  of ${\rm  Ind}( {  G},  -,\infty)$ 
   identically  
   to  a  $(k-1)$-simplex  of  ${\rm  Ind}( {\vec  G}, -,\infty)$   by   (\ref{eq-3.0bma12}).
   Similarly,  we    have  the      persistent  simplicial  embedding     
   (\ref{eq-3.b7999})  sending  a  $(k-1)$-simplex  of ${\rm  Ind}( {\vec  G},   {1}/{2},-)$ 
   identically  
   to  a  $(k-1)$-simplex  of  ${\rm  Ind}( {   G}, {1}/{2},-)$   by   (\ref{eq-3.0bma1222}).   
    By  (\ref{eq-mbzao1a}),  $ i_{\vec  G}(1/2,\infty)$    
   By   Proposition~\ref{pr-3.a1}~(1), 
    $i_{\vec  G}(1/2,  \infty)=j_{\vec  G}(1/2,  \infty)^{-1}$       is  the  identity  map.  
     \end{proof}

   \begin{corollary}
   For  any  digraph  $\vec  G$   with   its  underlying graph  $G$,  
   (\ref{eq-3.b79})  and   (\ref{eq-3.b7999}) induce    persistent  homomorphisms  of 
    persistent  homology  
     \begin{eqnarray}\label{eq-3.c79}
i_{\vec  G}(-,\infty)_*:  &&  H_*( {\rm  Ind}( {  G},  -,\infty)) \longrightarrow  H_*({\rm  Ind}( {\vec  G},-, \infty)),\\
  j_{\vec  G}(1/2,-)_*:  &&  H_*( {\rm  Ind}( {\vec  G},  1/2,-)) \longrightarrow 
   H_*({\rm  Ind}( { G},1/2,-))
   \label{eq-3.c7999}
   \end{eqnarray}
  respectively such  that  $i_{\vec  G}(1/2,  \infty)_*=j_{\vec  G}(1/2,  \infty)_*^{-1}$   is  the  identity  map.  
   \end{corollary}
   
   \begin{proof}
   Applying  the  simplicial  homology  functor  to  (\ref{eq-3.b79})  and  (\ref{eq-3.b7999}),  
   we  obtain     (\ref{eq-3.c79}) and  (\ref{eq-3.c7999})  respectively 
    such  that  $i_{\vec  G}(1/2,  \infty)_*=j_{\vec  G}(1/2,  \infty)_*^{-1}$  is  the  identity  map.  
   \end{proof}

Summarizing  Proposition~\ref{pr-3.a1}  and  Proposition~\ref{pr-3.v1},  
we  obtain  Theorem~\ref{th-main-intro-1}.

   \subsection{Configuration  spaces and  strong  totally  geodesic  embeddings}
   \label{ssa-3.3}

    \begin{proposition}\label{le-9.a1q}
    \begin{enumerate}[(1)]
    \item
     Let  $\varphi:  \vec  G\longrightarrow  \vec  G'$  be  a  strong   totally geodesic  immersion  
   of  digraphs  with radius  $m_0/2$.    
      Then   $\varphi$  induces  a  double-persistent     $\Sigma_k$-equivariant 
       isometric    embedding  
   of  double-filtered  metric  spaces 
   \begin{eqnarray}\label{eq-7a9-m1}
   \varphi_k(-,-):    ({\rm  Conf}_k(V_{\vec  G}, -,-),  (d_{\vec  G})^k) \longrightarrow   {\rm  Conf}_k(V_{\vec  G'}, -,-),  (d_{\vec  G'})^k) 
   \end{eqnarray}
   for  $1\leq  n<m\leq  m_0 $
    where  $n/2$  is  the  first parameter  and  $m/2$  is  the  second  parameter  
   in  the  double-persistence;
   \item
    Let  $\varphi:  \vec  G\longrightarrow  \vec  G'$  be  a  strong   totally geodesic  embedding   
   of  digraphs.    
      Then   $\varphi$  induces  a  double-persistent     $\Sigma_k$-equivariant  isometric  
        embedding  
   of  double-filtered  metric  spaces 
   (\ref{eq-7a9-m1})   
   for  $1\leq  n<m\leq\infty $ where  $n/2$  is  the  first parameter  and  $m/2$  is  the  second  parameter  
   in  the  double-persistence.  
   \end{enumerate}  
\end{proposition}

\begin{proof}
(1)    For  any $1\leq  n<m\leq  m_0$,  we  have  an  induced   $\Sigma_k$-equivariant    isometric 
embedding  
 \begin{eqnarray}\label{eq-7a9.5}
   \varphi_k(\frac{n}{2}, \frac{m}{2}):     {\rm  Conf}_k(V_{\vec  G}, \frac{n}{2}, \frac{m}{2})  \longrightarrow  
    {\rm  Conf}_k(V_{\vec  G'}, \frac{n}{2}, \frac{m}{2})
   \end{eqnarray}
given by 
\begin{eqnarray}\label{eq-25-ah}
 \varphi_k(\frac{n}{2}, \frac{m}{2})(v_1,\ldots,v_k)=(\varphi(v_1),\ldots,\varphi(v_k)).   
 \end{eqnarray}
Since  $\varphi:  \vec  G\longrightarrow  \vec  G'$  is  a  strong   totally geodesic  immersion  
   of  digraphs  with radius  $m_0/2$  and  $m\leq   m_0$,  
   we  have  
   \begin{eqnarray*}
   d_{\vec  G}(v_i,  v_j)  =  d_{\vec  G'}  (\varphi(v_i),  \varphi(v_j))
   \end{eqnarray*}
   in  (\ref{eq-25-ah})  
  for  any  $i\neq  j$. 
     Thus  (\ref{eq-7a9.5})  is  well-defined.   
  Moreover,    for  any   $(u_1,\ldots,u_k)$  and  $(v_1,\ldots,v_k)$   in
    $ {\rm  Conf}_k(V_{\vec  G}, n/2,m/2)$,  we  have  
  \begin{eqnarray*}
  (d_{\vec  G})^k((u_1,\ldots,u_k),(v_1,\ldots,v_k))= 
  (d_{\vec  G'})^k((\varphi(u_1),\ldots,\varphi(u_k)),
  (\varphi(v_1),\ldots,\varphi(v_k)) ).  
  \end{eqnarray*}  
  Thus  (\ref{eq-7a9.5})   is  an  isometry.    
  For any  $1\leq  n_1\leq  n_2<m\leq  m_0$,  the  diagram commutes  
  \begin{eqnarray}\label{diag-laz-1}
  \xymatrix{
  {\rm  Conf}_k(V_{\vec  G}, \dfrac{n_2}{2}, \dfrac{m }{2}) 
  \ar[rr]^-{\varphi_k(\frac{n_2}{2}, \frac{m }{2})}  
  \ar[d]
  &&  {\rm  Conf}_k(V_{\vec  G'}, \dfrac{n_2}{2}, \dfrac{m }{2}) 
  \ar[d]\\
   {\rm  Conf}_k(V_{\vec  G}, \dfrac{n_1}{2}, \dfrac{m }{2})  
   \ar[rr]^-{\varphi_k(\frac{n_1}{2}, \frac{m }{2})}
   && {\rm  Conf}_k(V_{\vec  G'}, \dfrac{n_1}{2}, \dfrac{m }{2})  
  }
  \end{eqnarray}
  and  for  any   $1\leq  n <m_1\leq  m_2\leq  m_0$,  the  diagram commutes 
    \begin{eqnarray}\label{diag-laz-2}
  \xymatrix{
  {\rm  Conf}_k(V_{\vec  G}, \dfrac{n }{2}, \dfrac{m_1 }{2}) 
  \ar[rr]^-{\varphi_k(\frac{n }{2}, \frac{m_1 }{2})}  
  \ar[d]
  &&  {\rm  Conf}_k(V_{\vec  G'}, \dfrac{n }{2}, \dfrac{m_1 }{2}) 
  \ar[d]\\
   {\rm  Conf}_k(V_{\vec  G}, \dfrac{n }{2}, \dfrac{m _2}{2})  
   \ar[rr]^-{\varphi_k(\frac{n }{2}, \frac{m_2 }{2})}
   && {\rm  Conf}_k(V_{\vec  G'}, \dfrac{n }{2}, \dfrac{m_2 }{2}).   
  }
  \end{eqnarray} 
  Here    in  (\ref{diag-laz-1})  and  (\ref{diag-laz-2}),  the vertical  maps  are  canonical  inclusions.  
  Hence   with  the  help  of  Lemma~\ref{le-2503a1z},  (\ref{diag-laz-1})  and  (\ref{diag-laz-2}),   
  we  can  take  the  double-persistence  of  $n/2$  and  $m/2$  in  (\ref{eq-7a9.5}).    
  We  obtain  the  double-persistent    $\Sigma_k$-equivariant   isometric      embedding  
   of  double-filtered  metric  spaces 
   (\ref{eq-7a9-m1})  for  $1\leq  n<m\leq  m_0$.

   (2)    The proof  of  (2)  is  an  analog  of  (1).  
\end{proof}

\begin{corollary}\label{co-bza2}
   \begin{enumerate}[(1)]
    \item
     Let  $\varphi:  \vec  G\longrightarrow  \vec  G'$  be  a  strong   totally geodesic  immersion  
   of  digraphs  with radius  $m_0/2$.    
      Then   $\varphi$  induces  a    double-persistent  isometric     embedding  
   of  double-filtered  metric  spaces 
   \begin{eqnarray}\label{eq-7b9-n1}
   \varphi_k(-,-)/\Sigma_k:    ({\rm  Conf}_k(V_{\vec  G}, -,-)/\Sigma_k,  (d_{\vec  G})^k/\Sigma_k) \longrightarrow   {\rm  Conf}_k(V_{\vec  G'}, -,-)/\Sigma_k,  (d_{\vec  G'})^k/\Sigma_k) 
   \end{eqnarray}
   for  $1\leq  n<m\leq  m_0$
    where  $n/2$    and  $m/2$  are the     parameters  
   in  the  double-persistence;
   \item
    Let  $\varphi:  \vec  G\longrightarrow  \vec  G'$  be  a  strong   totally geodesic  embedding   
   of  digraphs.    
      Then   $\varphi$  induces   a    double-persistent    isometric      embedding  
   of  double-filtered  metric  spaces 
   (\ref{eq-7b9-n1})   
   for  $1\leq  n<m $ where  $n/2$    and  $m/2$  are the     parameters  
   in  the  double-persistence.  
   \end{enumerate}  
\end{corollary}

\begin{proof}
(1)  
 For  any $1\leq  n<m\leq  m_0$,   the 
 $\Sigma_k$-equivariant   isometric   embedding   (\ref{eq-7a9.5})  induces   an   isometric  embedding  
 \begin{eqnarray}\label{eq-8a9.7}
   \varphi_k(\frac{n}{2}, \frac{m}{2})\Big /\Sigma_k:   
     {\rm  Conf}_k(V_{\vec  G}, \frac{n}{2}, \frac{m}{2})\Big /\Sigma_k  \longrightarrow  
    {\rm  Conf}_k(V_{\vec  G'}, \frac{n}{2}, \frac{m}{2})\Big /\Sigma_k
   \end{eqnarray}
given by 
\begin{eqnarray*} 
 \varphi_k(\frac{n}{2}, \frac{m}{2})(\{v_1,\ldots,v_k\})=\{\varphi(v_1),\ldots,\varphi(v_k)\}    
 \end{eqnarray*}
  such  that  
  \begin{eqnarray*}
  ((d_{\vec  G})^k/\Sigma_k)
  (\{u_1,\ldots,u_k\},\{v_1,\ldots,v_k\})= 
  ((d_{\vec  G'})^k/\Sigma_k)(\{\varphi(u_1),\ldots,\varphi(u_k)\},
  \{\varphi(v_1),\ldots,\varphi(v_k)\} ) 
  \end{eqnarray*}  
 for  any   $(u_1,\ldots,u_k)$  and  $(v_1,\ldots,v_k)$   in
    $ {\rm  Conf}_k(V_{\vec  G}, n/2,m/2)$.   
 Take   the  double-persistence  for  $n/2$  and  $m/2$   in  (\ref{eq-8a9.7})  with  $1\leq  n<m\leq  m_0$.      
 The  double-persistent      $\Sigma_k$-equivariant  isometric   embedding  (\ref{eq-7a9-m1})  induces  
 a       double-persistent       isometric    embedding  (\ref{eq-7b9-n1}).

(2)  Analogous  with  (1),  
the  proof of  (2)    follows  from Proposition~\ref{le-9.a1q}~(2).  
\end{proof}

\begin{corollary}\label{co-na-1}
  \begin{enumerate}[(1)]
    \item
     Let  $\varphi:  \vec  G\longrightarrow  \vec  G'$  be  a  strong   totally geodesic  immersion  
   of  digraphs  with radius  $m_0/2$.    
      Then   $\varphi$  induces  a  double-persistent   isometric    morphism  of      double-persistent   covering  maps
      \footnote[3]{
      The definition of  persistent  covering  map  is  introduced  by  the  present  author  in  
      \cite[Definition~4]{2025-reg}. }
      \begin{eqnarray}\label{eq-cvm5.1}
      \xymatrix{
      {\rm  Conf}_k(V_{\vec  G}, -,-) \ar[rr]^-{  \varphi_k(-,-)} \ar[d]_{\pi_{\vec  G,k}(-,-)}
        &&  {\rm  Conf}_k(V_{\vec  G'}, -,-)\ar[d]^{\pi_{\vec  G,k}(-,-)}\\
       ({\rm  Conf}_k(V_{\vec  G}, -,-)/\Sigma_k\ar[rr]^-{ \varphi_k(-,-)/\Sigma_k}
       &&{\rm  Conf}_k(V_{\vec  G'}, -,-)/\Sigma_k 
      }
      \end{eqnarray}
   for  $1\leq  n<m\leq  m_0$
    where  $n/2$    and  $m/2$  are the     parameters  
   in  the  double-persistence;
   \item
    Let  $\varphi:  \vec  G\longrightarrow  \vec  G'$  be  a  strong   totally geodesic  embedding   
   of  digraphs.    
      Then   $\varphi$  induces  a   double-persistent   isometric morphism  of      double-persistent   covering  maps 
   (\ref{eq-cvm5.1})   
   for  $1\leq  n<m\leq \infty $  where  $n/2$    and  $m/2$  are the     parameters  
   in  the  double-persistence.  
   \end{enumerate}  
\end{corollary}

\begin{proof}
(1)   The  proof  of  (1)  follows  from  Corollary~\ref{co-250328-cv},  
 Proposition~\ref{le-9.a1q}~(1)  and  Corollary~\ref{co-bza2}~(1).  
 
 (2)
 The  proof  of  (2)  follows  from  Corollary~\ref{co-250328-cv},  
 Proposition~\ref{le-9.a1q}~(2)  and  Corollary~\ref{co-bza2}~(2). 
\end{proof}

\begin{corollary}\label{co-au-1}
   \begin{enumerate}[(1)]
    \item
     Let  $\varphi:  \vec  G\longrightarrow  \vec  G'$  be  a  strong   totally geodesic  immersion  
   of  digraphs  with radius  $m_0/2$.    
      Then   $\varphi$  induces  a      double-persistent  embedding  
   of  double-filtered  simplicial  complexes 
   \begin{eqnarray}\label{eq-7c9-l1}
   \varphi(-,-):    {\rm  Ind}({\vec  G}, -,-)  \longrightarrow   {\rm   Ind} ({\vec  G'}, -,-)  
   \end{eqnarray}
   for  $1\leq  n<m\leq  m_0$
     where  $n/2$    and  $m/2$  are the     parameters  
   in  the  double-persistence;
   \item
    Let  $\varphi:  \vec  G\longrightarrow  \vec  G'$  be  a  strong   totally geodesic  embedding   
   of  digraphs.    
      Then   $\varphi$  induces  an    double-persistent   embedding  of   double-filtered 
      simplicial  complexes 
   (\ref{eq-7c9-l1})   
   for  $1\leq  n<m\leq\infty $  where  $n/2$    and  $m/2$  are the     parameters  
   in  the  double-persistence.  
   \end{enumerate}  
\end{corollary}

\begin{proof}
(1)   By  Corollary~\ref{co-bza2}~(1),  for  $1\leq  n<m\leq  m_0$,  
we  have  a  double-persistent  embedding  of  simplicial  simplicial  complexes 
   (\ref{eq-7c9-l1})  sending  any  vertex  $v$    
   to      $\varphi(v)$.

    (2)   Analogous  with (1),   
    the  double-persistent  embedding   
    follows  by  Corollary~\ref{co-bza2}~(2).     
    \end{proof}

Substituting  the  strong  totally  geodesic  embeddings  of  digraphs  
 with  strong   totally  geodesic  embeddings  of  graphs,  analogs  of  
 Proposition~\ref{le-9.a1q}  and  Corollaries~\ref{co-bza2},  \ref{co-na-1},  \ref{co-au-1}
 can  be  obtained.  
      
    \begin{proposition}\label{le-9.au-1q}
     Let  $\varphi:    G\longrightarrow     G'$  be  a  strong   totally geodesic  immersion  
   of   graphs  with radius  $m_0/2$  (resp.   a  strong   totally geodesic  embedding   
   of   graphs).    
      Then   $\varphi$  induces  a  double-persistent    $\Sigma_k$-equivariant  
       isometric    embedding  
   of  double-filtered  metric  spaces 
   \begin{eqnarray*} 
   \varphi_k(-,-):    ({\rm  Conf}_k(V_{ G}, -,-),  (d_{ G})^k) \longrightarrow   {\rm  Conf}_k(V_{  G'}, -,-),  (d_{  G'})^k) 
   \end{eqnarray*}
   for  $1\leq  n<m\leq  m_0$  (resp.   for  $1\leq  n<m \leq\infty$)
    where  $n/2$  is  the  first parameter  and  $m/2$  is  the  second  parameter  
   in  the  double-persistence.   
\end{proposition}

\begin{proof}
The  proof  is  analogous  with  Proposition~\ref{le-9.a1q}.
\end{proof}

\begin{corollary}\label{co-au-bza2}
     Let  $\varphi:   G\longrightarrow    G'$  be  a  strong   totally geodesic  immersion  
   of   graphs  with radius  $m_0/2$  (resp.  a  strong   totally geodesic  embedding   
   of   graphs).    
      Then   $\varphi$  induces  a  double-persistent   isometric      embedding  
   of  double-filtered  metric  spaces 
   \begin{eqnarray*} 
   \varphi_k(-,-)/\Sigma_k:    ({\rm  Conf}_k(V_{   G}, -,-)/\Sigma_k,  (d_{  G})^k/\Sigma_k) \longrightarrow   {\rm  Conf}_k(V_{   G'}, -,-)/\Sigma_k,  (d_{ G'})^k/\Sigma_k) 
   \end{eqnarray*}
   for  $1\leq  n<m\leq  m_0$   (resp.   for  $1\leq  n<m\leq\infty $)
    where  $n/2$    and  $m/2$  are the     parameters  
   in  the  double-persistence.   
\end{corollary}

\begin{proof}
The  proof  is  analogous  with Corollary~\ref{co-bza2}.  
\end{proof}

\begin{corollary}\label{co-au-na-1}
         Let  $\varphi:   G\longrightarrow    G'$  be  a  strong   totally geodesic  immersion  
   of   graphs  with radius  $m_0/2$  (resp.  a  strong   totally geodesic  embedding   
   of   graphs).    
      Then   $\varphi$  induces  a  double-persistent  isometric   morphism  of      double-persistent   covering  maps
      \begin{eqnarray*} 
      \xymatrix{
      {\rm  Conf}_k(V_{   G}, -,-) \ar[rr]^-{  \varphi_k(-,-)} \ar[d]_{\pi_{   G,k}(-,-)}
        &&  {\rm  Conf}_k(V_{\vec  G'}, -,-)\ar[d]^{\pi_{ G,k}(-,-)}\\
       ({\rm  Conf}_k(V_{   G}, -,-)/\Sigma_k\ar[rr]^-{ \varphi_k(-,-)/\Sigma_k}
       &&{\rm  Conf}_k(V_{   G'}, -,-)/\Sigma_k 
      }
      \end{eqnarray*}
   for  $1\leq  n<m\leq  m_0$  (resp.   for  $1\leq  n<m\leq\infty $)
    where  $n/2$    and  $m/2$  are the     parameters  
   in  the  double-persistence.
    \end{corollary}
    
    \begin{proof}
The  proof  is  analogous  with Corollary~\ref{co-na-1}.  
\end{proof}

\begin{corollary}\label{co-au-2}
     Let  $\varphi:       G\longrightarrow     G'$  be  a  strong   totally geodesic  immersion  
   of   graphs  with radius  $m_0/2$  (resp.  a  strong   totally geodesic  embedding   
   of   graphs).      
      Then   $\varphi$  induces  a       double-persistent  embedding  
   of  double-filtered  simplicial  complexes 
   \begin{eqnarray*} 
   \varphi(-,-):    {\rm  Ind}({   G}, -,-)  \longrightarrow   {\rm   Ind} ({ G'}, -,-)  
   \end{eqnarray*}
   for  $1\leq  n<m\leq  m_0$   (resp.   for  $1\leq  n<m \leq\infty$)
     where  $n/2$    and  $m/2$  are the     parameters  
   in  the  double-persistence.
\end{corollary}

    \begin{proof}
The  proof  is  analogous  with Corollary~\ref{co-au-1}.  
\end{proof}

Summarizing  Proposition~\ref{le-9.a1q},  Corollary~\ref{co-au-1}, 
 Proposition~\ref{le-9.au-1q}  and  Corollary~\ref{co-au-2},  
we  obtain  Theorem~\ref{th-main-intro-2}.

  \subsection{Geometric  realizations  of  the  independence  complexes}\label{s-7}

 Let  $X $  be  a  topological  space.  
 An  embedding  $f:  X\longrightarrow   \mathbb{R}^N$  is  called  
 {\it  affinely  $k$-regular}  if  for  any  distinct  $k$-points  $x_1,\ldots,  x_k\in  X$,  
 their  images  $f(x_1)$,  $\ldots$,  $f(x_k)$  are  affinely  independent
 (cf.  \cite{regular}).  
 Let  $d:  X\times  X\longrightarrow  [0, +\infty]$  be  a  metric  on  $X$.  
 For  any  $0\leq r<s\leq  +\infty$, 
 we  say  that  an  embedding  $f:  X\longrightarrow   \mathbb{R}^N$
 is  {\it  affinely  $k$-regular}  with respect  to  $(r,s]$
  if   for  any  distinct  $k$-points  $x_1,\ldots,  x_k\in  X$ such  that  
  $2r< d(x_i,x_j)\leq  2s$  where  $1\leq  i<  j\leq  k$,
 their  images  $f(x_1)$,  $\ldots$,  $f(x_k)$  are  affinely  independent.  
   
   \begin{proposition}\label{pr-geo-1}
  For  any  graph  $\vec  G$  with  its  underlying  graph  $G$
  and  any  $1\leq  n<m\leq  \infty$,  
 \begin{enumerate}[(1)]
 \item
 there  exists an   affinely  $k$-regular  embedding  
 $f:  (V_{\vec  G},d_{\vec  G})\longrightarrow  \mathbb{R}^N$  
   with respect  to  $(n/2,m/2]$ 
 if  and  only  if   ${\rm  sk}^{k-1}{\rm  Ind}(\vec  G, n/2,m/2)$,
 the  $(k-1)$-skeleton  of  ${\rm  Ind}(\vec  G, n/2,m/2)$, 
  has  a  geometric  realization  in   $ \mathbb{R}^N$;     
 
 \item
  there  exists an   affinely  $k$-regular  embedding  
 $f:  (V_{ G},d_{ G})\longrightarrow  \mathbb{R}^N$  
   with respect  to  $(n/2,m/2]$ 
 if  and  only  if   ${\rm  sk}^{k-1}{\rm  Ind}( G, n/2,m/2)$,
 the  $(k-1)$-skeleton  of  ${\rm  Ind}(  G, n/2,m/2)$, 
  has  a  geometric  realization  in   $ \mathbb{R}^N$. 
 \end{enumerate}
 \end{proposition}
 
 \begin{proof}
 (1)  Suppose  $f:  (V_{\vec  G},d_{\vec G})\longrightarrow  \mathbb{R}^N$
 is   an   affinely  $k$-regular  embedding
  with respect  to  $(n/2,m/2]$.  
  Then  for  any  simplex  $\sigma\in  {\rm  sk}^{k-1}{\rm  Ind}(\vec  G, n/2,m/2)$, 
  the  image  of  its  vertices   $f(\sigma)=\{f(v)\mid  v\in \sigma\}$  are  affinely  independent. 
  Thus  $f$  induces  a   geometric  realization  of  ${\rm  sk}^{k-1}{\rm  Ind}(\vec  G, n/2,m/2)$.

   Conversely,  suppose   ${\rm  sk}^{k-1}{\rm  Ind}( \vec  G, n/2,m/2)$ has  a  
    geometric  realization  in   $ \mathbb{R}^N$.  
    Note  that  the  vertex  set  of  ${\rm  sk}^{k-1}{\rm  Ind}( \vec  G, n/2,m/2)$  
    is   $V_{\vec  G}$.  
    Thus  there  exists  an  embedding  
    $f:  V_{\vec  G}\longrightarrow   \mathbb{R}^N$  such  that  
    for  any  distinct  $k$-vertices  $v_1,\ldots,  v_k\in  V_{\vec  G}$,  
    if  they  span a  simplex  in  ${\rm  sk}^{k-1}{\rm  Ind}( \vec  G, n/2,m/2)$,  
    then  their  images  $f(v_1),\ldots,  f(v_k)$  are  affinely  independent.  
    Note  that  $v_1,\ldots,  v_k\in  V_{\vec  G}$ 
      span a  simplex  in  ${\rm  sk}^{k-1}{\rm  Ind}( \vec  G, n/2,m/2)$
      if  and  only  if  $n< d_{\vec   G}(x_i,x_j)\leq  m$  for  any  $1\leq  i<  j\leq  k$.  
      Thus  $f$  
 is   an   affinely  $k$-regular  embedding
  with respect  to  $(n/2,m/2]$.

 (2)      The  proof  of  (2)  is  an  analog  of  (1).  
 \end{proof}
 
\begin{corollary}\label{co-bofij-1}
  For  any  graph  $\vec  G$  with  its  underlying  graph  $G$
  and  any  $1\leq  n<m\leq  \infty$,  
 \begin{enumerate}[(1)]
 \item
         an   affinely  $k$-regular  embedding  
 $f:  (V_{ G},d_{ G})\longrightarrow  \mathbb{R}^N$  
   with respect  to  $(n/2, m/2]$  induces  
    an   affinely  $k$-regular  embedding  
 $f:  (V_{\vec  G},d_{\vec  G})\longrightarrow  \mathbb{R}^N$  
   with respect  to  $(n/2, m/2]$;
   \item
     an   affinely  $k$-regular  embedding  
 $f:  (V_{\vec  G},d_{\vec  G})\longrightarrow  \mathbb{R}^N$  
   with respect  to  $(n/2, m/2]$  induces  an   affinely  $k$-regular  embedding  
 $f:  (V_{ G},d_{ G})\longrightarrow  \mathbb{R}^N$
 with respect  to  $(n/2, m/2]$.    
\end{enumerate}
\end{corollary}

\begin{proof}
The  corollary  follows  from Proposition~\ref{pr-3.v1}  and  
Proposition~\ref{pr-geo-1}.  
\end{proof}

  \begin{corollary}
 Let  $\varphi:         \vec  G\longrightarrow       \vec  G'$  be  a  strong   totally geodesic  immersion  
   of   digraphs  with radius  $m_0/2$  (resp.  a  strong   totally geodesic  embedding   
   of   digraphs).     Then  
    for  $1\leq  n<m\leq  m_0$   (resp.   for  $1\leq  n<m $), 
    an  affinely   $k$-regular  embedding  $f:  (V_{\vec  G'},  d_{\vec  G'})\longrightarrow 
     \mathbb{R}^N$
    with respect  to  $(n/2,m/2]$
    induces 
     an  affinely   $k$-regular  embedding 
      $f:  (V_{\vec  G},  d_{\vec  G})\longrightarrow  \mathbb{R}^N$ 
      with respect  to  $(n/2,m/2]$.  
\end{corollary}

\begin{proof}
The  corollary  follows  from  Corollary~\ref{co-au-1}.  
\end{proof}

\begin{corollary}
 Let  $\varphi:         G\longrightarrow        G'$  be  a  strong   totally geodesic  immersion  
   of   graphs  with radius  $m_0/2$  (resp.  a  strong   totally geodesic  embedding   
   of   graphs).     Then  
    for  $1\leq  n<m\leq  m_0$   (resp.   for  $1\leq  n<m $), 
    an  affinely   $k$-regular  embedding  $f:  (V_{G'},  d_{G'})\longrightarrow  \mathbb{R}^N$
    with respect  to  $(n/2,m/2]$
    induces 
     an  affinely   $k$-regular  embedding  $f:  (V_{G},  d_{G})\longrightarrow  \mathbb{R}^N$ 
      with respect  to  $(n/2,m/2]$.  
\end{corollary}

\begin{proof}
The  corollary  follows  from  Corollary~\ref{co-au-2}.  
\end{proof}

\section{Path   independence  complexes  for  digraphs  and  their  associated  chain  complexes }
\label{ssa-4}

In this  section,  we  consider the  directions  on the arcs  and  
introduce   the  path  independence  complexes  for  digraphs.  
Then  we  apply  the  infimum  chain  complex  and  the  supremum  chain  complex  
to  the  path  independence  complexes  to  give  associated  chain  complexes  for  the  
path  independence  complexes.  
In   Subsection~\ref{ssa-4.1},  we  introduce  the  path  configuration  spaces  and  path  independence  complexes.  
In   Subsection~\ref{ss2.1},  
we briefly  review  the  infimum  and  the   supremum  chain  complexes.  
In  Subsection~\ref{ssa-4.3},  
we  prove   canonical  embeddings  from  the  infimum  and  the  supremum   chain  complexes 
of  the  path  independence  complex  of  the  underlying graph  into   the  infimum  and  the  supremum   chain  complexes 
of  the  path  independence  complex  of  the  digraph.  
Moreover,  for  strong  totally geodesic  embeddings  of  (di)graphs,  
we prove  induced  monomorphisms  between  the   infimum  and  the  supremum   chain  complexes.

       \subsection{The  path  configuration  spaces  and  the  path  independence  complexes }
\label{ssa-4.1}

    Let  $\vec  G$     be  a    digraph. 
    Let   $k\geq  1$  and   let  $1\leq  n<m\leq  \infty$.  
   
     \begin{definition}\label{def-250326}
     We  define  
   an  {\it  independent     elementary  $k$-path}  on  $\vec  G$     with  constraint  interval  
   $(n/2,  m/2]$
  to  be       an   elementary     $k$-path  $v_0v_1\ldots  v_k$   on  $V_{\vec  G}$   
   such  that  
   $n<d_{\vec G}(v_{i-1}, v_i) \leq  m$    for  any  $1\leq  i \leq  k$.  
   Equivalently,  an   independent   elementary  $k$-path   on  $\vec  G$     
     with  constraint  interval  
   $(n/2,  m/2]$  
   is   an  allowed  elementary  $k$-path  on  the  $1$-skeleton  
   ${\rm  sk}^1({\rm  Ind} (\vec G,  n/2,  m/2))$
    of  ${\rm  Ind} (\vec G,  n/2,  m/2)$.   
     \end{definition}
     
       \begin{definition}\label{def-path-conf-12}
   We  define    the  $k$-th   {\it  ordered     path  configuration  space}   of  $\vec  G$  with  
   constraint  interval  $(n/2,  m/2]$  to  be  the  metric space 
    \begin{eqnarray}\label{eq-99.a1}
\overrightarrow{ {\rm    Conf}}_k(V_{\vec  G}, \frac{n}{2},  \frac{m}{2})=
  \{(v_1,\ldots, v_k)\in  (V_{\vec G})^k\mid
     n< d_{\vec  G}(v_i,v_{i+1}) \leq  m  {\rm~for~any~}   1\leq  i\leq  k-1 \}        
  \end{eqnarray} 
   consisting  of  all  the  independent     elementary  $(k-1)$-paths  on  $\vec  G$     with  constraint  interval  
   $(n/2,  m/2]$,  with  the  product  metric  $(d_{\vec  G})^k$.    
   \end{definition}

    \begin{lemma}\label{le-25apr1}
  For  any  digraph  $\vec  G$  and  any  positive  integer  $k$,    
  we have  a  $\mathbb{Z}_2$-equivariant    double-filtration   
    \begin{eqnarray}\label{eq-apr-e.1}
    \overrightarrow {{\rm  Conf}}_k(V_{\vec  G}, -,-)=
    \Big\{ \overrightarrow{\rm  Conf}_k(V_{\vec  G}, \frac{n}{2}, \frac{m}{2})
    ~\Big|~  1\leq  n<m\leq  \infty \Big\}
    \end{eqnarray}
       such  that 
    \begin{eqnarray*}
     \overrightarrow{\rm  Conf}_k(V_{\vec  G}, \frac{n_1}{2}, \frac{m}{2})  \supseteq 
      \overrightarrow{\rm  Conf}_k(V_{\vec  G}, \frac{n_2}{2}, \frac{m}{2})
    \end{eqnarray*}
   is  an   isometric  embedding  for  any  $n_1<n_2<m$
    and 
        \begin{eqnarray*}
     \overrightarrow{\rm  Conf}_k(V_{\vec  G}, \frac{n }{2}, \frac{m_1}{2})  \subseteq 
     \overrightarrow{\rm  Conf}_k(V_{\vec  G}, \frac{n }{2}, \frac{m_2}{2})
    \end{eqnarray*}
      is  an   isometric  embedding  for  any  $n<m_1<m_2$. 
    \end{lemma}
    
    \begin{proof}
    Let   $\mathbb{Z}_2$  act  on  $  \overrightarrow {{\rm  Conf}}_k(V_{\vec  G}, -,-)$  
    such that the  nontrivial  element  in  $\mathbb{Z}_2$ sends  each  $(k-1)$-path  
    $v_0v_1\ldots  v_{k-1}$  to  its  inverse  $v_{k-1}v_{k-2}\ldots  v_0$.   
    The  $\mathbb{Z}_2$-action  is  isometric  with  respect to  $(d_{\vec  G})^k$.  
    The   double-filtration   of  $  \overrightarrow {{\rm  Conf}}_k(V_{\vec  G}, -,-)$  is  
    $\mathbb{Z}_2$-equivariant.  
    \end{proof}

    \begin{definition}\label{def-25apr-5}
     We  define   the  {\it  path  independence  complex}        of  $\vec  G$ 
       with  
   constraint  interval  $(n/2,  m/2]$  
    to  be  the   union   
    \footnote[4]{   
  In  general,  (\ref{eq-apr-7.a1})  may  not  be  a  simplicial  complex.  
 }  
  \begin{eqnarray}\label{eq-apr-7.a1}
 \overrightarrow{ {\rm  Ind}} (\vec  G,  \frac{n}{2},\frac{m}{2})=  \bigcup_{k\geq  1}  
  \overrightarrow{ {\rm    Conf}}_k(V_{\vec  G}, \frac{n}{2},\frac{m}{2} 
   ).     
  \end{eqnarray}
  We  define  an  {\it  automorphism}  of  (\ref{eq-apr-7.a1})  
  to  be  a  self-bijection   $\varphi$   of   $V_{\vec  G}$  such  that 
  for
  any   path  
    $v_0v_1\ldots  v_{k}$  in  $\overrightarrow{ {\rm  Ind}} (\vec  G,   {n}/{2}, {m}/{2}) $,  
    its  image 
     $\varphi(v_0)\varphi(v_1)\ldots  \varphi(v_{k})$
    is   still  a  path  in   $\overrightarrow{ {\rm  Ind}} (\vec  G, {n}/{2}, {m}/{2}) $.  
    We  define  
    the  automorphism  group  ${\rm  Aut}(\overrightarrow{ {\rm  Ind}} (\vec  G,   {n}/{2}, {m}/{2}) )$ 
    to  be  the  group  of  all  the  automorphisms  of  (\ref{eq-apr-7.a1}).   
  \end{definition}

  \begin{corollary}\label{co-25apr-2}
   For  any  digraph  $\vec  G$,    
  we have  a  $\mathbb{Z}_2$-equivariant    double-filtration 
    \begin{eqnarray}
    \label{eq-apr-e.2}
    \overrightarrow {{\rm   Ind}} ({\vec  G}, -,-)=
    \Big\{ \overrightarrow{\rm   Ind} ({\vec  G}, \frac{n}{2}, \frac{m}{2})
    ~\Big|~  1\leq  n<m\leq  \infty \Big\}
  \end{eqnarray}
    such  that 
    \begin{eqnarray*}
     \overrightarrow{\rm   Ind} ( {\vec  G}, \frac{n_1}{2}, \frac{m}{2})  \supseteq 
      \overrightarrow{\rm   Ind} ( {\vec  G}, \frac{n_2}{2}, \frac{m}{2})
    \end{eqnarray*}
   for  any  $n_1<n_2<m$
    and 
        \begin{eqnarray*}
     \overrightarrow{\rm   Ind}_k( {\vec  G}, \frac{n }{2}, \frac{m_1}{2})  \subseteq 
     \overrightarrow{\rm   Ind}_k( {\vec  G}, \frac{n }{2}, \frac{m_2}{2})
    \end{eqnarray*}
     for  any  $n<m_1<m_2$. 
  \end{corollary}
  
  \begin{proof}
  The  proof  follows  from  Lemma~\ref{le-25apr1}   and  Definition~\ref{def-25apr-5}.  
  \end{proof}

  Let  $G$  be  a  graph.  
  Similar  with  Definition~\ref{def-250326},  
  we  define  
   an  {\it  independent     elementary  $k$-path}  on  $  G$     with  constraint  interval  
   $(n/2,  m/2]$
  to  be       an   elementary     $k$-path  $v_0v_1\ldots  v_k$   on  $V_{  G}$   
   such  that  
   $n<d_{ G}(v_{i-1}, v_i) \leq  m$    for  any  $1\leq  i \leq  k$.   
  Similar  with  Definition~\ref{def-path-conf-12},  
   we  define    the  $k$-th   {\it  ordered     path  configuration  space}   of  $  G$  with  
   constraint  interval  $(n/2,  m/2]$  by  
    \begin{eqnarray} \label{eq-250902-1}
\overrightarrow{ {\rm    Conf}}_k(V_{  G}, \frac{n}{2},  \frac{m}{2})=
  \{(v_1,\ldots, v_k)\in  (V_{\vec G})^k\mid
     n< d_{ G}(v_i,v_{i+1}) \leq  m  {\rm~for~any~}   1\leq  i\leq  k-1 \}.         
  \end{eqnarray} 
  We  have  a   $\mathbb{Z}_2$-equivariant  double-filtration  $\overrightarrow {{\rm  Conf}}_k(V_{   G}, -,-)$  of  (\ref{eq-250902-1}).  
  Similar  with  Definition~\ref{def-25apr-5},  
  we  define   the  {\it  path  independence  complex}        of  $  G$ 
       with  
   constraint  interval  $(n/2,  m/2]$  
    by   
  \begin{eqnarray} \label{eq-may-7.a2}
 \overrightarrow{ {\rm  Ind}} ( G,  \frac{n}{2},\frac{m}{2})=  \bigcup_{k\geq  1}  
  \overrightarrow{ {\rm    Conf}}_k(V_{   G}, \frac{n}{2},\frac{m}{2}.   
   ).     
  \end{eqnarray}
We  have  
a   $\mathbb{Z}_2$-equivariant  
double-filtration  $\overrightarrow {{\rm   Ind}} ( {   G}, -,-)$  of  (\ref{eq-may-7.a2}).
 We  define  an  {\it  automorphism}  of  (\ref{eq-may-7.a2})  
  to  be  a  self-bijection   $\varphi$   of   $V_{   G}$  such  that 
  for
  any   path  
    $v_0v_1\ldots  v_{k}$  in  $\overrightarrow{ {\rm  Ind}} (  G,   {n}/{2}, {m}/{2}) $,  
    its  image 
     $\varphi(v_0)\varphi(v_1)\ldots  \varphi(v_{k})$
    is   still  a  path  in   $\overrightarrow{ {\rm  Ind}} ( G, {n}/{2}, {m}/{2}) $.  
    We  define  
    the  automorphism  group  ${\rm  Aut}(\overrightarrow{ {\rm  Ind}} (  G,   {n}/{2}, {m}/{2}) )$ 
    to  be  the  group  of  all  the  automorphisms  of  (\ref{eq-may-7.a2}).

\begin{corollary}\label{pr-05-23-7}
For  any  digraph  $\vec  G$   with  its   underlying  graph  $G$   and  any 
   $1\leq  n<m \leq  \infty$,  we  have    canonical   group  homomorphisms 
\begin{eqnarray}\label{eq-250522-b1}
\beta_{\vec  G}(\frac{n}{2},\frac{m}{2}): &&  {\rm  Aut}(\vec   G)\longrightarrow    {\rm  Aut}( \overrightarrow{\rm Ind}(\vec  G,\frac{n}{2},\frac{m}{2})), 
 \\
\beta_{ G}(\frac{n}{2},\frac{m}{2}):  && {\rm  Aut}(   G)\longrightarrow    
{\rm  Aut}( \overrightarrow{\rm Ind}(   G,\frac{n}{2},\frac{m}{2})).  
\label{eq-250522-b2}
\end{eqnarray}
\end{corollary}

\begin{proof}
The  proof  is     analogous  with  Corollary~\ref{pr-05-05-22-1}. 
\end{proof}

  \begin{example}
\label{ex-97.a0}
   Consider  the  line  digraph  $\vec  L$  (cf.  Example~\ref{ex-3.3}) 
    and  the  zigzag  digraph  $\vec  Z$  (cf.  Example~\ref{ex-3.5})  whose 
      underlying  graphs  are   the  line  graph  
    $L$  (cf.  Example~\ref{ex-0.8}).   
      The   path  independence  complex  of  $\vec  L$  is    
      \begin{eqnarray*}
     \overrightarrow  {\rm  Ind}(\vec  L, \frac{n}{2},\frac{m}{2}) 
      =  \bigcup_{k=0}^\infty 
      \big\{(v_{i_0},v_{i_1}, \ldots, v_{i_k})\mid
         n <i_{j}-i_{j-1}\leq  m{\rm~for~any~}   
      1\leq  j\leq  k \big\} 
      \end{eqnarray*}
      such  that  
      \begin{eqnarray*}
      \overrightarrow {\rm  Ind}(\vec  L, \frac{n}{2},\infty) 
      =\bigcup_{n<m<\infty }{\rm  Ind}(\vec  L, \frac{n}{2},\frac{m}{2}).   
      \end{eqnarray*}
      The   path  independence  complex  of  $\vec  Z$  is     
      \begin{eqnarray*}
       \overrightarrow  {\rm  Ind}(\vec  Z, \frac{n}{2}, \frac{m}{2})= \emptyset 
      \end{eqnarray*}
      for  any  $n<m<\infty$  
      and 
      \begin{eqnarray*}
     \overrightarrow  {\rm  Ind}(\vec  Z, \frac{n}{2},\infty)=
        \bigcup_{k=0}^\infty 
        \big\{(v_{i_0},v_{i_1}, \ldots, v_{i_k})\mid
           1 <i_{j}-i_{j-1}{\rm~for~any~}   
      1\leq  j\leq  k \big\},   
      \end{eqnarray*}
      which  does  not  depend   on  the  choice  of  $n\geq  1$.  
      The  path   independence  complex  of  $L$  is     
      \begin{eqnarray*}
       \overrightarrow{\rm  Ind}( L, \frac{n}{2},\frac{m}{2})=   \overrightarrow {\rm  Ind}(\vec  L, \frac{n}{2},\frac{m}{2}).      
      \end{eqnarray*}
\end{example}
  
  \begin{example}
\label{ex-97.a1}
Let  $\mathbb{Z}^l$  be   the  lattice   
$\{\vec  z=  (z_1,\ldots,z_l)\mid  z_1,\ldots,z_l\in \mathbb{Z}\}$  in  $\mathbb{R}^l$.   
Let  $\vec  L^l$  be  the  digraph  with   vertices    $  \mathbb{Z}^l$ 
and  arcs  
\begin{eqnarray*}
(z_1,\ldots,z_l)\to (z_1,\ldots, z_i+1, \ldots, z_l),~~~~~~1\leq  i\leq  l 
\end{eqnarray*}
  for any  
   $\vec  z= (z_1,\ldots,z_l)$  in   $  \mathbb{Z}^l$.   
 The  underlying  graph  $L^l$  of  $\vec  L^l$  
 is  the  graph   with   vertices    $  \mathbb{Z}^l$ 
and  edges    
 \begin{eqnarray*}
\{(z_1,\ldots,z_l), (z_1,\ldots, z_i+1, \ldots, z_l)\},~~~~~~1\leq  i\leq  l 
\end{eqnarray*}
for any  
   $\vec  z= (z_1,\ldots,z_l)$  in   $  \mathbb{Z}^l$.  
   We  have  $d_{\vec  L^l}  =  d_{L^l}$  given by  
   \begin{eqnarray*}
   d_{\vec  L^l}(\vec  z,  \vec { z'})  =  d_{ L^l}(\vec  z,  \vec { z'})  =
    \sum_{i=1}^l  |z_i-z'_i|,     ~~~~~~\vec  z,  \vec { z'} \in  \mathbb{Z}^l.  
   \end{eqnarray*}
    The  path  independence  complexes  of  $\vec  L^l$   and  $L^l$  are     equal
   \begin{eqnarray*}
    \overrightarrow  {\rm  Ind}(\vec  L^l, \frac{n}{2},  \frac{m}{2}) =
        \overrightarrow{\rm  Ind}(   L^l,  \frac{n}{2},  \frac{m}{2}) 
        \end{eqnarray*}
    which  are  given  by 
      \begin{eqnarray*}
      \bigcup_{k=0}^\infty \big\{    (\vec  z(0),\cdots, \vec z(k))\mid 
       n<   \sum_{i=1}^l   |z(j)_i -z(j-1)_i| \leq  m {\rm~for~any~} 1\leq  j\leq  k \big\}  
      \end{eqnarray*}
      with  $\vec z(j)= (z(j)_1,\ldots, z(j)_l)$  in  $\mathbb{Z}^l$  for  $0\leq  j\leq  k$.   
      For  any  $0\leq  t\leq  l$,  we  have  strong  totally  geodesic  embeddings  of  (di)graphs  
 \begin{eqnarray*} 
 \varphi^t:  && \vec  L^m\longrightarrow  \vec  L^{m+1}, \\
                 &&  L^m\longrightarrow    L^{m+1}  
 \end{eqnarray*}
 sending  $(z_1,\ldots,z_l)$  to $(z_1,\ldots,z_{t}, 0,  z_{t+1}, \ldots, z_l)$.  
This  induces  double-persistent  embeddings       
 \begin{eqnarray*}
 \varphi^t(-,-):      {\rm  Ind}(\vec  L^m, -,-) \longrightarrow     {\rm  Ind}(\vec  L^{m+1}, -,-).  
 \end{eqnarray*}
\end{example}

\begin{example}
\label{ex-97.a2}
  Consider  the  cyclic  digraph  $\vec  C_r$ 
   with  its  underlying  graph  $C_r$   (cf.  Example~\ref{ex-3.3}).  
    The   path  independence  complexes  of  $\vec  C_r$  and  $C_r$  are  equal
     \begin{eqnarray*}
    \overrightarrow  {\rm  Ind}(\vec  C^r, \frac{n}{2},  \frac{m}{2}) =
        \overrightarrow{\rm  Ind}(   C^r,  \frac{n}{2},  \frac{m}{2}) 
        \end{eqnarray*}
    which  are  given by 
      \begin{eqnarray*}
          \bigcup_{k=0}^{[r/n]-1} \big\{(v_{i_0},v_{i_1}, \ldots, v_{i_k})\mid  n <i_{j}-i_{j-1}\leq  m{\rm~for~any~}   
      1\leq  j\leq  k{\rm~and~}     n <  i_0+r-i_k \leq  m  \big\}.     
      \end{eqnarray*}
      Here  $0\leq  i_0<\cdots  <i_k < r$.   In  particular,  
         if  $n>r/2$,  then  these   path  independence  complexes    are    the  discrete  vertex  set  $\mathbb{Z}/r \mathbb{Z}$,  which  is     
         of  dimension  zero.  
\end{example}
  
     \begin{proposition}
    \label{le-97.08.b}
     For  any  digraph  $\vec  G$  with   its  underlying graph  $G$, 
     we  have   
    persistent       $\mathbb{Z}_2$-equivariant  isometric  embeddings  
   of   filtered  metric  spaces  
   \begin{eqnarray} \label{eq-3.path.19}
  I_{\vec  G,k}(-,\infty): && (\overrightarrow {\rm  Conf}_k(V_{  G}, -,\infty),   (d_{  G})^k) \longrightarrow 
    ( \overrightarrow{\rm  Conf}_k(V_{\vec  G}, -,\infty),   (d_{\vec  G})^k),\\
       \label{eq-3.path.19999}
  J_{\vec  G,k}(\frac{1}{2},-): && (\overrightarrow {\rm  Conf}_k(V_{\vec  G}, \frac{1}{2},-),   (d_{\vec   G})^k) \longrightarrow 
    ( \overrightarrow{\rm  Conf}_k(V_{  G},\frac{1}{2},-),   (d_{ G})^k)
        \end{eqnarray}
   such  that  $ I_{\vec  G,k}(1/2,\infty)= J_{\vec  G,k}(1/2,\infty)^{-1}$  is  the  identity  map  
   satisfying
  the   commutative     diagrams  
    \begin{eqnarray}\label{eq-diag-path.1}
   & \xymatrix{
    \overrightarrow{\rm  Conf}_k(V_{  G}, -,\infty) \ar[rr]^-{ I_{\vec   G,k}(-,\infty)}  \ar[d]_-{ \sim/\mathbb{Z}_2} 
    && \overrightarrow{\rm  Conf}_k(V_{\vec  G}, -,\infty) \ar[d]^-{\sim /\mathbb{Z}_2} \\
     \overrightarrow{\rm  Conf}_k(V_{  G}, -,\infty)/\mathbb{Z}_2  \ar[rr]^-{ I_{\vec  G,k}(-,\infty)/\mathbb{Z}_2} 
      && \overrightarrow{\rm  Conf}_k(V_{\vec  G}, -,\infty)/\mathbb{Z}_2, 
    }\\
  &  \xymatrix{
    \overrightarrow{\rm  Conf}_k(V_{ \vec G}, \frac{1}{2},-) \ar[rr]^-{ J_{ \vec  G,k}(\frac{1}{2},-)}  \ar[d]_-{ \sim/\mathbb{Z}_2} 
    && \overrightarrow{\rm  Conf}_k(V_{  G}, \frac{1}{2},-) \ar[d]^-{\sim /\mathbb{Z}_2} \\
     \overrightarrow{\rm  Conf}_k(V_{  \vec  G},\frac{1}{2},-)/\mathbb{Z}_2  \ar[rr]^-{ J_{\vec  G,k}(\frac{1}{2},-)/\mathbb{Z}_2} 
      && \overrightarrow{\rm  Conf}_k(V_{  G}, \frac{1}{2},-)/\mathbb{Z}_2.  
    }
    \label{eq-diag-path.2}
    \end{eqnarray}
    \end{proposition}
    
    \begin{proof}
    The  proof  of  (\ref{eq-3.path.19})   is   analogous  with  Proposition~\ref{pr-3.a1}.  
    The  proofs  of  (\ref{eq-diag-path.1})  and (\ref{eq-diag-path.2}) are  analogous  with  Corollary~\ref{le-99.01}.  
    \end{proof}

\begin{corollary}\label{pr-path.v1}
   For  any  digraph  $\vec  G$   with   its  underlying graph  $G$,  
   we  have  
     $\mathbb{Z}_2$-equivariant    persistent   simplicial  embeddings   
     \begin{eqnarray}\label{eq-path.b79}
I_{\vec  G}(-,\infty):    && 
 \overrightarrow{\rm  Ind}( {  G},  -,\infty)  \longrightarrow   
 \overrightarrow{\rm  Ind}( {\vec  G}, -,\infty),\\
J_{\vec  G}(\frac{1}{2},-):    && 
 \overrightarrow{\rm  Ind}( {\vec  G},  \frac{1}{2},-)  \longrightarrow   
 \overrightarrow{\rm  Ind}( {  G}, \frac{1}{2},-)   
 \label{eq-path.b7999}
   \end{eqnarray}
   such  that $ I_{\vec  G}(1/2,  \infty)=J_{\vec  G}(1/2,  \infty)^{-1}$   is  the  identity  map.    
       \end{corollary}

\begin{proof}
The  corollary  follows  from   Definition~\ref{def-25apr-5}  and   Proposition~\ref{le-97.08.b}.  
\end{proof}

 \begin{proposition}\label{le-path.a1q}
    \begin{enumerate}[(1)]
    \item
     Let  $\varphi:  \vec  G\longrightarrow  \vec  G'$  be  a  strong   totally geodesic  immersion  
   of  digraphs  with radius  $m_0/2$  (resp.   a  strong   totally geodesic  embedding   
   of   digraphs).    
      Then   $\varphi$  induces  a  double-persistent      $\mathbb{Z}_2$-equivariant     isometric   embedding  
   of  double-filtered  metric  spaces 
   \begin{eqnarray*} 
   \Phi_k(-,-):    (\overrightarrow {\rm  Conf}_k(V_{\vec  G}, -,-),  (d_{\vec  G})^k) \longrightarrow   
   \overrightarrow{\rm  Conf}_k(V_{\vec  G'}, -,-),  (d_{\vec  G'})^k) 
   \end{eqnarray*}
   for  $1\leq  n<m\leq  m_0$    (resp.   for  $1\leq  n<m \leq  \infty$)
    where  $n/2$    and  $m/2$  are the     parameters  
   in  the  double-persistence;
   \item
   Let  $\varphi:   G\longrightarrow     G'$  be  a  strong   totally geodesic  immersion  
   of   graphs  with radius  $m_0/2$  (resp.   a  strong   totally geodesic  embedding   
   of     graphs).    
      Then   $\varphi$  induces  a  double-persistent     $\mathbb{Z}_2$-equivariant   isometric    embedding  
   of  double-filtered  metric  spaces 
   \begin{eqnarray*} 
   \Phi_k(-,-):    (\overrightarrow {\rm  Conf}_k(V_{  G}, -,-),  (d_{  G})^k) \longrightarrow   
   \overrightarrow{\rm  Conf}_k(V_{  G'}, -,-),  (d_{ G'})^k) 
   \end{eqnarray*}
   for  $1\leq  n<m\leq  m_0$    (resp.   for  $1\leq  n<m \leq  \infty$)
   where  $n/2$    and  $m/2$  are the     parameters  
   in  the  double-persistence.  
   \end{enumerate}  
\end{proposition}

\begin{proof}
The  proofs  of  (1)  and  (2) are   analogous  with  Proposition~\ref{le-9.a1q}
and  Proposition~\ref{le-9.au-1q}  respectively.  
\end{proof}

\begin{corollary}\label{co-path-conf-1}
   \begin{enumerate}[(1)]
    \item
     Let  $\varphi:  \vec  G\longrightarrow  \vec  G'$  be  a  strong   totally geodesic  immersion  
   of  digraphs  with radius  $m_0/2$   (resp.  a  strong   totally geodesic  embedding   
   of   digraphs).    
      Then   $\varphi$  induces  a       double-persistent  embedding  
   \begin{eqnarray*}
   \Phi(-,-):    \overrightarrow {\rm  Ind}({\vec  G}, -,-)  \longrightarrow  
    \overrightarrow  {\rm   Ind} ({\vec  G'}, -,-)  
   \end{eqnarray*}
   for  $1\leq  n<m\leq  m_0$     (resp.   for  $1\leq  n<m \leq  \infty$)  
     where  $n/2$    and  $m/2$  are the     parameters  
   in  the  double-persistence;
   \item
    Let  $\varphi:    G\longrightarrow   G'$  be  a  strong   totally geodesic  immersion  
   of   graphs  with radius  $m_0/2$   (resp.  a  strong   totally geodesic  embedding   
   of   graphs).    
      Then   $\varphi$  induces  a       double-persistent  embedding  
   \begin{eqnarray*} 
   \Phi(-,-):    \overrightarrow {\rm  Ind}({  G}, -,-)  \longrightarrow  
    \overrightarrow  {\rm   Ind} ({   G'}, -,-)  
   \end{eqnarray*}
   for  $1\leq  n<m\leq  m_0$     (resp.   for  $1\leq  n<m \leq  \infty$)  
     where  $n/2$    and  $m/2$  are the     parameters  
   in  the  double-persistence.  
   \end{enumerate}  
\end{corollary}

\begin{proof}
The  corollary (1)  and  (2)  follow   from  Proposition~\ref{le-path.a1q}  (1)  and  (2) 
respectively.  
\end{proof}

   \subsection{The  infimum   and   the   supremum  chain  complexes}\label{ss2.1}

Let  $C=(C_q,  \partial_q)_{q\in \mathbb{Z}}$  be  a  chain   complex  where 
$C_q$  are   abelian  groups  and 
$\partial_q:  C_q\longrightarrow  C_{q-1}$   are  homomorphisms
 such  that  $\partial_{q-1}\partial_q=0$ 
 for  any  $q\in \mathbb{Z}$.  
 Let  $D=(D_q)_{q\in \mathbb{Z}}$  be  a  graded   subgroup  of  $C$.  
 The  infimum  chain  complex  $ {\rm  Inf} (D,C)$  and  the  supremum  chain  complex 
 $  {\rm  Sup} (D,C)$  are 
 sub-chain  complexes  of  $C$   given by
 (cf.   \cite[Sec.  2]{h1})
 \begin{eqnarray*}
 {\rm  Inf}_q(D,C)&=&  D_q\cap \partial_q^{-1}  D_{q-1}, \\
  {\rm  Sup}_q(D,C)&=&  D_q+ \partial_{q+1} D_{q+1}   
 \end{eqnarray*}
 for  any  $q\in \mathbb{Z}$.    
Note  that  both   $ {\rm  Inf} (D,C)$  and  
 $  {\rm  Sup} (D,C)$   do  not   depend  on  the  choice  of   the  ambient  chain  complex  $C$,  
 i.e.  if  there  is  another  chain  complex  $C'=(C'_q,  \partial'_q)_{q\in \mathbb{Z}}$
  such  that  $D$  is  also   a  graded  subgroup  of  $C'$  
  and  $\partial'\mid_D =\partial\mid_D$,  
  then  
  \begin{eqnarray*}
  {\rm  Inf} (D,C')&=&{\rm  Inf} (D,C), \\   {\rm  Sup} (D,C') &=&  {\rm  Sup} (D,C). 
  \end{eqnarray*}
 Thus we  can  simply  write  ${\rm  Inf} (D,C)$  as  ${\rm  Inf} (D)$  and  
 write  $ {\rm  Sup} (D,C)$  as  $ {\rm  Sup} (D)$.  
 It  is  proved  in  \cite[Sec. 2]{h1}
    that  the  canonical  inclusion  $\iota:  {\rm  Inf} (D)\longrightarrow  {\rm  Sup} (D)$  
is  a  quasi-isomorphism  of  chain  complexes.  
 Denote  $H_q(D)$  for  
 $H_q({\rm  Inf} (D))\cong   H_q({\rm  Sup} (D))$  for  any  $q\in \mathbb{Z}$.

Let  $C=(C_q,  \partial_q)_{q\in \mathbb{Z}}$   and 
$C'=(C'_q,  \partial'_q)_{q\in \mathbb{Z}}$ be    chain   complexes.     
Let  $D=(D_q)_{q\in \mathbb{Z}}$  and  $D'=(D'_q)_{q\in \mathbb{Z}}$  be 
   graded   subgroups  of  $C$   and  $C'$  respectively.  
Let $\varphi:  C\longrightarrow  C'$  be  a  chain  map such  that  $\varphi(D)\subseteq  D'$.  
Then  $\varphi$  induces  chain  maps 
\begin{eqnarray*}
{\rm  Inf}(\varphi):  && {\rm  Inf} (D)\longrightarrow   {\rm  Inf} (D'),\\
{\rm  Sup}(\varphi):  && {\rm  Sup} (D)\longrightarrow   {\rm  Sup} (D')
\end{eqnarray*}
such  that  the  diagram  commutes 
\begin{eqnarray*}
\xymatrix{
{\rm  Inf} (D)\ar[rr]^-{{\rm  Inf}(\varphi)}  \ar[d]_-{\iota}
&&   {\rm  Inf} (D')\ar[d]^-{\iota'}\\
{\rm  Sup} (D)\ar[rr]^-{{\rm  Sup}(\varphi)}  
&&   {\rm  Sup} (D')   
}
\end{eqnarray*}
where  $\iota$  and  $\iota'$  are  the canonical  injective  quasi-isomorphisms.  
Thus   ${\rm  Inf}(\varphi)$  and  ${\rm  Sup}(\varphi)$  induce  the  same  homomorphism 
in  homology
\begin{eqnarray*}
\varphi_*= {\rm  Inf}(\varphi)_*={\rm  Sup}(\varphi)_*:  
H_q(D)\longrightarrow  H_q(D')
\end{eqnarray*} 
for  any  $q\in  \mathbb{Z}$.

Let  $G$  be  a  group.  Suppose  $G$  act  on  $C$  and  $C'$   such  that      
each  $g\in  G$  induces  self-chain  maps  on  $C$  and  $C'$  respectively.  
Suppose  $D$  and  $D'$  are  $G$-invariant  subgroups.  
Let   $\varphi:  C\longrightarrow  C'$  be  a  $G$-equivariant  chain  map,  
i.e.  a  chain  map  such  that  for any  $g\in  G$,  the  diagram   commutes 
\begin{eqnarray}\label{eq-diag-prelim-1}
\xymatrix{
C\ar[rr]^-{\varphi}\ar[d]_-{g}  && C'\ar[d]^-{g}\\
C\ar[rr]^-{\varphi}  && C'.  
}
\end{eqnarray}

\begin{lemma}\label{le-pre-1}
\begin{enumerate}[(1)]
\item
The  chain  complexes  $ {\rm  Inf} (D)$  and   
 $  {\rm  Sup} (D)$  are  $G$-invariant;  
 \item
 The  quasi-isomorphism  $\iota$  is  $G$-equivariant;  
\item
The  chain  maps  ${\rm  Inf}(\varphi)$  and  ${\rm  Sup}(\varphi)$  are  $G$-equivariant. 
\end{enumerate}
\end{lemma}

\begin{proof}
(1)  It  follows  from  the  $G$-equivariance  of  $\varphi$  and  the  $G$-invariance  of  $D$  
that  
  $\partial_q^{-1}  D_{q-1}$  as  well  as   $\partial_{q+1} D_{q+1}$  is  $G$-invariant 
   for  each  $q\in \mathbb{Z}$.  
Thus  $ {\rm  Inf} (D)$  and   
 $  {\rm  Sup} (D)$  are  $G$-invariant.

(2) 
 We  have  a  $G$-action on  $C$.   Restricted  to     $  {\rm  Sup} (D)$,  
this induces  a  $G$-action  on   $  {\rm  Sup} (D)$;   
 and  restricted  to     $  {\rm  Inf} (D)$,  
 this  induces  a  $G$-action  on   $  {\rm  Inf} (D)$. 
The  diagram  commutes 
\begin{eqnarray*}
\xymatrix{
 {\rm  Inf} (D)\ar[d]_-{\iota} \ar[r]^-{g} &  {\rm  Inf} (D)\ar[d]^-{\iota}\\
  {\rm  Sup} (D) \ar[r]^-{g} &  {\rm  Sup} (D)
}
\end{eqnarray*}
for  any  $g\in  G$.  
Therefore,  the  canonical  inclusion  $\iota$  is  $G$-equivariant.

(3)  
For  any  $g\in  G$,  the  diagram  (\ref{eq-diag-prelim-1})  induces     commutative  diagrams 
\begin{eqnarray*} 
\xymatrix{
  {\rm  Inf} (D)\ar[rr]^-{{\rm  Inf}(\varphi)}\ar[d]_-{g}  &&   {\rm  Inf} (D')\ar[d]^-{g}\\
  {\rm  Inf} (D)\ar[rr]^-{{\rm  Inf}(\varphi)}  && {\rm  Inf} (D'),  
}
~~~~~~
\xymatrix{
  {\rm  Sup} (D)\ar[rr]^-{{\rm  Sup}(\varphi)}\ar[d]_-{g}  &&   {\rm  Sup} (D')\ar[d]^-{g}\\
  {\rm  Sup} (D)\ar[rr]^-{{\rm  Sup}(\varphi)}  && {\rm  Sup} (D').  
}
\end{eqnarray*}
Therefore,   ${\rm  Inf}(\varphi)$  and  ${\rm  Sup}(\varphi)$  are  $G$-equivariant. 
\end{proof}

\begin{corollary}\label{co-pre.a1}
The  induced  homomorphism  $\varphi_*:  H_q(D)\longrightarrow  H_q(D')$  
is  $G$-equivariant.  
\end{corollary}

\begin{proof}
The $G$-equivariant  chain  maps  ${\rm  Inf}(\varphi)$  and  ${\rm  Sup}(\varphi)$
in     Lemma~\ref{le-pre-1}~(3)  induce   a   $G$-equivariant   homomorphism  $\varphi_*:  H_q(D)\longrightarrow  H_q(D')$  for  any  $q\in \mathbb{Z}$.    
\end{proof}

 \subsection{Chain  complexes  associated  with  the  path  independence  complexes} 
 \label{ssa-4.3}

     Let  $R$  be  a  commutative  ring with  unit.  
 Let  $\Lambda_{k}(V)$  be  the  free  $R$-module 
   spanned  by  all  the  elementary  $k$-paths    (cf.  Definition~\ref{def--2})  on  $V$. 
   Define  the  {\it   boundary  map}  
   $
   \partial_k:  \Lambda_{k}(V)\longrightarrow \Lambda_{k-1}(V)
   $ 
   by  
   \begin{eqnarray*}
   \partial_k(v_0v_1\ldots  v_k) =\sum_{i=0}^k (-1)^i  v_0\ldots\widehat{v_i} \ldots   v_k. 
   \end{eqnarray*}
   Then  $\partial_{k-1}\partial _k=0$  hence  $(\Lambda_{*}(V), \partial_*)$  is  a  chain  complex
    (cf.  \cite{lin2,lin3,lin6}). 
    Let   $\mathcal{R}_{k}(V)$  be  the  free  $R$-module 
   spanned  by  all  the   regular elementary  $k$-paths   on  $V$ and  
   let  $I_{k}(V)$  be  the  free  $R$-module 
   spanned  by  all  the   non-regular  elementary  $k$-paths   on  $V$.
    Then  $\Lambda_{k}(V)= \mathcal{R}_{k}(V)\oplus  I_{k}(V)$.  
    By   \cite[Lemma~2.9~(a)]{lin2},   
    $( I_*(V), \partial_*)$
    is  a   sub-chain  complex  of  $(\Lambda_{*}(V), \partial_*)$.  
    By   \cite[Definition~2.10]{lin2},  
    $(\mathcal{R}_{*}(V), \tilde\partial_*)$  is  a     chain  complex  
      with  the  {\it  regular  boundary  operator}  $\tilde \partial_*$, 
     i.e.   the  induced  boundary  operator  of  the  quotient  chain  complex  
    $\Lambda_{*}(V)/  I_*(V)$.

    We  have  a  canonical   $\mathbb{Z}_2$-action  on  $(\Lambda_{*}(V), \partial_*)$ 
    such  that    
    the  nontrivial  element  of   $\mathbb{Z}_2$  is  a  chain  map   with  respect  to  $ \partial_*$  
    sending    every  path to  its  inverse.   
    The  sub-chain  complex   $( I_*(V), \partial_*)$  of  $(\Lambda_{*}(V), \partial_*)$ 
     is       $\mathbb{Z}_2$-invariant.   
    This  induces  a    $\mathbb{Z}_2$-action  on  $(\mathcal{R}_{*}(V), \tilde\partial_*)$ 
     such  that    
    the  nontrivial  element  of   $\mathbb{Z}_2$  is  a  chain  map  with  respect  to  $\tilde\partial_*$  
    sending    every  regular  path to  its  inverse.   
    The  graded  sub-$R$-module   $\mathcal{A}_*(\vec  G)$  of  $  \mathcal{R}_{*}(V) $
       is       $\mathbb{Z}_2$-invariant.  
    Thus  the   sub-chain  complex  $(\Omega_*(\vec  G), \tilde\partial_*)$  of  $(\mathcal{R}_{*}(V), \tilde\partial_*)$ 
     is       $\mathbb{Z}_2$-invariant.  
                  Let  
          \begin{eqnarray*}
            \mathcal{D}_{k}( \vec G,  \frac{n}{2},  \frac{m}{2}  )   =    R\Big(\overrightarrow{ {\rm    Conf}}_k(V_{\vec  G}, \frac{n}{2},  \frac{m}{2})\Big) 
        ~~~~~~  \Big(  {\rm  resp.~}        \mathcal{D}_{k}(  G,  \frac{n}{2},  \frac{m}{2}  )   =    R\Big(\overrightarrow{ {\rm    Conf}}_k(V_{  G}, \frac{n}{2},  \frac{m}{2})\Big)\Big)
          \end{eqnarray*}
   be  the  free  $R$-module 
   spanned  by  all  the    independent   elementary  $k$-paths  on  $\vec  G$ 
   (resp.  $G$) with  constraint   $(n/2,  m/2]$.  
   Then  $\mathcal{D}_{k}( \vec G,  n/2,  m/2  )$  
   (resp.  $\mathcal{D}_{k}( G,  n/2,  m/2  )$)
   is  a 
   $\mathbb{Z}_2$-invariant  
    sub-$R$-module  of    $\mathcal{A}_k(\vec  G)$
    (resp.  $\mathcal{A}_k(   G)$).

   \begin{lemma}\label{le-10.1}
   For  any  graph  $\vec  G$  with  its  underlying  graph  $G$,  
   we  have    $\mathbb{Z}_2$-equivariant  
      persistent  momomorphisms  of  free   persistent  $R$-modules  
   \begin{eqnarray}\label{eq-10.1}
   I_{\vec  G}(-,\infty)_\#:  &&  \mathcal{D}_{k}( G,  -,  \infty )\longrightarrow    \mathcal{D}_{k}(\vec G,  -,  \infty),\\
   J_{\vec  G}(\frac{1}{2},-)_\#:  &&  \mathcal{D}_{k}(\vec  G, \frac{1}{2},-)\longrightarrow    \mathcal{D}_{k}( G, \frac{1}{2},-)
   \label{eq-10.19912}
   \end{eqnarray}
    such  that $ I_{\vec  G}(1/2, \infty)_\#= J_{\vec  G}(1/2, \infty)_\#^{-1}$   is  the  identity  map.   
   \end{lemma}
   
   \begin{proof}
   The  proof  follows  from  Proposition~\ref{le-97.08.b}.    
   \end{proof}

 \begin{lemma}\label{le-12-q}
    \begin{enumerate}[(1)]
    \item
     Let  $\varphi:  \vec  G\longrightarrow  \vec  G'$  be  a  strong   totally geodesic  immersion  
   of  digraphs  with radius  $m_0/2$  (resp.   a  strong   totally geodesic  embedding   
   of   digraphs).    
      Then   $\varphi$  induces  a      double-persistent   $\mathbb{Z}_2$-equivariant   monomorphism 
       of  free  double-persistent  $R$-modules   
   \begin{eqnarray*} 
   \Phi_k(-,-)_\#:    \mathcal{D}_k({\vec  G}, -,-) \longrightarrow   
 \mathcal{D}_k({\vec    G'}, -,-) 
   \end{eqnarray*}
   for  $1\leq  n<m\leq  m_0$    (resp.   for  $1\leq  n<m\leq  \infty $)
    where  $n/2$    and  $m/2$  are the     parameters  
   in  the  double-persistence;
   \item
   Let  $\varphi:   G\longrightarrow     G'$  be  a  strong   totally geodesic  immersion  
   of   graphs  with radius  $m_0/2$  (resp.   a  strong   totally geodesic  embedding   
   of     graphs).    
      Then   $\varphi$  induces  a     double-persistent   $\mathbb{Z}_2$-equivariant  
      monomorphism 
       of  free  double-persistent  $R$-modules  
   \begin{eqnarray*} 
  \Phi_k(-,-)_\#:    \mathcal{D}_k({   G}, -,-) \longrightarrow   
 \mathcal{D}_k({    G'}, -,-) 
   \end{eqnarray*}
   for  $1\leq  n<m\leq  m_0$    (resp.   for  $1\leq  n<m\leq  \infty $)
   where  $n/2$    and  $m/2$  are the     parameters  
   in  the  double-persistence.  
   \end{enumerate}  
\end{lemma}

\begin{proof}
The  proofs  of  (1)  and  (2)  follow   from  Proposition~\ref{le-path.a1q}~(1)  and  (2)  respectively.  
\end{proof}

\begin{proposition}
\label{pr-chain-1}
 For  any  graph  $\vec  G$  with  its  underlying  graph  $G$, 
 we  have    commutative  diagrams  of  persistent chain  complexes  
 \begin{eqnarray}\label{diag-25-apr-08-3}
&\xymatrix{
{\rm  Inf}_*(\mathcal{D} ( G,  -,  \infty ))\ar[rr]^-{{\rm  Inf}(I_{\vec  G}(-,\infty)_\#)} \ar[d]
&&  {\rm  Inf}_*( \mathcal{D} (\vec G,  -,  \infty )) \ar[d]\\
{\rm  Sup}_*(\mathcal{D} ( G,  -, \infty ))\ar[rr]^-{{\rm  Sup}(I_{\vec  G}(-,\infty)_\#)}  
&&  {\rm  Sup}_*( \mathcal{D} (\vec G,  -,  \infty )),
}\\
\label{diag-25-apr-08-33333}
&\xymatrix{
{\rm  Inf}_*(\mathcal{D} ( \vec  G,  \frac{1}{2}, - ))\ar[rr]^-{{\rm  Inf}(J_{\vec  G}( \frac{1}{2}, -)_\#)} \ar[d]
&&  {\rm  Inf}_*( \mathcal{D} ( G,  \frac{1}{2}, - )) \ar[d]\\
{\rm  Sup}_*(\mathcal{D} ( \vec  G,  \frac{1}{2}, - ))\ar[rr]^-{{\rm  Sup}(J_{\vec  G}( \frac{1}{2}, -)_\#)}  
&&  {\rm  Sup}_*( \mathcal{D} ( G,  \frac{1}{2}, -))
}
\end{eqnarray}
such  that  all  the  maps  are    persistent 
 $\mathbb{Z}_2$-equivariant  monomorphic  chain  maps    
and  all  the  vertical  maps  are     
quasi-isomorphisms.  
\end{proposition}

\begin{proof}
By  Lemma~\ref{le-pre-1}  and     Lemma~\ref{le-10.1},  
we  have  induced  persistent     $\mathbb{Z}_2$-equivariant   monomorphic  chain  maps  
\begin{eqnarray*}
{\rm Inf}(I_{\vec  G}(-,\infty)_\#): &&  {\rm  Inf}_*( \mathcal{D} (  G,  -,  \infty )) 
\longrightarrow  {\rm  Inf}_*( \mathcal{D} (\vec G,  -,  \infty )),  \\
{\rm  Sup}(I_{\vec  G}(-,\infty)_\#): &&  {\rm  Sup}_*( \mathcal{D} (  G,  -,  \infty )) 
\longrightarrow  {\rm  Sup}_*( \mathcal{D} (\vec G,  -,  \infty))   
\end{eqnarray*}
such that  (\ref{diag-25-apr-08-3})   commutes.  
Similarly,  we  have  
 induced  persistent     $\mathbb{Z}_2$-equivariant   monomorphic  chain  maps  
\begin{eqnarray*}
{\rm Inf}(J_{ \vec  G}(\frac{1}{2},-)_\#): &&  {\rm  Inf}_*( \mathcal{D} ( \vec   G,  \frac{1}{2},- )) 
\longrightarrow  {\rm  Inf}_*( \mathcal{D} (  G, \frac{1}{2},- )),  \\
{\rm  Sup}(J_{\vec  G}(\frac{1}{2},-)_\#): &&  {\rm  Sup}_*( \mathcal{D} ( \vec   G,  \frac{1}{2},- )) 
\longrightarrow  {\rm  Sup}_*( \mathcal{D} (  G,  \frac{1}{2},-))   
\end{eqnarray*}
such that  (\ref{diag-25-apr-08-33333})   commutes.  
\end{proof}

\begin{corollary}\label{co-25-apr-b1}
For  any  graph  $\vec  G$  with  its  underlying  graph  $G$, 
 we  have    canonical  
 $\mathbb{Z}_2$-equivariant   persistent  homomorphisms  of   persistent  homology  groups 
 \begin{eqnarray*}
 I_{\vec  G}(-,\infty)_*:  &&   H_q(\mathcal{D} (  G,  -,  \infty ))
 \longrightarrow   H_q(\mathcal{D} (  \vec G,  -,  \infty )),\\
   J_{\vec  G}(\frac{1}{2},-)_*:  &&   H_q(\mathcal{D} ( \vec  G, \frac{1}{2},- ))
 \longrightarrow   H_q(\mathcal{D} (   G, \frac{1}{2},- ))
 \end{eqnarray*}
 such  that  $ I_{\vec  G}(1/2,\infty)_*= J_{\vec  G}(1/2,\infty)_*^{-1}$  is  the  identity.  
\end{corollary}

\begin{proof}
The  proof  follows  from  Corollary~\ref{co-pre.a1},
Lemma~\ref{le-10.1}   and  Proposition~\ref{pr-chain-1}.  
\end{proof}

By  Proposition~\ref{le-97.08.b},  Corollary~\ref{pr-path.v1}
  and  Corollary~\ref{co-25-apr-b1},
  we  obtain  Theorem~\ref{th-main-intro-3}.

\begin{corollary}\label{co-25-05-23-9}
For  any  digraph  $\vec  G$  with   its  underlying graph  $G$,
 the  group    
 ${\rm  Aut}( \overrightarrow{\rm Ind}(\vec  G, {n}/{2}, {m}/{2}))$  acts  on
 $H_*(\mathcal{D} (  \vec G,{n}/{2}, {m}/{2}))$
 and  the  group   ${\rm  Aut}( \overrightarrow{\rm Ind}(  G, {n}/{2}, {m}/{2}))$  acts  on
 $H_*(\mathcal{D} (  G, {n}/{2}, {m}/{2}))$. 
 Consequently,  
  the  group    
 ${\rm  Aut}( \vec  G )$  acts  on
 $H_*(\mathcal{D} (  \vec G,{n}/{2}, {m}/{2}))$
 and  the  group   ${\rm  Aut}(   G)$  acts  on
 $H_*(\mathcal{D} (  G, {n}/{2}, {m}/{2}))$. 
\end{corollary}

\begin{proof}
Let  $\varphi\in   {\rm  Aut}( \overrightarrow{\rm Ind}(\vec  G, {n}/{2}, {m}/{2}))$.  
The  diagram  commutes
\begin{eqnarray*}
\xymatrix{
{\rm  Inf}_*( \mathcal{D} (\vec G,  {n}/{2}, {m}/{2} )) \ar[d]_-{\varphi} \ar[r]^-{} 
 &{\rm  Sup}_*( \mathcal{D} (\vec G,  {n}/{2}, {m}/{2} ))\ar[d]^-{\varphi} \\ 
 {\rm  Inf}_*( \mathcal{D} (\vec G,  {n}/{2}, {m}/{2} ))   \ar[r]^-{} 
 &{\rm  Sup}_*( \mathcal{D} (\vec G,  {n}/{2}, {m}/{2} )) 
}
\end{eqnarray*}
where  the  horizontal   maps are  canonical  inclusions  
and  the  vertical  maps  are  chain  maps  induced by  $\varphi$.  
This  induces  a  homomorphism  
\begin{eqnarray*}
\varphi:  H_*(\mathcal{D} (  \vec G,{n}/{2}, {m}/{2}))\longrightarrow  H_*(\mathcal{D} (  \vec G,{n}/{2}, {m}/{2}))   
\end{eqnarray*}
 and  consequently   ${\rm  Aut}( \overrightarrow{\rm Ind}(\vec  G, {n}/{2}, {m}/{2}))$  acts  on
 $H_*(\mathcal{D} (  \vec G,{n}/{2}, {m}/{2}))$.  
 With  the  help  of  (\ref{eq-250522-b1}),  ${\rm  Aut}( \vec  G )$  acts  on
 $H_*(\mathcal{D} (  \vec G,{n}/{2}, {m}/{2}))$.

 Similarly,  ${\rm  Aut}( \overrightarrow{\rm Ind}(   G, {n}/{2}, {m}/{2}))$  acts  on
 $H_*(\mathcal{D} (   G,{n}/{2}, {m}/{2}))$.  
 With  the  help  of  (\ref{eq-250522-b2}),  ${\rm  Aut}(   G )$  acts  on
 $H_*(\mathcal{D} (    G,{n}/{2}, {m}/{2}))$.  
\end{proof}

 \begin{proposition}\label{pr-chain-2}
    \begin{enumerate}[(1)]
    \item
     Let  $\varphi:  \vec  G\longrightarrow  \vec  G'$  be  a  strong   totally geodesic  immersion  
   of  digraphs  with radius  $m_0/2$  (resp.   a  strong   totally geodesic  embedding   
   of   digraphs).    
      Then   $\varphi$  induces  a  
      commutative  diagram  
      \begin{eqnarray}\label{diag-25-apr-08-1}
      \xymatrix{
      {\rm  Inf}_*( \mathcal{D} (\vec G,  -,  - ))\ar[rr]^-{{\rm Inf}(\Phi_k(-,-)_\#)} \ar[d] 
      &&{\rm  Inf}_*( \mathcal{D} (\vec G',  -,  - ))\ar[d]\\
            {\rm  Sup}_*( \mathcal{D} (\vec G,  -,  - ))\ar[rr] ^-{{\rm  Sup}(\Phi_k(-,-)_\#)}
      &&{\rm  Sup}_*( \mathcal{D} (\vec G',  -,  - ))
      }
      \end{eqnarray}
      such  that   all  the  maps  are  double-persistent   $\mathbb{Z}_2$-equivariant   monomorphic 
            chain  maps    
    for  $1\leq  n<m\leq  m_0$    (resp.   for  $1\leq  n<m\leq  \infty $)
    where  $n/2$    and  $m/2$  are the     parameters  
   in  the  double-persistence;
   \item
   Let  $\varphi:   G\longrightarrow     G'$  be  a  strong   totally geodesic  immersion  
   of   graphs  with radius  $m_0/2$  (resp.   a  strong   totally geodesic  embedding   
   of     graphs).    
      Then   $\varphi$  induces  a  
      commutative  diagram 
        \begin{eqnarray*}
      \xymatrix{
      {\rm  Inf}_*( \mathcal{D} (  G,  -,  - ))\ar[rr]^-{{\rm Inf}(\Phi_k(-,-)_\#)} \ar[d] 
      &&{\rm  Inf}_*( \mathcal{D} ( G',  -,  - ))\ar[d]\\
            {\rm  Sup}_*( \mathcal{D} (  G,  -,  - ))\ar[rr] ^-{{\rm  Sup}(\Phi_k(-,-)_\#)}
      &&{\rm  Sup}_*( \mathcal{D} (  G',  -,  - ))
      }
      \end{eqnarray*}
    such  that   all  the  maps  are  double-persistent 
         $\mathbb{Z}_2$-equivariant    monomorphic    chain  maps    
    for  $1\leq  n<m\leq  m_0$    (resp.   for  $1\leq  n<m\leq  \infty $)
    where  $n/2$    and  $m/2$  are the     parameters  
   in  the  double-persistence.
   \end{enumerate}  
\end{proposition}

\begin{proof}
(1)  
By  Lemma~\ref{le-pre-1}  and  Lemma~\ref{le-12-q}~(1), 
we  have  induced   double-persistent   $\mathbb{Z}_2$-equivariant    monomorphic  chain  maps  
\begin{eqnarray*}
{\rm Inf}(\Phi_k(-,-)_\#): &&  {\rm  Inf}_*( \mathcal{D} (\vec G,  -,  - )) 
\longrightarrow  {\rm  Inf}_*( \mathcal{D} (\vec G',  -,  - )),  \\
{\rm  Sup}(\Phi_k(-,-)_\#): &&  {\rm  Sup}_*( \mathcal{D} (\vec G,  -,  - )) 
\longrightarrow  {\rm  Sup}_*( \mathcal{D} (\vec G',  -,  - ))   
\end{eqnarray*}
such  that  (\ref{diag-25-apr-08-1})  commutes.

(2)  The  proof   is   analogous with   (1). 
It  follows from  Lemma~\ref{le-pre-1}  and  Lemma~\ref{le-12-q}~(2).  
\end{proof}

 \begin{corollary}\label{co-apr-chain-2}
    \begin{enumerate}[(1)]
    \item
     Let  $\varphi:  \vec  G\longrightarrow  \vec  G'$  be  a  strong   totally geodesic  immersion  
   of  digraphs  with radius  $m_0/2$  (resp.   a  strong   totally geodesic  embedding   
   of   digraphs).    
      Then   $\varphi$  induces  a  
       $\mathbb{Z}_2$-equivariant   double-persistent  homomorphism    
    for  $1\leq  n<m\leq  m_0$    (resp.   for  $1\leq  n<m \leq  \infty$)
    \begin{eqnarray*}
    \Phi_k(-,-)_*:  H_q(\mathcal{D} (\vec G,  -,  - ))\longrightarrow  H_q(\mathcal{D} (\vec G',  -,  - ))
    \end{eqnarray*}
    where  $n/2$    and  $m/2$  are the     parameters  
   in  the  double-persistence;
   \item
   Let  $\varphi:   G\longrightarrow     G'$  be  a  strong   totally geodesic  immersion  
   of   graphs  with radius  $m_0/2$  (resp.   a  strong   totally geodesic  embedding   
   of     graphs).    
      Then   $\varphi$  induces   a  
       $\mathbb{Z}_2$-equivariant   double-persistent  homomorphism    
    for  $1\leq  n<m\leq  m_0$    (resp.   for  $1\leq  n<m \leq  \infty$)
    \begin{eqnarray*}
    \Phi_k(-,-)_*:  H_q(\mathcal{D} (  G,  -,  - ))\longrightarrow  H_q(\mathcal{D} (  G',  -,  - ))
    \end{eqnarray*}
    where  $n/2$    and  $m/2$  are the     parameters  
   in  the  double-persistence.
   \end{enumerate}  
\end{corollary}

   \begin{proof}
   The  proof   of  (1)    follows   from  Corollary~\ref{co-pre.a1}  and  Proposition~\ref{pr-chain-2}~(1).  
   The  proof  of  (2)    follows   from  Corollary~\ref{co-pre.a1}  and  Proposition~\ref{pr-chain-2}~(2).   
   \end{proof}
   
  By  Proposition~\ref{le-path.a1q},  Corollary~\ref{co-path-conf-1} 
    and  Corollary~\ref{co-apr-chain-2},  
  we  obtain  Theorem~\ref{th-main-intro-4}.



 \section{The    Shannon  capacities  of  digraphs}\label{s-5}

 In this  section,  we  apply  the  independence  complexes of  (di)graphs in  Section~\ref{ssa-4}  
  to  study  the  Shannon capacities.  
  In  Subsection~\ref{ssa-5.1},  
  we  prove  some  lemmas  on the strong products  of  (di)graphs.  
  In  Subsection~\ref{ssa-5.2},  
  we  prove that the  Shannon  capacity  of the underlying graph  is  smaller than or  equal
   to the  Shannon  capacity  of  a  digraph.  We  prove  that  given  a 
   strong  totally  geodesic  embedding   of  (di)graphs,  
  the  Shannon  capacity  of  the  ambient  (di)graph   is  larger.  
 
 \subsection{Strong  products  of  (di)graphs}
 \label{ssa-5.1}
 
Given  two  graphs  $G_1=(V_1,E_1)$  and  $G_2=(V_2,E_2)$,  
recall  that  
their  {\it strong product}  $G_1\boxtimes  G_2$      
 is the graph with vertex set  $V_1\times V_2$  and with edge set specified by putting
  $(u,v)$  adjacent to $(u',v')$  iff    one  of  the followings  is   satisfied:  
   (1)  $u=u'$  and  $\{v,v'\}\in E_{2}$,     (2)  $v=v'$  and  $\{u,u'\}\in E_{1}$,  or  
   (3) $\{u,u'\}\in E_{1}$  and  $\{v,v'\}\in E_{2}$
 (cf. \cite{2,1,4}).

 Similarly,   
given two digraphs $\vec  G_1=(V_1,E_1)$ and $\vec  G_2=(V_2,E_2)$,  
we  define  their  {\it strong product} 
$\vec  G_1\boxtimes \vec  G_2$    as   the   digraph  whose  vertex set  is  $V_1\times V_2$   and  whose  arc  set  is  specified  by the following  rule:   
 for  any  distinct two  vertices  $(u,v )$  and  $ (u',v')$  in $V_1\times V_2$,    
 there  is  an  arc  
  $(u,v )\to  (u',v')$
  iff       one  of  the followings  is   satisfied:   
  (1)  $u=u'$  and  $v\to  v'$  is  an  arc  of  $ \vec  G_2$,   
    (2)  $v=v'$  and  $ u\to  u'$   is  an  arc  of  $ \vec  G_1$,  or
      (3) $ u\to  u'$   is  an  arc  of  $ \vec  G_1$    and   $v\to  v'$  is  an  arc  of  $ \vec  G_2$.

\begin{lemma}\label{le-apr-12-1}
For any  digraphs   $\vec   G_1$  and  $\vec  G_2$,  
we  have  
\begin{eqnarray}\label{eq-apr-12-sub1}
\pi(\vec  G_1\boxtimes  \vec  G_2) \subseteq  \pi(\vec  G_1)\boxtimes  \pi(\vec  G_2),
\end{eqnarray}
i.e.  the  underlying   graph  of  the   strong  product   of  digraphs    is  a  subgraph  of  
 the  strong  product   of  the  underlying  graphs.  
\end{lemma}

\begin{proof}
The  vertex  sets  of  both  $\pi(\vec  G_1\boxtimes  \vec  G_2) $  
and  $ \pi(\vec  G_1)\boxtimes  \pi(\vec  G_2)$  are  $V_1\times  V_2$.   
 For  any  distinct two  vertices  $(v_1,v_2)$  and  $ (u_1,u_2)$  in $V_1\times V_2$,
\begin{eqnarray*}
&&\{(v_1,v_2),  (u_1,u_2)\} {\rm~is~an ~edge~of~}\pi(\vec  G_1\boxtimes  \vec  G_2) \\
&\Longleftrightarrow &
  (v_1,v_2)\to  (u_1,u_2){\rm~or~}   
  (u_1,u_2)\to  (v_1,v_2) {\rm~ is~ an ~ arc~of~}\vec  G_1\boxtimes  \vec  G_2\\
 &\Longleftrightarrow &
 \Big[(v_1=u_1 {\rm~or~}  v_1\to  u_1 {\rm~is~an~arc~of~} \vec  G_1)  
 {\rm~and~} (v_2=u_2 {\rm~or~}  v_2\to  u_2 {\rm~is~an~arc~of~} \vec  G_2) \Big]\\
 &&{\rm~or~}
 \Big[(v_1=u_1 {\rm~or~}  u_1\to  v_1 {\rm~is~an~arc~of~} \vec  G_1)  
 {\rm~and~} (v_2=u_2 {\rm~or~}  u_2\to  v_2 {\rm~is~an~arc~of~} \vec  G_2) \Big]
\end{eqnarray*}
and  
\begin{eqnarray*}
&&\{(v_1,v_2),  (u_1,u_2)\} {\rm~is~an ~edge~of~} 
\pi(\vec  G_1)\boxtimes  \pi(\vec  G_2)\\
&\Longleftrightarrow &
{\rm~for~}i=1,2,  {\rm~ we  ~have~} 
 v_i=u_i{\rm~or~}   
  \{v_i,u_i\} {\rm~ is~ an ~ edge~of~}\pi(\vec  G_i)\\
 &\Longleftrightarrow &
 {\rm~for~}i=1,2,    {\rm~ we  ~have~}   
  v_i=u_i,  {\rm~or~}
  v_i\to  u_i  {\rm~ is~ an ~ arc~of~} \vec  G_i,  
  {\rm~or~}
  u_i\to  v_i  {\rm~ is~ an ~ arc~of~} \vec  G_i.  
\end{eqnarray*}
 Therefore,  
 each  edge  of  $\pi(\vec  G_1\boxtimes  \vec  G_2) $   
 is  an  edge  of  $ \pi(\vec  G_1)\boxtimes  \pi(\vec  G_2)$.  
 We  obtain  (\ref{eq-apr-12-sub1}).  
\end{proof}

\begin{corollary}\label{le-8.a1}
For  any  digraph  $\vec  G$  with  its  underlying  graph  $G$ and any  positive  integer  $p$,
  we  have  
\begin{eqnarray}\label{eq-apr-v1}
\pi(\vec  G^{\boxtimes  p})  \subseteq      G^{\boxtimes  p},  
\end{eqnarray} 
i.e.  the  underlying graph  of  the   $p$-fold  strong  product  of    $\vec  G $  is  a  subgraph  of  
the   $p$-fold  strong  product  of  
$   G $.   
\end{corollary}

\begin{proof}
The  proof  follows  from  Lemma~\ref{le-apr-12-1}  and  an  induction  on  $p$.   
\end{proof}

\begin{lemma}
\label{le-25-apr-le08a}
\begin{enumerate}[(1)]
\item
For  any  digraphs  $\vec  G_1$  and  $\vec  G_2$,   any  vertices  $u,u'$ of  $\vec  G_1$
and  any  vertices   $v,v'$  of  $\vec  G_2$,  we  have  
\begin{eqnarray}\label{eq-25-apr-d1}
d_{\vec  G_1\boxtimes \vec  G_2}  ((u,v),(u',v'))= \max\{d_{\vec  G_1}(u,u'),  d_{\vec  G_2}(v,v')\}; 
\end{eqnarray}
\item
For  any   graphs  $   G_1$  and  $   G_2$,   any  vertices  $u,u'$ of  $  G_1$
and  any  vertices   $v,v'$  of  $   G_2$,  we  have  
\begin{eqnarray}\label{eq-25-apr-d2}
d_{  G_1\boxtimes  G_2}  ((u,v),(u',v'))= \max\{d_{   G_1}(u,u'),  d_{   G_2}(v,v')\}.  
\end{eqnarray}
\end{enumerate}
\end{lemma}
\begin{proof}
(1)  Suppose $d_{ \vec  G_1\boxtimes  \vec  G_2}((u,v),(u',v'))=n$.  
  Then  there  exists  a  minimal  path  
 \begin{eqnarray*}
 \gamma=(u_0, v_0)   (u_1, v_1)   \ldots  (u_n,v_n) 
 \end{eqnarray*}
 in  $  G_1\boxtimes  G_2$  
   such  that  
 $(u_0,v_0)=(u,v)$,  $(u_n,v_n)=(u',v')$  and  
 $(u_{i-1}, v_{i-1})\to  (u_i,v_i)$  is  an  arc  of  $\vec  G_1\boxtimes  \vec  G_2$
  for  any  $1\leq  i\leq  n$.   
 By  the  definition  of  the  strong  product  of  digraphs,     either  $u_{i-1}=u_i$  or 
 $u_{i-1}\to  u_i$  is  an  arc  of  $\vec  G_1$,    and  
  either  $v_{i-1}=v_i$  or  $v_{i-1}\to  v_i$  is  an  arc  of  $\vec  G_2$.  
  Thus   $\eta_1(\gamma)$  is  a  path   in  $\vec  G_1$  from  $u$  to  $u'$   and   $\eta_2(\gamma)$    is  a  path   in  $\vec  G_2$   from  $v$  to  $v'$.   
 We  claim  that  
 either  
 $\eta_1(\gamma)=u_0u_1\ldots  u_n$  is  a  minimal  path  in  $\vec  G_1$  
  or    $\eta_2(\gamma)=v_0v_1\ldots  v_n$  is  a  minimal  path  in  $\vec  G_2$.  
 Since  the  minimality of   $\eta_1(\gamma)$  is equivalent to  $d_{ \vec  G_1}(u,u')=n$ 
 and   the  minimality of   $\eta_2(\gamma)$  is equivalent to  $d_{ \vec  G_2}(v,v')=n$,  
 this  claim  implies (\ref{eq-25-apr-d1}).

 To  prove  the  claim, 
we  suppose  to  the  contrary  that    $\eta_i(\gamma)$  
   is  not  minimal  in  $\vec  G_i$  for  both  $i=1,2$.   
   Then  there  exists   $\theta=\tilde  u_0 \tilde u_1 \ldots \tilde  u_m$  such  that  
   $m\leq  n-1$,  $\tilde  u_0=u$,  $\tilde  u_m=u'$  and  
   $\tilde  u_{i-1}\to  \tilde  u_i$  is  an  arc  of  $\vec  G_1$  for  any  $1\leq  i\leq  m$
    as  well  as    $\eta=\tilde  v_0 \tilde v_1 \ldots \tilde  v_l$  such  that  
   $l\leq  n-1$,  $\tilde  v_0=v$,  $\tilde  v_l=v'$  and  
   $\tilde  v_{i-1}\to  \tilde  v_i$  is  an  arc  of  $\vec  G_2$  for  any  $1\leq  i\leq  l$.  
   Without  loss  of  generality,  assume  $m\leq   l$.  
   Let  
   \begin{eqnarray*}
   \tilde \gamma =  ( \tilde  u_0,\tilde  v_0)  ( \tilde  u_1,\tilde  v_1)
   \ldots   ( \tilde  u_m,\tilde  v_m) \ldots  ( \tilde  u_m,\tilde  v_l).  
   \end{eqnarray*}
   Then $\tilde \gamma$  is  a  path  in  $\vec  G_1\boxtimes  \vec  G_2$
   from  $(u,v)$  to  $(u',v')$  of  length  $l$.  
   This  contradicts  that  $\gamma$  is  minimal.  
     We  obtain  the  claim.

(2)  The  proof  is  analogous  with  (1).  
\end{proof}

\begin{corollary}\label{co-25-apr-prod-5}
\begin{enumerate}[(1)]
\item
For  any   digraph   $ \vec  G $  and   any  vertices  
$v_1,\ldots, v_p,  v'_1,\ldots, v'_p $  of  $\vec  G$,    we  have  
\begin{eqnarray}\label{eq-25-apr-d5}
d_{\vec  G^{\boxtimes p}}  ((v_1,\ldots,v_p),( v'_1,\ldots,v'_p))= \max_{1\leq  i\leq  p}\{d_{\vec  G }(v_i,v'_i)\}; 
\end{eqnarray}
\item
For  any   graph  $   G $  and   any  vertices  
$v_1,\ldots, v_p,  v'_1,\ldots, v'_p $  of  $   G$,    we  have  
\begin{eqnarray}\label{eq-25-apr-d7}
d_{   G^{\boxtimes p}}  ((v_1,\ldots,v_p),( v'_1,\ldots,v'_p))= \max_{1\leq  i\leq  p}\{d_{   G }(v_i,v'_i)\}.  
\end{eqnarray}
\end{enumerate}
\end{corollary}

\begin{proof}
The  proofs  of  
(1)  and  (2)  follow  from  Lemma~\ref{le-25-apr-le08a}~(1)  and  (2)  respectively.  
\end{proof}

\begin{corollary}
\label{pr-25-apr-ind-o1}
\begin{enumerate}[(1)]
\item
For  any  digraph  $\vec  G$,    the  constraint  independence  complex  
${\rm  Ind}(\vec  G^{\boxtimes  p}, n/2,  m/2)$   for  $1\leq  n<m\leq  \infty$   is  given by 
 the  simplices  of  the  form
\begin{eqnarray*}
\sigma^{(k)} = \{(v^0_1,\ldots,v^0_p),(v^1_1,\ldots,v^1_p),\ldots,(v^k_1,\ldots,v^k_p) \},~~~~~~
k\geq  0
\end{eqnarray*}  
such  that  $v_i^t$  are  vertices  of  $\vec  G$  for any  $1\leq  i\leq  p$  and  any  
$0\leq  t\leq  k$   satisfying  
\begin{eqnarray*}
\frac{n}{2}<\max_{1\leq  i\leq  p}d_{\vec  G}(v_i^j, v_i^l)\leq \frac{m}{2}, ~~~~~~0\leq  j<l\leq  k;  
\end{eqnarray*}
 
\item
For  any   graph  $  G$,    the  constraint  independence  complex  
${\rm  Ind}(   G^{\boxtimes  p}, n/2,  m/2)$   for  $1\leq  n<m\leq  \infty$   is  given by 
 the  simplices  of  the  form
\begin{eqnarray*}
\sigma^{(k)} = \{(v^0_1,\ldots,v^0_p),(v^1_1,\ldots,v^1_p),\ldots,(v^k_1,\ldots,v^k_p) \},~~~~~~
k\geq  0
\end{eqnarray*}  
such  that  $v_i^t$  are  vertices  of  $   G$  for any  $1\leq  i\leq  p$  and  any  
$0\leq  t\leq  k$   satisfying  
\begin{eqnarray*}
\frac{n}{2}<\max_{1\leq  i\leq  p}d_{ G}(v_i^j, v_i^l)\leq \frac{m}{2}, ~~~~~~0\leq  j<l\leq  k.   
\end{eqnarray*}
\end{enumerate}
\end{corollary}

\begin{proof}
The  proofs  of  
(1)  and  (2)  follow  from  Corollary~\ref{co-25-apr-prod-5}~(1)  and  (2)  respectively.  
\end{proof}

\begin{lemma}\label{pr-25-apr-tot-geod-1}
\begin{enumerate}[(1)]
\item
  Let  $\varphi_i:    \vec  G_i\longrightarrow   \vec   G'_i$  be  a  strong   totally geodesic  immersion  
   of  digraphs  with radius  $m_0/2$  (resp.  a  strong   totally geodesic  embedding   
   of   digraphs),  $i=1,2$.      
      Then   $\varphi_i$,  $i=1,2$,   induce   a    strong   totally geodesic  immersion  
   of  digraphs    with radius  $m_0/2$  (resp.  a  strong   totally geodesic  embedding   
   of   digraphs)  
   \begin{eqnarray*} 
   \varphi_1\boxtimes  \varphi_2:   
     { \vec    G_1}\boxtimes { \vec    G_2}   \longrightarrow  
      {\vec   G'_1}\boxtimes  {\vec   G'_2};   
   \end{eqnarray*}
   
\item
  Let  $\varphi_i:      G_i\longrightarrow     G'_i$  be  a  strong   totally geodesic  immersion  
   of   graphs  with radius  $m_0/2$  (resp.  a  strong   totally geodesic  embedding   
   of   graphs),  $i=1,2$.      
      Then   $\varphi_i$,  $i=1,2$,   induce   a    strong   totally geodesic  immersion  
   of   graphs    with radius  $m_0/2$  (resp.  a  strong   totally geodesic  embedding   
   of    graphs)  
   \begin{eqnarray*} 
   \varphi_1\boxtimes  \varphi_2:   
    {   G_1}\boxtimes {     G_2} \longrightarrow  
     {    G'_1}\boxtimes  {   G'_2}.     
   \end{eqnarray*}
\end{enumerate}
\end{lemma}

\begin{proof}
With  the  help  of  Definition~\ref{def-2},  
the  proofs  of  
(1)  and  (2)  follow  from  Lemma~\ref{le-25-apr-le08a}~(1)  and  (2)  respectively.  
\end{proof}

\begin{corollary}\label{co-25-apr-tot-geod-2}
\begin{enumerate}[(1)]
\item
  Let  $\varphi:    \vec  G \longrightarrow   \vec   G' $  be  a  strong   totally geodesic  immersion  
   of  digraphs  with radius  $m_0/2$  (resp.  a  strong   totally geodesic  embedding   
   of   digraphs).      
      Then  we  have  an   induced       strong   totally geodesic  immersion  
   of  digraphs    with radius  $m_0/2$  (resp.  a  strong   totally geodesic  embedding   
   of   digraphs)  
   \begin{eqnarray}\label{eq-25-apr-map1}
   \varphi^{\boxtimes  p}:   
     {\vec    G }^{\boxtimes  p }   \longrightarrow  
      ({\vec   G'  })  ^{\boxtimes   p};   
   \end{eqnarray}
   
\item
  Let  $\varphi_i:      G \longrightarrow     G' $  be  a  strong   totally geodesic  immersion  
   of   graphs  with radius  $m_0/2$  (resp.  a  strong   totally geodesic  embedding   
   of   graphs).      
      Then  we  have  an   induced         strong   totally geodesic  immersion  
   of   graphs    with radius  $m_0/2$  (resp.  a  strong   totally geodesic  embedding   
   of    graphs)  
   \begin{eqnarray} \label{eq-25-apr-map2}
   \varphi^{\boxtimes  p}:   
     {  G }^{\boxtimes  p }   \longrightarrow  
      ({   G' })^{\boxtimes   p}.     
   \end{eqnarray}
\end{enumerate}
\end{corollary}

\begin{proof}
The  proofs  of  
(1)  and  (2)  follow  from  Lemma~\ref{pr-25-apr-tot-geod-1}~(1)  and  (2)  respectively.  
\end{proof}

\subsection{Shannon  capacities of   digraphs  and  their  underlying graphs}
\label{ssa-5.2}

   For any positive  integer  $n$,  let ${  G}^{\boxtimes  n}$ be the $n$-fold  self-strong product of 
 $   G$.  
 The  {\it Shannon capacity}  $c(G)$  of $G$ is  defined to be  (cf. \cite[p. 1]{1}  and   \cite{2,7,4})
\begin{eqnarray}\label{eq-0001}
c(G)=\sup_{p\geq 1}~ \big(\alpha(G^{\boxtimes  p})\big)^{\frac{1}{p}}=\lim _{p\to \infty}  \big(\alpha(G^{\boxtimes  p})\big)^{\frac{1}{p}} 
\end{eqnarray}
where  
\begin{eqnarray*}
\alpha(G^{\boxtimes  p})= \dim  {\rm  Ind}(G^{\boxtimes   p}, \frac{1}{2}, \infty)+1 
\end{eqnarray*}
is  the maximum size of  an  independent set of vertices in $G^{\boxtimes   p}$.   
We  generalize (\ref{eq-0001})  and  define  the  {\it   double-persistent   Shannon capacity}
of  $G$   by  
\begin{eqnarray*}
c(G,-,-)=\Big\{c(G,\frac{n}{2},\frac{m}{2})~\Big|~ 1\leq  n<m \leq  \infty \Big\}
\end{eqnarray*}
  where   
  \begin{eqnarray}\label{eq-000x1}
c(G,\frac{n}{2},\frac{m}{2})=\limsup _{p\to \infty}  \big(\alpha(G^{\boxtimes  p},\frac{n}{2},\frac{m}{2})\big)^{\frac{1}{p}} 
\end{eqnarray}
and  
\begin{eqnarray*}
\alpha(G^{\boxtimes  p},\frac{n}{2},\frac{m}{2})=  \dim  {\rm  Ind}(G^{\boxtimes  p},\frac{n}{2},\frac{m}{2})+1.    
\end{eqnarray*}  
Similarly,   
we   define  the  {\it   double-persistent   Shannon capacity}
of  $\vec   G$   by  
\begin{eqnarray*}
c(\vec   G,-,-)=\Big\{c(\vec   G,\frac{n}{2},\frac{m}{2})~\Big|~ 1\leq  n<m \leq  \infty \Big\}
\end{eqnarray*}
  where   
  \begin{eqnarray}\label{eq-100x1}
c(\vec   G,\frac{n}{2},\frac{m}{2})
  =\lim \sup_{p\to \infty}  \big(\alpha({\vec  G}^{\boxtimes  p},\frac{n}{2},\frac{m}{2})\big)^{\frac{1}{p}} 
\end{eqnarray}
and  
\begin{eqnarray*}
\alpha({\vec  G}^{\boxtimes  p},\frac{n}{2},\frac{m}{2})=  \dim  {\rm  Ind}({\vec  G}^{\boxtimes  p},\frac{n}{2},\frac{m}{2})+1.    
\end{eqnarray*}

\begin{proposition}\label{pr-25-apr-12-sha1}
For any  digraph  $\vec  G$  with  its  underlying  graph  $G$, 
we  have    persistent  embeddings  of   filtered  simplicial  complexes  
\begin{eqnarray}\label{eq-apr-12-m2}
&{\rm  Ind}({  G}^{\boxtimes  p}, -,\infty) \longrightarrow  {\rm  Ind}(\pi({\vec  G}^{\boxtimes  p}), -,\infty)
\longrightarrow       {\rm  Ind}( {\vec  G}^{\boxtimes  p}, -,\infty),\\
&{\rm  Ind}( { \vec G}^{\boxtimes  p},  \frac{1}{2},-)  \longrightarrow 
{\rm  Ind}(\pi({\vec  G}^{\boxtimes  p}),\frac{1}{2},-)    
\longrightarrow {\rm  Ind}({  G}^{\boxtimes  p}, \frac{1}{2},-).   
\label{eq-apr-12-m23456}
\end{eqnarray}
\end{proposition}

\begin{proof}
Both  $ \pi(\vec  G^{\boxtimes  p}) $  and  $   G^{\boxtimes  p}$  have  vertex  sets   
$V^p$.  
By  (\ref{eq-apr-v1}), 
$ G^{\boxtimes  p}$  is  obtained  from  $\pi(\vec  G^{\boxtimes  p})$  
by  adding  more  edges,   which  implies  
\begin{eqnarray}\label{eq-25-09-87}
d_{ G^{\boxtimes  p}}\leq  d_{\pi(\vec  G^{\boxtimes  p})}. 
\end{eqnarray}
It  follows  from  (\ref{eq-25-09-87})  that  
\begin{eqnarray*}
{\rm  Ind}({  G}^{\boxtimes  p}, \frac{n}{2},\infty) &\subseteq&  
{\rm  Ind}(\pi({\vec  G}^{\boxtimes  p}),\frac{n}{2},\infty),    \\
{\rm  Ind}({  G}^{\boxtimes  p}, \frac{1}{2},\frac{m}{2}) &\supseteq&  
{\rm  Ind}(\pi({\vec  G}^{\boxtimes  p}),\frac{1}{2},\frac{m}{2})    
\end{eqnarray*}
for   any  $1\leq  n<m\leq  \infty$.  
Consequently,  we  have  persistent  simplicial  embeddings  
\begin{eqnarray*}
i'_{\vec  G^{\boxtimes  p}}(-,\infty): &&  {\rm  Ind}({  G}^{\boxtimes  p}, -,\infty) \longrightarrow
{\rm  Ind}(\pi({\vec  G}^{\boxtimes  p}),-,\infty),    \\
j'_{\vec  G^{\boxtimes  p}}(\frac{1}{2},-):  && {\rm  Ind}(\pi({\vec  G}^{\boxtimes  p}),\frac{1}{2},-)    
\longrightarrow {\rm  Ind}({  G}^{\boxtimes  p}, \frac{1}{2},-).     
\end{eqnarray*}
On  the  other  hand,   
by 
Proposition~\ref{pr-3.v1},  we  have    persistent  simplicial  embeddings  
 \begin{eqnarray*} 
i_{\vec  G^{\boxtimes  p}}(-,\infty):   &&  {\rm  Ind}(\pi({\vec  G}^{\boxtimes  p}),  -,\infty)  \longrightarrow 
 {\rm  Ind}( {\vec  G}^{\boxtimes  p}, -,\infty),\\
j_{\vec  G^{\boxtimes  p}}(\frac{1}{2},-):  &&  
 {\rm  Ind}( { \vec G}^{\boxtimes  p},  \frac{1}{2},-)  \longrightarrow 
 {\rm  Ind}( \pi({\vec  G}^{\boxtimes  p}), \frac{1}{2},-). 
   \end{eqnarray*}
Therefore,  
the  composition  of  $i'_{\vec  G^{\boxtimes  p}}(-,\infty)$  and  
$i_{\vec  G^{\boxtimes  p}}(-,\infty)$  implies  (\ref{eq-apr-12-m2});  
and  the  composition  of  $j_{\vec  G^{\boxtimes  p}}({1}/{2},-)$  and  
$j'_{\vec  G^{\boxtimes  p}}({1}/{2},-)$  implies  (\ref{eq-apr-12-m23456}). 
\end{proof}

\begin{proposition}
\label{pr-25-apr-12-sha5}
For any  digraph  $\vec  G$  with  its  underlying  graph  $G$  and  any  positive integer  $q$, 
\begin{eqnarray}\label{25-apr-o3}
&
c(G,-,\infty)\leq    c(\pi({\vec  G}^{\boxtimes  q}),-,\infty)^{\frac{1}{q}}  \leq   c(\vec  G,-,\infty),\\
\label{25-apr-o33333}
&
c(\vec  G,\frac{1}{2},-)\leq    c(\pi({\vec  G}^{\boxtimes  q}),\frac{1}{2},-)^{\frac{1}{q}}  
\leq   c(   G,\frac{1}{2},-).  
\end{eqnarray}
\end{proposition}

\begin{proof}
For  any  positive  integer  $p$,  
it  follows  from  (\ref{eq-apr-12-m2})  that  
we  have     double-persistent  embeddings  of   double-filtered   simplicial  complexes  
\begin{eqnarray}\label{eq-apr-12-m8}
{\rm  Ind}({  G}^{\boxtimes  pq}, -,\infty) \longrightarrow  {\rm  Ind}(\pi({\vec  G}^{\boxtimes  pq}), -,\infty)
\longrightarrow  {\rm  Ind}(\pi({\vec  G}^{\boxtimes  q})^{\boxtimes  p}, -,\infty)
\longrightarrow       {\rm  Ind}( {\vec  G}^{\boxtimes  pq}, -,\infty).  
\end{eqnarray}
Taking  the  dimensions  of  the  simplicial  complexes  in   (\ref{eq-apr-12-m8}),  
we  have  
\begin{eqnarray*}
\alpha({  G}^{\boxtimes  pq}, -,\infty) \leq   \alpha(\pi({\vec  G}^{\boxtimes  pq}), -,\infty)
\leq   \alpha(\pi({\vec  G}^{\boxtimes   q})^{\boxtimes   p}, -,\infty)
\leq        \alpha( {\vec  G}^{\boxtimes  pq}, -,\infty).    
\end{eqnarray*}
This   implies 
\begin{eqnarray}\label{eq-25-apr-c9}
\limsup_{p\to\infty} \alpha({  G}^{\boxtimes  pq}, -,\infty)^{\frac{1}{pq}}
\leq   \limsup _{p\to\infty}   \alpha(\pi({\vec  G}^{\boxtimes   q})^{\boxtimes   p}, -,\infty)^{\frac{1}{pq}}
\leq  \limsup _{p\to\infty} \alpha({\vec  G}^{\boxtimes  pq}, -,\infty)^{\frac{1}{pq}}. 
\end{eqnarray}
Substituting  
\begin{eqnarray*}
c(G,-,\infty)  &=&  \limsup_{p\to\infty} \alpha({  G}^{\boxtimes  pq}, -,\infty)^{\frac{1}{pq}},\\
 c(\pi({\vec  G}^{\boxtimes  q}),-,\infty)^{\frac{1}{q}}  &=&
   \limsup _{p\to\infty}   \alpha(\pi({\vec  G}^{\boxtimes   q})^{\boxtimes   p}, -,\infty)^{\frac{1}{pq}},\\
    c(\vec  G,-,\infty) &=& \limsup _{p\to\infty} \alpha({\vec  G}^{\boxtimes  pq}, -,\infty)^{\frac{1}{pq}}  
\end{eqnarray*}
in  (\ref{eq-25-apr-c9}),
we  obtain  (\ref{25-apr-o3}).  
By  a  similar  argument,  we  obtain   (\ref{25-apr-o33333}) from  (\ref{eq-apr-12-m23456}).  
\end{proof}

 \begin{proposition}\label{pr-25-apr-shan-98}
 \begin{enumerate}[(1)]
\item
  Let  $\varphi:    \vec  G \longrightarrow   \vec   G' $  be  a  strong   totally geodesic  immersion  
   of  digraphs  with radius  $m_0/2$  (resp.  a  strong   totally geodesic  embedding   
   of   digraphs).      
      Then   
   \begin{eqnarray}\label{eq-8-l15}   
     c ({\vec    G },-,-)    \leq   
      c({\vec   G'  },-,-)      
   \end{eqnarray}
  for  $1\leq  n<m\leq  m_0$    (resp.   for  $1\leq  n<m\leq  \infty $)
    where  $n/2$    and  $m/2$  are the     parameters  
   in  the  double-persistence;
\item
  Let  $\varphi_i:      G \longrightarrow     G' $  be  a  strong   totally geodesic  immersion  
   of   graphs  with radius  $m_0/2$  (resp.  a  strong   totally geodesic  embedding   
   of   graphs).      
      Then           
   \begin{eqnarray} \label{eq-8-l25}   
     c ({     G },-,-)    \leq   
      c({   G'  },-,-)      
   \end{eqnarray}
  for  $1\leq  n<m\leq  m_0$    (resp.   for  $1\leq  n<m\leq  \infty $)
    where  $n/2$    and  $m/2$  are the     parameters  
   in  the  double-persistence.
   \end{enumerate}
 \end{proposition}
 
 \begin{proof}
 (1)  
 By  Corollary~\ref{co-25-apr-tot-geod-2}~(1),  
 for  any  positive  integer  $p$,  we  have  an  induced  strong   totally geodesic  immersion  
   of  digraphs  with radius  $m_0/2$  (resp.  a  strong   totally geodesic  embedding   
   of   digraphs)  (\ref{eq-25-apr-map1}).  
   By  Corollary~\ref{co-au-1},  
   we  have  an  induced        double-persistent  embedding  
   of  double-filtered  simplicial  complexes 
      \begin{eqnarray}\label{eq-8-l1}
   \varphi(-,-):    {\rm  Ind}({\vec  G}^{\boxtimes  p}, -,-)  \longrightarrow  
    {\rm   Ind} (({\vec  G'})^{\boxtimes  p}, -,-)  
   \end{eqnarray}
     for  $1\leq  n<m\leq  m_0$    (resp.   for  $1\leq  n<m\leq\infty $)
     where  $n/2$    and  $m/2$  are the     parameters  
   in  the  double-persistence.
   It  follows  from  (\ref{eq-8-l1})   that 
   \begin{eqnarray}\label{eq-8-l12}
   \alpha({\vec  G}^{\boxtimes  p},-,-)^{\frac{1}{p}}
   \leq   
    \alpha(({\vec  G '})^{\boxtimes  p},-,-)^{\frac{1}{p}}.  
   \end{eqnarray}
   Let $p\to\infty$  in  (\ref{eq-8-l12}).   We  obtain  (\ref{eq-8-l15}).

    (2)  
 By  Corollary~\ref{co-25-apr-tot-geod-2}~(2),  
 for  any  positive  integer  $p$,  we  have  an  induced  strong   totally geodesic  immersion  
   of   graphs  with radius  $m_0/2$  (resp.  a  strong   totally geodesic  embedding   
   of    graphs)  (\ref{eq-25-apr-map2}).  
   By  Corollary~\ref{co-au-2},  
   we  have  an  induced        double-persistent  embedding  
   of  double-filtered  simplicial  complexes 
      \begin{eqnarray}\label{eq-8-l2}
   \varphi(-,-):    {\rm  Ind}({   G}^{\boxtimes  p}, -,-)  \longrightarrow  
    {\rm   Ind} (({   G'})^{\boxtimes  p}, -,-)  
   \end{eqnarray}
     for  $1\leq  n<m\leq  m_0$    (resp.   for  $1\leq  n<m \leq\infty$).
   It  follows  from  (\ref{eq-8-l2})   that 
   \begin{eqnarray*} 
   \alpha({   G}^{\boxtimes  p},-,-)^{\frac{1}{p}}
   \leq   
    \alpha(({   G '})^{\boxtimes  p},-,-)^{\frac{1}{p}},   
   \end{eqnarray*}
whose  limit  $p\to\infty$  implies  (\ref{eq-8-l25}).
 \end{proof}

\bigskip

Shiquan Ren     

 Affiliation:  School of Mathematics and Statistics,  Henan University

 Address:    Kaifeng  475004,  China 

E-mail:  renshiquan@henu.edu.cn

\end{document}